\documentclass[12pt, twoside]{article}
\usepackage{amssymb,amsmath,amsfonts,amsthm,mathrsfs}
\usepackage{indentfirst}
\usepackage{graphics}
\usepackage{color}
\usepackage{geometry}
\usepackage{enumerate}
\usepackage{amsmath}
\usepackage{bbm}           
\usepackage{bm}
\usepackage{eso-pic,graphicx}
\usepackage{tikz}
\usepackage{cite}
\usepackage{esint}
\usepackage{hyperref}
\usepackage{pslatex}
\usepackage{rotating}
\usepackage{verbatim}

\makeatletter
\def\Ddots{\mathinner{\mkern1mu\raise\p@
		\vbox{\kern7\p@\hbox{.}}\mkern2mu
		\raise4\p@\hbox{.}\mkern2mu\raise7\p@\hbox{.}\mkern1mu}}
\makeatother

\def\XXint#1#2#3{{\setbox0=\hbox{$#1{#2#3}{\int}$}
		\vcenter{\hbox{$#2#3$}}\kern-.5\wd0}}

\allowdisplaybreaks
	
\newtheorem{cor}{Corollary}[section]
\newtheorem{thm}{Theorem}[section]
\newtheorem{lem}{Lemma}[section]
\newtheorem{defn}{Definition}[section]
\newtheorem{rem}{Remark}[section]
\newtheorem{pro}{Proposition}[section]

\makeatletter
\@addtoreset{equation}{section}
\makeatother

\def\rn{{\mathbb R^n}}

\def\d{{\mathrm d}}
\def\D{{\mathcal D}}

\begin{document}
\title
{\bf\Large 
Weighted Bourgain-Morrey-Besov type and Triebel-Lizorkin type spaces associated with operators
 	\footnotetext{Jingshi Xu is supported by the National Natural Science Foundation of China (Grant No. 12161022) and the Science and Technology Project of Guangxi (Guike AD23023002).
		Pengfei Guo is supported by Hainan Provincial Natural Science Foundation of China (Grant No. 122RC652).
	}
}
	
\date{}

\author{Tengfei Bai\textsuperscript{a}, Pengfei Guo\textsuperscript{a},  Jingshi Xu\textsuperscript{b,c,d}\footnote{Corresponding author, E-mail: jingshixu@126.com}  \\
	{\scriptsize \textsuperscript{a}  College of Mathematics and Statistics, Hainan Normal University, Haikou, Hainan 571158,
		China}\\
	{\scriptsize  \textsuperscript{b} School of Mathematics and Computing Science, Guilin University of Electronic Technology, Guilin 541004, China} \\
	{\scriptsize  \textsuperscript{c} Center for Applied Mathematics of Guangxi (GUET), Guilin 541004, China}\\
	{\scriptsize  \textsuperscript{d} Guangxi Colleges and Universities Key Laboratory of Data Analysis and Computation, Guilin 541004, China}
}

\pagestyle{myheadings}\markboth{\footnotesize\rm\sc Tengfei Bai,  Pengfei Guo, Jingshi Xu}
{\footnotesize\rm\sc Weighted Bourgain-Morrey type Besov and Triebel-Lizorkin type spaces}
	
\maketitle

\begin{abstract}
	Let $(X,\mu)$ be a space of homogeneous type satisfying $\mu(X) =\infty$, the doubling property and the reverse doubling condition.	
	 Let $L$ be a nonnegative self-adjoint operator on $L^2(X)$	whose heat kernel enjoys  a Gaussian upper bound.  
	We introduce the 	weighted  homogeneous  Bourgain-Morrey-Besov type spaces and Triebel-Lizorkin type spaces  associated with 	the operator $L$. 
We obtain their continuous   characterizations	in terms of  Peetre maximal functions,   noncompactly supported functional calculus, heat kernel.  Atomic  and molecular decompositions of weighted  homogeneous  Bourgain-Morrey-Besov type spaces and Triebel-Lizorkin type spaces are also given. As an application, we obtain the boundedness of the fractional power of $L$, the spectral multiplier of $L$ on  Bourgain-Morrey-Besov type spaces and Triebel-Lizorkin type spaces. 
\end{abstract}


\section{Introduction}\label{sec1}
Let  $X$ be a spaces of homogeneous type, with quasidistance $\rho$ and  nonnegative Borel measure $\mu$ on $X$, which satisfies the doubling property below. In this paper, we assume that $\mu (X) = \infty$.

For $x\in X$ and $r>0$, denote by $B(x,r) = \{ y\in X : \rho(x,y) < r\}$  the open ball with radius $r>0$ and center $x \in X$. Set $V(x,r) = \mu ( B (x,r )) $.  We say that $\mu $ has the {\it doubling property} if there exists a constant $C>0$ such that for all $x \in X$ and $r>0$, 
\begin{equation} \label{double}
	V(x,2r) \le C V(x,r).
\end{equation}
The doubling property (\ref{double}) yields a constant $n>0 $  playing the role of a dimension such that 
\begin{equation*} 
	V(x, \lambda r) \le C \lambda ^n  V(x,  r), 
\end{equation*}
for all $\lambda \ge 1$, $x\in X$ and $r>0$, and that 
\begin{equation} \label{V x r V y r}
	V(x,r) \le C \Big( 1+ \frac{\rho(x,y)}{r} \Big)^{\tilde n}  V(y,r),
\end{equation}
for all $x,y \in X$, $r>0$ and for some $\tilde{n} \in [0,n]$; see for example \cite{CW71}.
To fit function spaces of this paper, we  also assume that $X$  satisfies the {\it reverse doubling condition}; see Remark \ref{dya cube}. 

Let $L$ be a nonnegative self-adjoint operator on $L^2 (X)$ which generates a semigroup $\{ e^{-\alpha L} \} _{ \alpha>0 }$. Let $p_\alpha (x,y)$ be the kernel of the semigroup $ e^{-\alpha L} $. In this paper, we suppose that the kernel  $p_\alpha (x,y)$  satisfies a Gaussian upper bound, that is, there exist constants $C, c >0$ such that for all $x,y \in X$ and $\alpha >0$, 
\begin{equation*}
 p_\alpha (x,y) \le \frac{C}{\mu( B(x, \sqrt{\alpha} )  )}  \exp \bigg(   -\frac{\rho(x,y) ^2 }{c \alpha}   \bigg) . 
\end{equation*}

The classical Besov and Triebel-Lizorkin spaces including many famous function spaces have played an essential role in approximation theory and partial differential equations, see \cite{Bes59, Bes61, BIN79, BPT95, BPT96, BPT97, BT00, FJ85, FJ90, Pee75, Pee76, Tri83, Tri92, YY17}. The theory of  Besov and Triebel-Lizorkin spaces associated with differential operators $L$ has been developed by many mathematicians in various settings; see \cite{BBD20, BDY12, BD15, BD17, CKP12, DP05, KP15, KPPX09, LYY16, PX08, YY14, GHS23,SYY24}. For example,  Liu, Yang and Yuan \cite{LYY16} introduced the inhomogeneous Besov-type and Triebel-Lizorkin spaces associated to a nonnegative self adjoint operator which satisfies sub-Gaussian upper bound estimate, H\"older continuity, and stochastic completeness. They obtained the characterizations of these spaces in terms of Peetre maximal functions, heat kernel, frame. Embedding properties were also given. 
In \cite{GHS23}, Gon{\c{c}}alves, Haroske and Skrzypczak considered the necessary and sufficient conditions for the continuity of limiting embeddings of Besov-type and Triebel-Lizorkin-type spaces on bounded Lipschitz domains.
In \cite{SYY24}, Sun, Yang and Yuan introduced ball quasi-Banach function sequence spaces and obtained characterizations of these spaces via  $\varphi$-transform, atoms, molecules, maximal functions, Littlewood-Paley functions and wavelets. They also showed the boundedness of almost diagonal operators, pseudodifferential operators and generalized Calderon-Zygmund operators on these spaces.

Under the assumptions that $X$  is a space of homogeneous type and that $L$ is a nonnegative self-adjoint operator satisfying the Gaussian upper bound, in \cite{BBD20}, Bui, Bui, and Duong established  the theory of weighted Besov spaces $\dot{B}_{p,q,\omega}^{\alpha,L} (X) $ and  weighted Triebel-Lizorkin spaces $\dot{F}_{p,q,\omega}^{\alpha,L}(X) $ associated with the operator $L$ for the full range $0<p,q\le \infty$, $\alpha \in \mathbb R$ and $\omega$ being in the Muckenhoupt weight class $A_\infty$. They obtained  characterizations via Peetre maximal functions,   noncompactly supported functional calculus, heat kernel, Lusin functions, Littlewood-Paley functions, atom and molecule.

Since a special case of Bourgain-Morrey spaces was introduced by Bourgain in \cite{B91},  Bourgain-Morrey spaces have been applied to
various partial differential equations, especially the Strichartz estimate and nonlinear Schr\"odinger equations, see, e.g., \cite{BPV07,B91,C64, M16,MV98,MVV99}.

Recently, Hatano, Nogayama, Sawano and Hakim investigated
the Bourgain--Morrey function spaces $M_{p}^{t,r}$ 
from the viewpoints of harmonic analysis and functional analysis in \cite{HNS23}.
After then, some function spaces extending Bourgain-Morrey spaces were established.
For example, in \cite{ZSTYY23}, Zhao et al.  introduced Besov-Bourgain-Morrey spaces which connect Bourgain-Morrey spaces with amalgam-type spaces.  They obtained predual, dual spaces and complex interpolation  of these spaces. They also gave an equivalent norm with an integral expression  and  obtained the boundedness on these spaces
of the Hardy-Littlewood maximal operator, the fractional integral, and the Calder\'on-Zygmund operator. 

Immediately after \cite{ZSTYY23},
Hu, Li, and Yang  introduced  Triebel-Lizorkin-Bourgain-Morrey spaces which connect Bourgain-Morrey spaces and global Morrey spaces in \cite{HLY23}. 
They considered the embedding relations between Triebel-Lizorkin-Bourgain-Morrey spaces and Besov-Bourgain-Morrey spaces.  
They studied various fundamental real-variable properties of these spaces. They obtained the sharp boundedness of  the Hardy-Littlewood maximal operator, the Calder\'on-Zygmund operator, and the fractional integral on these spaces.

Inspired by the generalized grand Morrey spaces and Besov-Bourgain-Morrey spaces, Zhang et al. introduced generalized grand Besov-Bourgain-Morrey spaces in \cite{ZYZ24}. They obtained predual spaces and the Gagliardo-Peetre  interpolation theorem,  extrapolation theorem.  The boundedness of Hardy-Littlewood maximal operator, the fractional integral and the Calder\'on-Zygmund operator on generalized grand Besov-Bourgain-Morrey spaces is also proved. 

Very recently, the first author and the third author of the paper \cite{BX25} introduced the weighted homogeneous Bourgain-Morrey Besov spaces and Triebel-Lizorkin associated with the operator. They obtained the Hardy-Littlewood maximal function and Fefferman-Stein maximal inequality on weighted Bourgain-Morrey sapces. Some characterizations of homogeneous Bourgain-Morrey Besov and Triebel-Lizorkin spaces were also obtained, such as  Peetre maximal function, compact support functions,  atomic  and molecular decompositions.

Motivated by above literature, we will establish the theory of weighted homogeneous Bourgain-Morrey-Besov type and Triebel-Lizorkin type spaces. 
The paper is organized as follows.
	In  Section \ref{Pre},   
	 we recall notations such as Muckenhoupt weights, Hardy-Littlewood maximal functions, the class of distributions and   dyadic cubes; and some lemmas.
	 In Section \ref{Bou spa},
	the definition of weighted homogeneous Bourgain-Morrey Besov type and Triebel-Lizorkin type  spaces is given first.  We  get the Hardy type inequality  on sequence Bourgain-Morrey spaces. These are  useful in the proof of this paper.
	Charaterizations of weighted homogeneous Bourgain-Morrey Besov type and Triebel-Lizorkin type  spaces via Peetre maximal function, compact support functions, noncompact support functions  are obtained in this section. 
	We get the 	charaterizations of weighted homogeneous Bourgain-Morrey Triebel-Lizorkin type  spaces via Lusin functions and Littlewood-Paley functions in the last of Section \ref{Bou spa}.
	 In Section \ref{atm dec}, we obtain  atomic  and molecular decompositions of weighted homogeneous Bourgain-Morrey-Besov type  and Triebel-Lizorkin type  spaces.
	 In Section \ref{app}, using the results in Section \ref{Bou spa},  we obtain that fractional powers and spectral multipliers of Laplace transform type of $L$  are bounded  on weighted homogeneous Bourgain-Morrey Besov type and Triebel-Lizorkin type spaces.

Throughout this paper, we let $c, C$ denote constants that are independent of the main parameters involved but whose value may differ from line to line. 
Let $\mathbb N = \{0,1,2,3,\ldots \}$ and $\mathbb N_+ = \{1,2,3,\ldots \}$.
Let $\mathbb Z$ be the set of all integers.
Let $\chi_{E}$ be the characteristic function of the set $E\subset X$.
Given a ball $B:=B(x_B, r_B)$ and $\lambda>0$, $\lambda B$ denotes the ball with the same center as $B$ whose radius is $\lambda$ times that of $B$. 
 By $A\lesssim B$ we mean that $A\leq CB$ with some positive constant $C$ independent of appropriate quantities. By $ A \approx B$, we mean that $A\lesssim B$ and $B\lesssim A$.
Let $\log a := \log_2 a$ for $a>0$.
Let $ \mathscr S (\mathbb R)$ be the calss of Schwartz functions on $\mathbb R$.

\section{Preliminaries}\label{Pre}

\subsection{Dyadic cubes}
The following covering lemma comes from \cite{C90}.
\begin{lem}\label{cube}
	Let $\mu (X)  =\infty$.
	There exists a collection of open sets $\{  Q_\tau ^k \subset X : k \in \mathbb Z, \tau \in I_k  \}$, where $I_k$ denotes certain  index set depending on $k$ and constants $\gamma \in (0,1)$, $a_0 \in (0,1]$ and $ \kappa_0 \in (0,\infty)$ such that
	\begin{itemize}
		\item[\rm (i)] $\mu (  X \backslash \cup_{ \tau \in I_k  } Q_\tau^k) =0 $ for all $k \in \mathbb Z$;
		\item[\rm (ii)] if $i\ge k$, then either $Q_\tau^i \subset Q_\beta^k$ or $Q_\tau^i \cap Q_\beta^k  = \emptyset $;
		\item[\rm (iii)] for every $(k,\tau)$ and each $i < k$, there exists a unique $\tau '$ such that $Q_\tau^k \subset Q_{\tau '}^i$; we say that $ Q_{\tau '}^i  $ is a parent of cube $Q_\tau^k$.
		\item[\rm (iv)] the diameter of  $Q_\tau ^k \le \kappa _0 \gamma ^k$;
		\item[\rm (v)] each $Q_\tau ^k $ contains certain ball $B(x_{Q_\tau ^k}, a_0 \gamma ^k )$.
	\end{itemize}
\end{lem}
\begin{rem}\label{dya cube}
	Since the constants $\gamma$ and $a_0$ are not essential in the paper, without loss of generality, we may assume that $\gamma = a_0 = 1/2$. Then fix a collection of open sets in Lemma \ref{cube} and denote this collection by $\mathcal D$. We call open sets in $\mathcal D$ the dyadic cubes in $X$ and $  x_{Q_\tau ^k}$ the center of the cube $Q_\tau ^k \in \mathcal{D}$. Denote $\mathcal D _\nu : = \{  Q_\tau ^{\nu +1} \in \mathcal D : \tau \in I_{\nu+1}  \}$ for each $\nu \in \mathbb Z$. Then for each cube $Q \in \mathcal D_\nu$, we have $ B(x_Q, c_0 2^{-\nu} )  \subset Q \subset   B(x_Q, \kappa_0 2^{-\nu} ) =: B_Q $, where $c_0 $ is a constant independent of $Q$.

Note that the diameter of cubes in $\mathcal D_k$ approximates $2^{-k}$. Hence the parent cube $Q^{k - \nu}_{\tau'}$ has  at most $c2^{\nu n}$ cubes in $\mathcal D _k$, where $c$ is a not important positive number.

For each cube $Q \in \mathcal D$, let $2^1 Q$ be its parent cube. Since $X$ satisfies the doubling condition and the relation $B(x_Q, c_0 2^{-\nu} )  \subset Q \subset   B(x_Q, \kappa_0 2^{-\nu} )$, we have
\begin{equation*}
	\mu(2^1 Q) \le C 2^n \mu (Q),
\end{equation*}
where $n$ is the role of a dimension of $X$.
\end{rem}
In the sequel, we further suppose that $X$ satisfies    the {\it  reverse doubling condition}. That is,  for a dyadic cube $Q $ and its parent $R$, there exists a constant $\alpha_0 \in (0,1) $ such that 
\begin{equation*}
	\mu (Q) \le \alpha_0 \mu (R).
\end{equation*} 
And  $\alpha_0$ is called the reverse doubling  constant of $X$.
Hence in this article,  $X$  is a space of homogeneous type, with quasidistance $\rho$ and  nonnegative Borel measure $\mu$ on $X$, which satisfies the doubling property,  the reverse doubling condition, and $\mu (X) = \infty$.

\subsection{Muckenhoupt  weights}
A weight function $\omega$ is a locally integrable function on $X$ that takes values in $(0,\infty)$ almost everywhere. For a given weight function $\omega$ and a measurable set $E \subset X$, we denote  the weighted measure of $E$ by $\omega(E)$, where $\omega(E)=\int_E \omega(x)\,\mathrm {d} \mu (x)$ and $V(E) = \mu(E)$.
For $1\le p \le \infty $, $p' $ is the conjugate exponent of $p$, that is $1/p + 1/p' =1$.

\begin{defn}
	Let $1<p<\infty $. We say that a weight $\omega$ belongs to class $A_p$ if
	\begin{equation*}
		[\omega]_{A_p} := \sup_{B\;  \mathrm{ balls \; in  } \;X} \Big(   \frac{1}{V(B)} \int_B \omega(x) \mathrm {d} \mu(x)    \Big) ^{1/p} \Big(   \frac{1}{V(B)} \int_B \omega(x) ^ { -1/(p-1) } \mathrm {d} \mu (x)   \Big) ^{ (p-1 ) /p} <\infty.
	\end{equation*}
 We call a weight $\omega$   an $A_1$ weight if
	\begin{equation*}
		\mathcal M (w) (x) \le C w(x)
	\end{equation*}
	for almost all $x \in X$ and some $C>0$, where and what follows $\mathcal M$ is the usual Hardy-Littlewood maximal function.
	Set $A_\infty = \cup_{p\ge 1} A_p$.
\end{defn}

For  $\omega\in A_\infty $ and $ 0<p<\infty  $, define the weighted Lebesgue space $L^p(\omega)$ by
\begin{equation*}
	\bigg\{   f : \| f\| _{  L^p(\omega)}  = \bigg(  \int_X |f(x)|^p \omega (x)\mathrm {d} \mu(x)     \bigg) ^{1/p} <\infty      \bigg\}.
\end{equation*}

In the following lemma, 
we list some properties  of $A_p$ weights, whose proofs are  similar to those in \cite[Chapter 7]{Gra1}. 
More properties of $A_p$ weights can be seen in \cite[Chapter 7]{Gra1} and  \cite{ST89}. 

\begin{lem} \label{weights} 
	Let $\omega \in A_p$ for some $1\le p <\infty$. Then
	\begin{itemize}
		\item[\rm (i)] The classes $A_p$ are increasing as $p$ increases.
		\item[\rm (ii)] The following is an equivalent characterization of the $A_p$ characteristic constant of $\omega $:
		\begin{equation*}
			[\omega]_{A_p} = \sup_{Q} \sup_{f \in L^p(Q,\omega), \int_Q |f|^p \omega \mathrm{d} \mu (t) >0}  \bigg\{    \frac{ (\frac{1}{|Q|} \int_Q |f(t)| \mathrm {d} \mu (t) )^p  }{   \frac{1}{\omega(Q)} \int_Q |f(t)|^p \omega (t) \mathrm {d} \mu (t)  }    \bigg\}.
		\end{equation*}
		\item[\rm (iii)] The measure $\omega (x) \mathrm {d} \mu (x) $ is doubling: for all cubes $Q$ and all $\mu$-measurable subsets $A$ of $Q$ we have
		\begin{equation*}
			\Big(    \frac{\mu(A)}{ \mu(Q) } \Big) ^p \le [ \omega]_{A_p}  \frac{\omega (A)  }{  \omega(Q) }   .
		\end{equation*}
		\item[\rm (iv)] Let $0<\alpha <1$. Then there exists $\beta : = 1-   {(1-\alpha)^p} / {[\omega]_{A_p} }  <1$ such that  whenever $S$ is a measurable subset of a cube $Q$ with $  \mu (S) \le \alpha \mu (Q)$, we have $ \omega (S) \le \beta \omega (Q).$
		\item[\rm (v)] If $ \omega \in  A_p, 1 < p <\infty,$ then there exists $1<r < p<\infty$ such that $\omega \in A _r$.
	\end{itemize}
\end{lem}

For $\omega \in A_\infty$,  we define $q_\omega  = \inf \{ q: \omega \in A_q   \}$. The self-improvement	property of the Muckenhoupt weights shows that $\omega \in A_{q_\omega}$ can never happen unless $\omega \in A_1$.
Let $\omega \in A_\infty $ and $0< A <\infty $. The weighted Hardy-Littlewood maximal function $\mathcal M _{ A , \omega}$  is defined by
\begin{equation*}
	\mathcal M _{ A , \omega} f(x) = \sup_{x\in B}  \Big(    \frac{1}{\omega(B)} \int_B  |f(y)|^A  \omega (y) \mathrm {d} \mu (y)       \Big) ^{1/A},
\end{equation*}
where the supremum is taken over all balls $B$ containing $x$. We will drop the subscripts $A$ or $\omega $ when either $A=1$  or $\omega =1.$

Let $\omega \in A_\infty$, and $0<A< p \le  \infty$. It is well known that (for example, see \cite[(6)]{BBD20})
\begin{equation}\label{M a omega Lp}
	\|  	\mathcal M _{A , \omega} f \|_{L^p (\omega) }  \lesssim \| f \| _   {  L^p (\omega) }.
\end{equation} 
Moreover, let $0<A< p \le  \infty$,  and $\omega \in A_{p/A}$. Then we have
\begin{equation*}
	\|  	\mathcal M _{A } f \|_{L^p (\omega) }  \lesssim \| f \| _   {  L^p (\omega) }.
\end{equation*}
The following is the Fefferman-Stein vector maximal inequality and its variant in \cite{GLY09, KP15}.
For $0<p < \infty$, $ 0<q \le \infty $, $0<A <\min \{ p,q \} $ and $\omega \in A_{p/A}$, then for any sequence of measurable functions $ \{  f_\nu \}$,
\begin{equation} \label{FS}
	\bigg\| \bigg(  \sum_\nu | \mathcal M _A f_\nu |^q \bigg)^{1/q} \bigg\|_{L^p (\omega)} \lesssim  \bigg\| \bigg(  \sum_\nu |  f_\nu |^q \bigg)^{1/q} \bigg\|_{L^p (\omega)}.
\end{equation}
The Young inequality and (\ref{FS}) imply the following inequality: if $ \{a_\nu \} \in \ell^q \cap  \ell^1 $, then
\begin{equation*}
	\bigg\|  \sum_j  \bigg(  \sum_\nu |a_{j-\nu} \mathcal M _A f_\nu |^q \bigg)^{1/q} \bigg\|_{L^p (\omega)} \lesssim  \bigg\| \bigg(  \sum_\nu |  f_\nu |^q \bigg)^{1/q} \bigg\|_{L^p (\omega)},
\end{equation*}
where the implicit constant is independent of the sequence of functions $\{f_\nu\}$.

\subsection{The class of distributions}
Fix $x_0 \in X$ as a reference point  in $X$. The class of test functions $\mathcal S$ associated with $L$ is defined as the set of all functions $\phi \in \cap _{m\ge 1} D(L^m) $ such that
\begin{equation*}
	\mathcal P _{m, \ell } (\phi) = \sup_{x\in X}  (1+ \rho(x, x_0) )^m  | L^\ell \phi (x) | <\infty, \; \forall m >0, \ell \in \mathbb{N}.
\end{equation*}
From \cite{KP15}, we know that $\mathcal S$ is a complete locally convex space with topology generated by the family of seminorms  $ \{  \mathcal P _m^\ell : m >0, \ell \in \mathbb N  \}. $
The space of distributions $\mathcal S ' $   is   the set of all continuous linear functional on $\mathcal S$ with the dual defined by
\begin{equation*}
	\langle f , \phi \rangle = f ( \bar \phi)
\end{equation*}
for all $f\in \mathcal S' $ and $\phi \in \mathcal S$.

Following \cite{GKKP17} , we define the space $\mathcal S _\infty$ as the set of all functions $\phi \in \mathcal S$ such that for each $k \in \mathbb N$, there exists $g_k \in \mathcal S$ such that $\phi = L^k g_k $. 
That is, $L^{-k}\phi \in \mathcal S $  for all $k\ge 1$.
 Note that such an $g_k$, if exists, is unique; see \cite{GKKP17}. The topology in $\mathcal S _\infty $  is generated by the following family of seminorms:
\begin{equation*}
	\mathcal P _{m,\ell, k} ^{*} (\phi) = \mathcal P_{m,\ell} (g_k), \; m >0; \ell , k \in  \mathbb N
\end{equation*}
where $\phi = L^k g_k$. Denote by $\mathcal S _\infty ' $ the set of all continuous linear functionals on $\mathcal  S_\infty $.
We also define
\begin{equation*}
	\mathcal P_m = \{  g \in \mathcal S ' : L^m g = 0 \}, \; m\in \mathbb N
\end{equation*}
and set $\mathcal P = \cup _{m\in \mathbb N} \mathcal P _m$. From \cite[Proposition 3.7]{GKKP17}, we have the identification $\mathcal S ' / \mathcal P = \mathcal S _\infty '$.
Note that if $L = - \Delta$, the Laplacian on $\rn$, the distributions in $\mathcal S ' / \mathcal P = \mathcal S _\infty ^{'} $ are identical with the classical tempered distributions modulo polynomial.

From \cite[Lemma 2.6]{BBD20}, we can see that if $\varphi \in \mathscr S (\mathbb R)$ with supp $\varphi \subset(0,\infty)$, then $K_{  \varphi ( t \sqrt {L} ) }  (x, \cdot) \in \mathcal S _\infty $ and  $K_{  \varphi ( t \sqrt {L} ) }  (\cdot, y) \in \mathcal S _\infty $. Hence, for $f\in \mathcal S _\infty ' $ we can define
\begin{equation} \label{varphi}
	\varphi( t \sqrt{L}  ) f(x) = \langle f, K_{  \varphi( t \sqrt{L} ) }  (x, \cdot) \rangle .
\end{equation}

For  $a,b \in \mathbb R$, let $a\wedge b=\min\{a,b\} $ and $a\vee b = \max\{a,b\}$. 
For $x,y \in X$,  $\alpha >0$, let $ V(x   \wedge y, \alpha) = \min\{   V(x,\alpha),V(y,\alpha)  \}$ and $ V(x   \vee  y, \alpha) = \max\{   V(x,\alpha),V(y,\alpha)  \}$. 

%

\begin{lem} [Lemma 3.2, \cite{BDL18}]  \label{basic est}
	Let $\epsilon >0$.
	
	{\rm (i)} For any $1\le p \le \infty $, we have
	\begin{equation*}
		\bigg(   \int_X     \Big(  1+ \frac{\rho(x,y)}{\alpha}  \Big)^{-(n+ \epsilon)  p}  \mathrm {d} \mu(y)   \bigg)^{1/p} \lesssim V(x, \alpha)^{1/p}
	\end{equation*}
	for all $x\in X$ and $\alpha > 0$.
	
	{\rm (ii)} For all $f\in L_{\rm {loc}}^1 (X)$, we have
	\begin{equation*}
		\int_X \frac{1}{  V(x   \wedge y, \alpha) }  \Big(  1+ \frac{\rho(x,y)}{\alpha}  \Big)^{-n-\epsilon} |f(y)| \mathrm {d} \mu (y) \lesssim \mathcal M f(x)
	\end{equation*}
	for all $x\in X$ and $\alpha > 0$.
	
\end{lem}
\begin{defn}
	A function $\psi$ is called  a partition of unity if $\psi \in \mathscr S (\mathbb R)$ with supp $\psi \subset [1/2,2]$, $\int \psi (\xi )  \xi ^{-1} \mathrm{d} \xi  \neq 0$ and
	\begin{equation*}
		\sum_{j\in \mathbb Z} \psi_j (\lambda) =1 \; \mathrm{on} \; (0,\infty),
	\end{equation*}
	where $\psi_j (\lambda) := \psi(2^{-j} \lambda )$ for each $j \in \mathbb Z$.
 \end{defn}

\begin{lem}[Proposition 2.11, \cite{BBD20}] \label{iden}
	Let $\psi \in \mathscr S (\mathbb R)$ be such that supp $\psi \subset [1/2, 2]$ and $\int \psi(\xi )  \xi^{-1} \mathrm{d} \xi  \neq 0$. Then for any $f\in \mathcal S_\infty '$, we have
	\begin{equation*}
		f = c_{\psi}  \int_0 ^\infty  \psi  (t \sqrt{L}) (f)  \frac{ \mathrm{d} t }{t} \; \mathrm{in} \;  \mathcal S_\infty ',
	\end{equation*}
	where $c_\psi = [\int_0^\infty \psi(t)  t^{-1}  \mathrm{d}  t] ^{-1} .$
	Moreover, if $f\in\mathcal S '$, then there exists $g\in \mathcal P$ such that
	\begin{equation*}
		f-g = c_\psi \int_0 ^\infty \psi (t \sqrt{L})  f \frac{ \mathrm{d} t }{t} \; \mathrm{in} \; \mathcal S '.
	\end{equation*}
\end{lem}			

We will often use the following inequality  and may use these in the sequel without
stating any reasons.
 For all $x,y, z \in X$ and all $\alpha,W >0$, we have
\begin{equation*}
	\bigg(1 + \frac{\rho(x,y)}{\alpha} \bigg)^{-W} 	\bigg(1 + \frac{\rho(y,z)}{\alpha} \bigg)^{-W} \lesssim  	\bigg(1 + \frac{\rho(x,z)}{\alpha} \bigg)^{-W}.
\end{equation*}

\section{Bourgain-Morrey  Besov type and Triebel-Lizorkin type spaces associated with $L$ } \label{Bou spa}

\subsection{Definitions}
First, we recall the definitions  of the weighted Bourgain-Morrey spaces and
 the weighted homogeneous Bourgain-Morrey Besov-Triebel-Lizorkin spaces  on $X$  from  \cite{BX25}.
\begin{defn}
	Let $\mathcal D$ be the dyadic cubes in $X$ as in Remark \ref{dya cube}.
	Let $0<p\le t<\infty$
	, $0<r\le\infty$ and $\omega \in A_\infty$. Define $M_{p,\omega}^{t,r}  :=  M_{p,\omega}^{t,r}  (X)$ as the space of $f\in L_{\mathrm{loc}}^{p}(\omega)$
	such that
	\[
	\|f\|_{M_{p,\omega}^{t,r}} :=\bigg\|  \bigg\{\omega (Q_\tau ^k  ) ^{1/t-1/p}\bigg(\int_{Q_\tau ^k } |f|^{p}\omega(x) \mathrm {d} \mu(x)  \bigg)^{1/p}\bigg\}_{k \in\mathbb{Z}, \tau \in I_k}\bigg\|_{\ell^{r}}<\infty.
	\]
\end{defn}

\begin{defn} \label{def spa 1}
	Let $\mathcal D$ be the dyadic cubes in $X$ as in Remark \ref{dya cube}.
	Let $\psi$ be a partition of unity. Let $ 0< q\le \infty$, $s\in \mathbb R$ and $\omega \in A_\infty $. Let $0<p<t<r<\infty$ 	or $0<p\le t<r=\infty$.
	The weighted homogeneous Bourgain-Morrey Besov  space $\dot{B}^{s,q,\psi,L}_{p,t,r,\omega} (X)$ is defined as follows:
	\[
	\dot{B}^{s,q,\psi,L}_{p,t,r,\omega} (X) = \{ f\in \mathcal S_\infty ' : \| f \| _{ \dot{B}^{s,q,\psi,L}_{p,t,r,\omega} (X)} <\infty  \},
	\]
	where
	\[
	\| f\|_{\dot{B}^{s,q,\psi,L}_{p,t,r,\omega} (X)} =
	\bigg(  \sum_{j\in \mathbb{Z}}  2^{js}\Big\|  \psi_j \big(\sqrt{L}\big)f\Big\|^q_{M_{p,\omega}^{t,r}}   \bigg)^{1/q}.
	\]
	The weighted homogeneous Bourgain-Morrey	 Triebel-Lizorkin  space $\dot{F}^{s,q,\psi,L}_{p,t,r,\omega} (X)$ is defined by
	\[
	\dot{F}^{s,q,\psi,L}_{p,t,r,\omega} (X) = \{ f\in \mathcal S_\infty ' : \| f \| _{ \dot{F}^{s,q,\psi,L}_{p,t,r,\omega} (X)} <\infty  \},
	\]
	where
	\[
	\| f\|_{\dot{F}^{s,q,\psi,L}_{p,t,r,\omega} (X)} =  \bigg\|\bigg( \sum_{j\in \mathbb{Z}} 2^{jsq} \Big| \psi_j \big(\sqrt{L}\big)f \Big|^q\bigg)^{1/q}\bigg\|_{M_{p,\omega}^{t,r}}   .
	\]
\end{defn}

The classical Besov-Triebel-Lizorkin spaces can date back to Besov \cite{Besov59}, Triebel \cite{Triebel73} and Lizorkin \cite{Lizorkin74,Lizorkin742}. We refer readers to monographs \cite{Tri83,Tri92,Tri06,Tri20} of Triebel for the systematic theory of Besov-Triebel-Lizorkin spaces. Weighted Besov-Triebel-Lizorkin spaces with Muckenhoupt weights were investigated by  Bui \cite{Bui82,Bui84}, Bui et al.\cite{BPT96}, Besoy et al. \cite{BHT22}, Haroske et al. \cite{HP08}. Weighted Besov-Triebel-Lizorkin spaces with more general weight were studied by Rychkov \cite{R01}, Izuki and Sawano \cite{IS09,IS12}. Kozono and Yamazaki \cite{KY94}   introduced Besov-Morrey spaces. Later Besov-Triebel-Morrey spaces were  systematically investigated by Mazzucato \cite{M03,M032}, Tang and Xu \cite{TX05}, Sawano and Tanaka \cite{ST07},  Sawano \cite{S10,S08,S09} and Rosenthal \cite{R13}. Mixed-norm Besov-Triebel-Lizorkin spaces were systematically studied by Cleanthous et al. \cite{CGN17,CGN19,CGN192}, Georgiadis et al. \cite{GJN17}, Georgiadis and Nielsen \cite{GN16}. Besov-Triebel-Lizorkin spaces with variable exponents were investigated by the third author of the paper \cite{X08,X09}, Diening et al. \cite{DHR09}, Almeida and H{\"a}st{\"o} \cite{AH10}. 
Yang and Yuan introduced Triebel-Lizorkin-type spaces in \cite{YY08}.
Besov-Triebel-Lizorkin-type spaces were intensively investigated by Liang et al. \cite{LSUYY12}, Sawano \cite{SYY10}, Yang and Yuan \cite{YY10,YY13}, Yuan et al. \cite{YSY10}.
 Recently, Bu et al.  systematically studied  Besov-Triebel-Lizorkin type spaces with martrix weight in \cite{BHYY25I,BHYY23,BHYY24,BHYY25}.

Now, we introduce the weighted homogeneous Bourgain-Morrey-Besov type spaces and  the weighted homogeneous Bourgain-Morrey Triebel-Lizorkin type spaces.

\begin{defn} \label{def spa}
	Let $\mathcal D$ be the dyadic cubes in $X$ as in Remark \ref{dya cube}. 
	For each $Q \in \mathcal D _{\nu }$, let $j_Q := -\nu$. 
Let $\psi$ be a partition of unity. Let $ 0< q\le \infty$, $s\in \mathbb R$ and $\omega \in A_\infty $. Let $0<p<t<r<\infty$ 	or $0<p\le t<r=\infty$.
	We define the weighted homogeneous Bourgain-Morrey-Besov type space $\dot{\mathcal B}^{s,q,\psi,L}_{p,t,r,\omega} (X)$ as follows:
	\begin{equation*}
		\dot{\mathcal B}^{s,q,\psi,L}_{p,t,r,\omega} (X) = \{ f\in \mathcal S_\infty ' : \| f \| _{ \dot{\mathcal B}^{s,q,\psi,L}_{p,t,r,\omega} (X)} <\infty  \},
	\end{equation*}
	where
	\begin{equation*}
		\| f\|_{\dot{\mathcal B}^{s,q,\psi,L}_{p,t,r,\omega} (X)} = 
		\bigg(    \sum_{Q \in \mathcal D}  \omega (Q)^{r/t-r/p} \Big(  \sum_{j=j_Q}^\infty \Big(   \int_Q | 2^{js}  \psi_j (   \sqrt{L} )f (x)|^p \omega (x)\d \mu (x)   \Big)^{q/p}   \Big)^{r/q}                 \bigg)^{1/r}.
		\end{equation*}
	Similarly, the weighted homogeneous Bourgain-Morrey	Triebel-Lizorkin type  space $\dot{\mathcal F}^{s,q,\psi,L}_{p,t,r,\omega} (X)$ is defined by
		\begin{equation*}
		\dot{\mathcal F}^{s,q,\psi,L}_{p,t,r,\omega} (X) = \{ f\in \mathcal S_\infty ' : \| f \| _{ \dot{\mathcal F}^{s,q,\psi,L}_{p,t,r,\omega} (X)} <\infty  \},
	\end{equation*}
	where
	\begin{equation*}
		\| f\|_{\dot{\mathcal F}^{s,q,\psi,L}_{p,t,r,\omega} (X)} = \bigg(    \sum_{Q \in \mathcal D}  \omega (Q)^{r/t-r/p} \Big(  \int_Q  \Big(  \sum_{j=j_Q}^\infty | 2^{js}  \psi_j  ( \sqrt{L} )f (x)|^q \Big)^{p/q}  \omega (x) \d \mu (x)      \Big)^{r/p}                 \bigg)^{1/r} .
	\end{equation*}
	\end{defn}
	
	\begin{rem}
	 	If $p =\infty$, then $M_p^{t,r} =L^\infty$. This case is included in \cite[Definitions 3.1, 3.16]{BBD20}. Hence, we do not consider the case $p=\infty$.
From  Definitions \ref{def spa 1} and \ref{def spa}, we have 
\begin{equation*}
		\| f\|_{\dot{\mathcal F}^{s,q,\psi,L}_{p,t,r,\omega} (X)} \le 	\| f\|_{\dot{F}^{s,q,\psi,L}_{p,t,r,\omega} (X)}.
\end{equation*}	
In what follows, the symbol $\hookrightarrow$ stands for continuous embedding.
Hence 
\begin{equation*}
	\dot{F}^{s,q,\psi,L}_{p,t,r,\omega} (X) \hookrightarrow \dot{\mathcal F}^{s,q,\psi,L}_{p,t,r,\omega} (X).
\end{equation*}
If $p =q$, then 
\begin{equation*}
		\| f\|_{\dot{\mathcal B}^{s,p,\psi,L}_{p,t,r,\omega} (X)} = 	\| f\|_{\dot{\mathcal F}^{s,p,\psi,L}_{p,t,r,\omega} (X)}.
\end{equation*}
	\end{rem}
		
	Let $ 0< q\le \infty$, and $\omega \in A_\infty $. Let $0<p<t<r<\infty$ 	or $0<p\le t<r=\infty$.	
For convenience, we define sequence valued Bourgain-Morrey spaces  by 
\begin{align*}
	\| \{ g_j\}_{j\in \mathbb Z}  \| _{M_{p,\omega}^{t,r} (\ell^q)} := \Big\|  \| \{ g_j \}_{j\in \mathbb Z} \| _{\ell^q} \Big\| _{M_{p,\omega}^{t,r} },   
\end{align*}
and 
\begin{align*}
	\| \{ g_j\}_{j\in \mathbb Z}  \| _{  \ell ^q ( M_{p,\omega}^{t,r} ) } := \Big\|    \big  \{ \|  g_j \| _{   M_{p,\omega}^{t,r}   }  \big \}_{j\in \mathbb Z}  \Big\| _{  \ell^q },
\end{align*}
with finite norm respectively, where $\{ g_j \}_{j\in \mathbb Z}$  is a  sequences of measurable functions on $X$. 
		
\begin{defn}
	Let $\mathcal D$ be the dyadic cubes in $X$ as in Remark \ref{dya cube}.   The dyadic  maximal operator  with respect to $\mathcal D$ is defined by
	\begin{equation*}
		\mathcal M^{\mathcal D}  f (x ) =\sup_{x\in Q, Q\in \mathcal D }  \frac{1}{V(Q)}\int_Q |f(y)|  \d  \mu (y ).
	\end{equation*}
	Let $\omega$ be a weight and  define the weighted dyadic  maximal operator  with respect to $\mathcal D$  by
	\begin{equation*}
		\mathcal M^{\mathcal D} _\omega  f (x ) =\sup_{x\in Q, Q\in \mathcal D }  \frac{1}{\omega (Q)}\int_Q |f(y)|  \omega (y) \d  \mu (y ).
	\end{equation*}
\end{defn}
The following lemma from \cite[Theorem 2.6]{CS15} says that there  is a  finite family of dyadic cubes such that an arbitrary quasi-metric ball is contained in a dyadic cube from one of these cubes. Such a finite family of dyadic cubes is referred to as an adjacent system of dyadic cubes.
\begin{lem} \label{ball adjacent}
	There exists a positive integer $K=K(X)$, a finite constant $C=C(X)$, and a finite collection of dyadic cubes, $\mathcal D^{\theta}$, $1\le \theta  \le K $, such that given any ball $B=B(x,r) \subset X$ there exists  $\theta$ and a dyadic cube $ Q \in \mathcal D^\theta $ such that $B\subset Q$ and {\rm diam} $Q \le C r$. 
\end{lem}

\begin{lem}  [Proposition 2.13, \cite{CS15}]  \label{M le dya}
	Let $\{ \mathcal D ^\theta \}_{1\le \theta \le K}$ be the adjacent dyadic system from Lemma  \ref{ball adjacent}. Then there exists a constant $C\ge 1$ such that for all $f\in L^1_{\rm loc} (X)$,
	\begin{equation*}
		\mathcal  M f (x) \le C \sum_{\theta=1}^K \mathcal M^{ \mathcal D^\theta} f (x).
	\end{equation*}
\end{lem}

\begin{rem}
	If we replace $\mathcal M$, $\mathcal M ^\mathcal {D^\theta}$    by $\mathcal M_\omega$ , $\mathcal M_\omega ^\mathcal {D^\theta}$in Lemma \ref{M le dya} respectively, the result still holds true.
\end{rem}
		
	The following result coming from \cite[Theorem 3.1]{BX25} says that	the weighted  Bourgain-Morrey spaces are independent of  the choice of dyadic cubes.
\begin{lem} \label{spa dya inde}
	Let   $\{ \mathcal D ^\theta \}_{1\le \theta \le K}$ be the adjacent dyadic system from Lemma  \ref{ball adjacent}. Then $\| \cdot \|_{ M_{p,\omega}^{t,r} } $ and  $\| \cdot \|_{ M_{p,\omega}^{t,r} (  \widetilde{ \mathcal D }   )} $  are equivalent for any dyadic cubes $ \widetilde{ \mathcal D } \in \{ \mathcal D ^\theta \} _{1\le \theta \le K}$ where
	\begin{equation*}
		\|f\|_{ M_{p,\omega}^{t,r} ( \widetilde{ \mathcal D }  )} : = \bigg\|  \bigg\{\omega (Q ) ^{1/t-1/p}\bigg(\int_{Q } |f(x)|^{p}\omega(x) \mathrm {d} \mu(x)  \bigg)^{1/p}\bigg\}_{Q \in \widetilde{ \mathcal D } }\bigg\|_{\ell^{r}}
	\end{equation*}
	and $f$ is  a measurable functions over $X$.
\end{lem}

	The following result comes from \cite[Theorem 3.2]{BX25}.

\begin{lem}\label{M bour weight}
	Let $1<p\le t<r=\infty $ or  let  $1<p <t <r <\infty$.  Let $\omega \in A_p$.
Let $\beta = 1-  (1- \alpha_0)^p   / [\omega]_{A_p}$ where $\alpha_0$  is the reverse doubling  constant of $X$. If $ r >  -nt / \log \beta  $, then 	

{\rm(i) }  the Hardy-Littlewood maximal operator $\mathcal M, \mathcal M_\omega$ are bounded on $M_{p,\omega}^{t,r}$;

{\rm 	(ii)  }  for all sequences $\{f_j\} _{j \in \mathbb Z}  \in M_{p,\omega}^{t,r} (\ell^u)$, $1<u\le \infty $,   we have
\[
\|  \{  \mathcal M f_j \}_{j\in \mathbb Z}  \|_{  M_{p,\omega}^{t,r} (\ell^u) } \lesssim 	\|  \{  f_j \}_{j \in \mathbb Z}\|_{  M_{p,\omega}^{t,r} (\ell^u) },
\]
and
\[
\|  \{  \mathcal M _\omega f_j \}_{j\in \mathbb Z}  \|_{  M_{p,\omega}^{t,r} (\ell^u) } \lesssim 	\|  \{  f_j \}_{j \in \mathbb Z} \|_{  M_{p,\omega}^{t,r} (\ell^u) }.
\]
\end{lem}

	The variant of weighted (and unweighted) Fefferman-Stein maximal inequality and  maximal inequality comes from \cite[Theorem 3.3]{BX25}.
\begin{lem}\label{max FS Bou}
 Let $ 0< q\le \infty$  and $\omega \in A_\infty $. Let $0<p<t<r<\infty$ 	or $0<p\le t<r=\infty$.
Then the following statements hold:

{\rm (i)} 
For $0<A<p$ and $\omega \in A_{p/A}$, let
 $\beta = 1-  (1- \alpha_0)^{p/A} /  [\omega]_{A_{p/A}} $. If $ r >  -nt / \log \beta  $,
then
\[
\|  	\mathcal M _{A } f \|_{  M_ {p,\omega }^{t,r}  }  \lesssim \| f \| _   {  M_ {p,\omega }^{t,r}  }.
\]

{\rm (ii)} For $0<A<p$ and $\omega \in A_{p/A}$, let
 $\beta = 1-  (1- \alpha_0)^{p/A} /  [\omega]_{A_{ p /A }} $.    If $r >  -nt / \log \beta  $,
then
\[
\| \mathcal M _{A , \omega} f \|_{M_ {p,\omega }^{t,r} }  \lesssim \| f \| _{ M_ {p,\omega }^{t,r}  }.
\]

{\rm (iii)} For  $0<A <\min \{ p,q \} $ and $\omega \in A_{p/A}$, 	let $\beta = 1-  (1- \alpha_0)^{p/A} /  [\omega]_{A_{ p /A }} $.    If $r >  -nt / \log \beta  $,
then
for any sequence of measurable functions $ \{  f_\nu \}$,
\begin{equation*} 
	\bigg\| \bigg(  \sum_\nu | \mathcal M _A f_\nu |^q \bigg)^{1/q} \bigg\|_{   M_ {p,\omega }^{t,r}   } \lesssim  \bigg\| \bigg(  \sum_\nu |  f_\nu |^q \bigg)^{1/q} \bigg\|_{ M_ {p,\omega }^{t,r}  }.
\end{equation*}

{\rm (iv)} For  $0<A < \min \{ p,q \}  $ and $\omega \in A_{p/A}$,
let $\beta = 1-  (1- \alpha_0)^{p/A} /  [\omega]_{A_{ p /A }} $.    If $r >  -nt / \log \beta  $,
and $ \{a_\nu \}  \in \ell^q \cap  \ell^1 $, then
\[
\bigg\|  \sum_j  \bigg(  \sum_\nu |a_{j-\nu} \mathcal M _A f_\nu |^q \bigg)^{1/q} \bigg\|_{ M_ {p,\omega }^{t,r}    } \lesssim  \bigg\| \bigg(  \sum_\nu |  f_\nu |^q \bigg)^{1/q} \bigg\|_{  M_ {p,\omega }^{t,r}    },
\]
where the implicit constant is independent of  the sequence of functions $\{f_\nu\}$.

{\rm (v)} For  $0<A <p$ and $\omega \in A_{p/A}$, 
let $\beta = 1-  (1- \alpha_0)^{p/A} /  [\omega]_{A_{ p /A }} $.    If $r >  -nt / \log \beta  $ and
 $ \{a_\nu \}  \in \ell^q \cap  \ell^1 $,
then
\[
\sum_j  \bigg(  \sum_\nu  \big|  a_{j-\nu} \|  \mathcal M _A f_\nu \| _{M_ {p,\omega }^{t,r}   } \big|^q \bigg)^{1/q}   \lesssim  \bigg(  \sum_\nu  \|  f_\nu  \| _ {  M_ {p,\omega }^{t,r}  } ^q \bigg)^{1/q},
\]
where the implicit constant is independent of the sequence of functions $\{f_\nu\}$.
\end{lem}

For brevity, we define the following notations.
\begin{defn}
	Let $ 0< q\le \infty$, and $\omega \in A_\infty $. Let $0<p<t<r<\infty$ 	or $0<p\le t<r=\infty$. Let $G = \{g_j\}_{j\in \mathbb Z}$ be a sequence of measurable functions on $X$.
	Denote by $ \widetilde{ \ell^q ( M_{t,r} ^{p, \omega})   }$  the set of  all sequences $\{g_j\}_{j\in\mathbb Z}$ of measurable functions on $X$ such that 
	\begin{equation*}
		\|G \|_{ \widetilde{ \ell^q ( M_{t,r} ^{p, \omega})   }} :=
		\bigg(    \sum_{Q \in \mathcal D}  \omega (Q)^{r/t-r/p} \Big(  \sum_{j=j_Q}^\infty \Big(   \int_Q         |   g_j (x)    |^p
		\omega (x)\d \mu (x)   \Big)^{q/p}   \Big)^{r/q}                 \bigg)^{1/r}  <\infty.
	\end{equation*}
	Similarly, 	denote by $\widetilde{   M_{t,r} ^{p, \omega} (\ell^q )  }$  the set of all sequences all sequences $\{g_j\}_{j\in\mathbb Z}$ of measurable functions on $X$ such that 
	\begin{equation*}
		\|G \|_{ \widetilde{   M_{t,r} ^{p, \omega} (\ell^q )  } } :=
		\bigg(    \sum_{Q \in \mathcal D}  \omega (Q)^{r/t-r/p} \Big(   \int_Q     \Big(  \sum_{j=j_Q}^\infty     |   g_j (x)    |^q \Big)^{p/q} 
		\omega (x)\d \mu (x)     \Big)^{r/p}                 \bigg)^{1/r}  <\infty.
	\end{equation*}
\end{defn}
\begin{rem} \label{min tri}
		It is basic that 
	\begin{align*}
		\|  f +g  \|_{ \widetilde{ \ell^q ( M_{t,r} ^{p, \omega})   }} ^{\min( 1,p,q )} \le 	\|  f  \|_{ \widetilde{ \ell^q ( M_{t,r} ^{p, \omega})   }} ^{\min( 1,p,q ) } + 	\|  g \|_{ \widetilde{ \ell^q ( M_{t,r} ^{p, \omega})   }} ^{\min( 1,p,q )}, 	
	\end{align*}
and
\begin{equation*}
		\|  f +g  \|_{  \widetilde{   M_{t,r} ^{p, \omega} (\ell^q )  } } ^{\min( 1,p,q )} \le 	\|  f  \|_{ \widetilde{   M_{t,r} ^{p, \omega} (\ell^q )  }  } ^{\min( 1,p,q ) } + 	\|  g \|_{\widetilde{   M_{t,r} ^{p, \omega} (\ell^q )  }  } ^{\min( 1,p,q )}.
\end{equation*}
	We call the above inequalities  $\min(1,p,q)$-inequality.
\end{rem}

\begin{thm} \label{M ell Bour}
	Let $1<p\le t<r=\infty $ or  let  $1<p <t <r <\infty$.  Let $\omega \in A_p$ and
$\beta = 1-  (1- \alpha_0)^{p} /  [\omega]_{A_{p}} $. 	Let $ 0< q\le \infty$. If $r >- n t/  \log{  \beta } $,
		 then the Hardy-Littlewood maximal operator $\mathcal M, \mathcal M_\omega$ are bounded on $ \widetilde{   \ell^q ( M_{p,\omega}^{t,r}  ) }.$

\end{thm}
\begin{proof}
	The proof is similar as \cite[Theorem  3.2]{BX25}. For the reader's convenience, we give some details. 
		Thanks to Lemmas \ref{M le dya}  and  \ref{spa dya inde}, it is sufficient to show the boundedness for the dyadic maximal operator $\mathcal M ^\mathcal D$,  $\mathcal M_\omega ^\mathcal D$   instead of investigating the operator $\mathcal M$, $\mathcal M_\omega$. Here and below, we identify $\mathcal M$, $\mathcal M_\omega$  with $\mathcal M ^\mathcal D$,  $\mathcal M_\omega ^\mathcal D$, respectively.
	
	For a fixed $Q \in \mathcal D$, each $j\ge j_Q$, let $f_j = f_{j,1} +f_{j,2}$ and $f_{j,1} =f_j \chi _Q$.  By (\ref{M a omega Lp}), we have
	\begin{equation}
		\Big( \sum_{j=j_Q}^\infty  \Big(   \int_{Q} |	\mathcal M _\omega f_{j,1} |^{p}\omega(x) \d \mu (x) \Big)^{q/p} \Big)^{1/q} \lesssim  	\Big( \sum_{j=j_Q}^\infty  \Big(   \int_{Q} |	 f_{j,1} (x)  |^{p}\omega(x) \d \mu (x) \Big)^{q/p} \Big)^{1/q}.
	\end{equation}
	For $k\in \mathbb N$, let $Q_k$ be the $k$-th dyadic parent of $Q$, which is the dyadic cube in $\mathcal D$ satisfying $Q \subset Q_k $ and $\ell(Q_k) \approx 2^k \ell (Q).$  
	By the H\"older inequality, we have
	\begin{align} 
		\nonumber
	&	\omega(Q ) ^{1/t-1/p}\Big( \sum_{j=j_Q}^\infty  \Big(  \int_{Q }\mathcal M _\omega f_2(x)^{p}\omega(x) \d \mu (x) \Big)^{q/p} \Big)^{1/q} \\
&	\le \omega (Q)^{1/t} \Big( \sum_{j=j_Q}^\infty  \Big(
\sum_{k=1}^\infty   \frac{1}{\omega(Q_k)}  \int_{Q_k } | f (x)|^{p}\omega(x) \d \mu (x)  \Big)^{q/p} \Big)^{1/q}.
	\end{align}

As in Remark \ref{dya cube}, there are at most $c2^{kn}$ dyadic cubes $P$ such that $P \subset Q_k $ where $c$ is a constant depending only on $X$.
	Using 
$\min(1,p,q)$-inequality (see Remark \ref{min tri}),
we have
	\begin{align*}
	\nonumber
	& 	\| \mathcal M _\omega f_2  \|_{ \widetilde{ \ell^q ( M_{t,r} ^{p, \omega})   }} ^ { \min(1,p,q) }  \\
	\nonumber
	&\lesssim 	\bigg(    \sum_{Q \in \mathcal D}  
	\omega (Q)^{r/t} \Big( \sum_{j=j_Q}^\infty  \Big(
	\sum_{k=1}^\infty   \frac{1}{\omega(Q_k)}  \int_{Q_k } | f (x)|^{p}\omega(x) \d \mu (x)  \Big)^{q/p} \Big)^{r/q}   \bigg)^{ \min(1,p,q)  /r} \\
	\nonumber
	& \lesssim \sum_{k=1}^\infty  \bigg(    \sum_{Q \in \mathcal D}  
	\omega (Q)^{r/t} \Big( \sum_{j=j_Q}^\infty  \Big(
	 \frac{1}{\omega(Q_k)}  \int_{Q_k } | f (x)|^{p}\omega(x) \d \mu (x)  \Big)^{q/p} \Big)^{r/q}   \bigg)^{ \min(1,p,q) /r} \\
	\nonumber
	& \lesssim \sum_{k=1}^\infty \beta ^{ \min(1,p,q)  k/t}
	\\
	& \quad \times
	 \bigg(    \sum_{Q \in \mathcal D}  
	\omega (Q_k)^{r/t} \Big( \sum_{j=j_Q}^\infty  \Big(
	  \frac{1}{\omega(Q_k)}  \int_{Q_k } | f (x)|^{p}\omega(x) \d \mu (x)  \Big)^{q/p} \Big)^{r/q}   \bigg)^{ \min(1,p,q) /r} \\
	\nonumber
	& \lesssim \sum_{k=1}^\infty \beta ^{ \min(1,p,q) k/t} 2^{ \min(1,p,q) kn/r} 
	\\
	\nonumber
	& \quad \times
	 \bigg(    \sum_{Q_k \in \mathcal D}  
	\omega (Q_k)^{r/t} \Big( \sum_{j=j_Q}^\infty  \Big(
	 \frac{1}{\omega(Q_k)}  \int_{Q_k } | f (x)|^{p}\omega(x) \d \mu (x)  \Big)^{q/p} \Big)^{r/q}   \bigg)^{\min(1,p,q) /r} \\	
	& \lesssim \|  f  \|_{ \widetilde{ \ell^q ( M_{t,r} ^{p, \omega})   }}  ^{  \min(1,p,q)  },
\end{align*}
	as long as $r > - n t / \log  \beta $.
	The proof of $\mathcal M$  is similar and we refer the reader to \cite[Theorem 3.2]{BX25} for the details.
\end{proof}

Similar with Lemma \ref{max FS Bou}, we give the  following result without  proofs. 

	\begin{thm} \label{Ma ell q Mptr}
	 Let $ 0< q\le \infty$. Let $0<p<t<r<\infty$ 	or $0<p\le t<r=\infty$.
	  Let $0<A<p$ such that  $\omega \in A_{p/A}$ and $\beta = 1-  (1- \alpha_0)^{p/A} /  [\omega]_{A_{p/A}} $. If $r >- n t / \log{ \beta } $,  	
	 then 

		\begin{equation*}
			\|  	\mathcal M _{A } f \|_{   \widetilde{   \ell^q ( M_{p,\omega}^{t,r}  ) }  }  \lesssim \| f \| _   {  \widetilde{   \ell^q ( M_{p,\omega}^{t,r}  ) }  }.
		\end{equation*} 
	and
		\begin{equation*}
			\|  	\mathcal M _{A , \omega} f \|_{\widetilde{   \ell^q ( M_{p,\omega}^{t,r}  ) } }  \lesssim \| f \| _   {  \widetilde{   \ell^q ( M_{p,\omega}^{t,r}  ) }  }.
		\end{equation*} 
\end{thm}

The following result is useful to the proof of characterizations of weighted Bourgain-Morrey-Besov type and Triebel-Lizorkin type spaces.	
\begin{thm}[Hardy type inequality]  \label{hardy bourgain}
		Let $ 0< q\le \infty$. Let $0<p<t<r<\infty$ 	or $0<p\le t<r=\infty$. 
	 Let $0<A<p$ such that   $\omega \in A_{p/A}$. 
	Let  $\{g_m\}_{m\in\mathbb Z}$ be  a sequence of  measurable functions on $X$. For all $j\in \mathbb Z$, $x\in X$, let $G_j (x) =\sum_{m\in\mathbb Z} 2^{-|m-j|\delta}  g_m (x)  $. 	If  $\delta > (1/p -1/t) n p /A$, then
	\begin{equation*}
			\|G \|_{ \widetilde{ \ell^q ( M_{t,r} ^{p, \omega})   }} \lesssim 	\|g\|_{ \widetilde{ \ell^q ( M_{t,r} ^{p, \omega})   }},
	\end{equation*}
	and 
	\begin{equation*}
			\|G \|_{ \widetilde{   M_{t,r} ^{p, \omega} (\ell^q )  }  } \lesssim 	\|g\|_{ \widetilde{   M_{t,r} ^{p, \omega} (\ell^q )  } }.
	\end{equation*}
\end{thm}

\begin{rem}
	 The condition $\delta > (1/p -1/t) n p /A$ can be replaced by $\delta > (1/p -1/t) n q_\omega$. Indeed, let $\epsilon>0$ be sufficiently small such that $\omega \in A_{q_\omega +\epsilon}$ and  $\delta > (1/p -1/t) n (q_\omega +\epsilon)  > (1/p -1/t) n q_\omega$. Then Theorem \ref{hardy bourgain} holds under the condition $\delta > (1/p -1/t) n (q_\omega +\epsilon)$.
\end{rem}

\begin{proof}[Proof of Theorem \ref{hardy bourgain}]
	The ideas comes from \cite[Lemma 2.3]{YY10}.
	By similarity, we only prove the first inequality. 
	For brevity, we write
	\begin{equation*}
			I_Q   :=    \omega (Q)^{1/t-1/p} \Big(  \sum_{j=j_Q}^\infty \Big(    \int_Q         \Big|  \sum_{m\in\mathbb Z} 2^{-|m-j|\delta}  g_m (x)      \Big|^p
		\omega (x)\d \mu (x)   \Big)^{q/p}   \Big)^{1/q}   .
	\end{equation*}

	Case $0<p\le t<r=\infty$.

	Subcase  $p \in (0,1]$.
	 Applying the inequality that for all $d \in (0,1] $ and sequence $\{a_j\}_{j\in \mathbb Z}$,
	\begin{equation} \label{ell^d}
		\Big(  \sum_{j \in \mathbb Z} |a _j |     \Big) ^d \le \sum_{j \in \mathbb Z} |a_j|^d,
	\end{equation}
	we obtain 
	\begin{align}
		\nonumber
		I_Q 
		& \lesssim    \omega (Q)^{1/t-1/p} \Big(  \sum_{j=j_Q}^\infty \Big(   \sum_{m\in\mathbb Z} 2^{-|m-j|\delta  p}   \int_Q         |     g_m (x)     |^p
		\omega (x)\d \mu (x)   \Big)^{q/p}   \Big)^{1/q}                .
	\end{align}
	When $q \in (0,p]$, applying  (\ref{ell^d}) again,  we obtain
		\begin{align} \label{IQ le IQ1 IQ2}
		\nonumber
	I_Q	& \lesssim
	\omega (Q)^{1/t-1/p} \left(  \sum_{j=j_Q}^\infty    \sum_{m\in\mathbb Z} 2^{-|m-j|\delta  q}  \left(   \int_Q         |     g_m (x)     |^p
	\omega (x)\d \mu (x) \right) ^{q/p}     \right)^{1/q}        \\
		\nonumber
		&\lesssim  \omega (Q)^{1/t-1/p} \left(  \sum_{j=j_Q}^\infty    \sum_{m = j_Q}^\infty 2^{-|m-j|\delta  q}  \left(   \int_Q         |     g_m (x)     |^p
		\omega (x)\d \mu (x) \right) ^{q/p}     \right)^{1/q}  
		\\
		\nonumber
		& \quad + \omega (Q)^{1/t-1/p} \left(  \sum_{j=j_Q}^\infty    \sum_{m= -\infty}^{ j_Q - 1} 2^{-|m-j|\delta  q}  \left(   \int_Q         |     g_m (x)     |^p
		\omega (x)\d \mu (x) \right) ^{q/p}     \right)^{1/q}    
		\\
		&      =: I_{Q,1} +I_{Q,2}. 
	\end{align}
The estimate of $I_{Q,1}$ is easy and we have
\begin{align*}
	I_{Q,1}  & \le \|g\|_{ \widetilde{ \ell^q ( M_{t,\infty} ^{p, \omega}) }} \left(  \sum_{j=j_Q}^\infty    \sum_{m = j_Q}^\infty 2^{-|m-j|\delta  q} \right)^{1/q} \lesssim  \|g\|_{ \widetilde{ \ell^q ( M_{t,\infty} ^{p, \omega}) }}.
\end{align*}
Now we estimate $I_{Q,2}$. Note that for each $Q\in \D$, there is one cube $Q' \in \D _m$ such that $Q \subset Q'$ and 
\begin{equation*}
	\frac{\omega (Q ')}{ \omega  (Q)} \le [\omega]_{A_{P/A}} \left( \frac{\mu (Q')}{\mu (Q)}  \right)^{p/A} \lesssim 2^{(j_Q -m)n p/A},
\end{equation*} 
where we use the properties of weight and doubling condition. Since $ 1/t - 1/p \le 0 $, 
\begin{equation} \label{omega Q le omega fa Q}
	\omega  (Q) ^{1/t- 1/p} \lesssim  \omega (Q ')^{1/t- 1/p}  2^{  (1/p -1/t) (j_Q -m)n p/A }.
\end{equation}
Hence 
\begin{align*}
	I_{Q,2} & \lesssim \omega (Q ')^{1/t- 1/p}  2^{  (1/p -1/t) (j_Q -m)n p/A }  \left(  \sum_{j=j_Q}^\infty    \sum_{m= - \infty}^{ j_Q - 1} 2^{-|m-j|\delta  q}  \left(   \int_Q         |     g_m (x)     |^p
	\omega (x)\d \mu (x) \right) ^{q/p}     \right)^{1/q} \\
	& \le  \omega (Q ')^{1/t- 1/p}  \\
	& \quad \times \left(  \sum_{j=j_Q}^\infty    \sum_{m= - \infty}^{ j_Q - 1} 2^{-|m-j|\delta  q}   2^{  (1/p -1/t) (j_Q -m)n q p/A }  \left(   \int_{Q'}        |     g_m (x)     |^p
	\omega (x)\d \mu (x) \right) ^{q/p}     \right)^{1/q} \\
	& \lesssim  \|g\|_{ \widetilde{ \ell^q ( M_{t,\infty} ^{p, \omega}) }}  \left(     \sum_{m= -\infty}^{ j_Q - 1}  2^{  (1/p -1/t) (j_Q -m)n q p/A } \sum_{j=j_Q}^\infty   2^{-( m-j) \delta  q}   \right)^{1/q} \\
	& \lesssim  \|g\|_{ \widetilde{ \ell^q ( M_{t,\infty} ^{p, \omega}) }}  \left(     \sum_{m= -\infty}^{ j_Q - 1}  2^{  (1/p -1/t) (j_Q -m)n q p/A }   2^{-( m-j_Q) \delta  q}   \right)^{1/q}  \\
	& \lesssim \|g\|_{ \widetilde{ \ell^q ( M_{t,\infty} ^{p, \omega}) }}
\end{align*}
as long as $  (1/p -1/t) n p /A - \delta <0 $. That is $\delta > (1/p -1/t) n p /A$.

When $q \in (p,\infty] $, choosing $\epsilon \in (0, \delta - (1/p -1/t) n p /A  )$, and applying H\"older's inequality, similarly to the above argument, we obtain
\begin{align*}
	I_Q &  \lesssim  \omega (Q)^{1/t-1/p} \left(  \sum_{j=j_Q}^\infty   \sum_{m\in\mathbb Z} 2^{-|m-j| ( \delta -\epsilon)  q}  \left(  \int_Q         |     g_m (x)     |^p
	\omega (x)\d \mu (x)   \right)^{q/p}   \right)^{1/q}      \\
	&  \lesssim  \omega (Q)^{1/t-1/p} \left(  \sum_{j=j_Q}^\infty   \sum_{m = j_Q}^\infty 2^{-|m-j| ( \delta -\epsilon)  q}  \left(  \int_Q         |     g_m (x)     |^p
	\omega (x)\d \mu (x)   \right)^{q/p}   \right)^{1/q}      \\
	& \quad + \omega (Q)^{1/t-1/p} \left(  \sum_{j=j_Q}^\infty   \sum_{m= -\infty}^{j_Q -1} 2^{-|m-j| ( \delta -\epsilon)  q}  \left(  \int_Q         |     g_m (x)     |^p
	\omega (x)\d \mu (x)   \right)^{q/p}   \right)^{1/q}      \\
	& \lesssim \|g\|_{ \widetilde{ \ell^q ( M_{t,\infty} ^{p, \omega}) }}.
\end{align*}
Subcase $ p \in (1,\infty)$. Applying Minkowski's inequality, we have
\begin{align*}
	I_Q \lesssim \omega (Q)^{1/t-1/p} \left(  \sum_{j=j_Q}^\infty \left(  \sum_{m\in\mathbb Z} 2^{-|m-j|\delta}    \left(\int_Q  | g_m (x) |^p \d  \mu (x)   \right)^{1/p}          \right)^{q}   \right)^{1/q}.    
\end{align*}
Then by  H\"older's inequality when $q \in (1,\infty]$ or (\ref{ell^d}) when $q \in (0,1)$, we also obtain $I_Q \lesssim \|g\|_{ \widetilde{ \ell^q ( M_{t,\infty} ^{p, \omega}) }}$, which further implies that
\begin{equation*}
		\|G \|_{ \widetilde{ \ell^q ( M_{t,\infty} ^{p, \omega})   }}  = \sup_{ Q \in \mathcal D } I_Q \lesssim 	\|g\|_{ \widetilde{ \ell^q ( M_{t,\infty} ^{p, \omega})   }}. 
\end{equation*}
Thus the proof of case $ r=\infty$ is complete. 

Case $0<p<t<r<\infty$.

	If $0< p \le 1 $ and $q \in (0,p]$,
	from (\ref{IQ le IQ1 IQ2}), we have
	\begin{equation*}
		\|G \|_{ \widetilde{ \ell^q ( M_{t,r} ^{p, \omega})   }} = \left( \sum_{Q  \in \mathcal D}	I_Q ^r \right)^{1/r} \lesssim   \left( \sum_{Q  \in \mathcal D}I_{Q,1} ^r \right)^{1/r}+  \left( \sum_{Q  \in \mathcal D}I_{Q,2} ^r \right)^{1/r} .
	\end{equation*}
	It is clear that 
	\begin{align*}
	 & \left( \sum_{Q  \in \mathcal D}I_{Q,1} ^r \right)^{1/r} \\
	 & \lesssim 	\bigg(    \sum_{Q \in \mathcal D}  \omega (Q)^{r/t-r/p} \left(  \sum_{m= j_Q}^\infty   2^{-|m-j_Q|\delta  q} \left(    \int_Q         |     g_m (x)     |^p
	 \omega (x)\d \mu (x)   \right)^{q/p}   \right)^{r/q}                 \bigg)^{1/r} \\
	 \nonumber
	 & \lesssim		\bigg(    \sum_{Q \in \mathcal D}  \omega (Q)^{r/t-r/p} \left(  \sum_{m= j_Q}^\infty   \left(    \int_Q         |     g_m (x)     |^p
	 \omega (x)\d \mu (x)   \right)^{q/p}   \right)^{r/q}                 \bigg)^{1/r}  = \|g \|_{ \widetilde{ \ell^q ( M_{t,r} ^{p, \omega})   }} .
	\end{align*}

	Note that $q \le p < t<r <\infty$. Let   $\epsilon \in (0, \delta - (1/p -1/t) n p /A  )$.
	As for $\left( \sum_{Q  \in \mathcal D}I_{Q,2} ^r \right)^{1/r}$, using (\ref{omega Q le omega fa Q}) and H\"older's inequality with $r/q>1$, we have 
	\begin{align*}
		I_{Q,2} ^r &= \omega (Q)^{r/t-r/p} \left(  \sum_{j=j_Q}^\infty    \sum_{m= -\infty}^{ j_Q - 1} 2^{-|m-j|\delta  q}  \left(   \int_Q         |     g_m (x)     |^p
		\omega (x)\d \mu (x) \right) ^{q/p}     \right)^{r/q}  \\
		& \lesssim  \omega (Q ')^{r/t- r/p} \\
		& \quad \times  \left(  \sum_{j=j_Q}^\infty    \sum_{m= -\infty}^{ j_Q - 1} 2^{-|m-j|\delta  q} 2^{  (1/p -1/t) (j_Q -m)n q p/A } \left(   \int_Q         |     g_m (x)     |^p
		\omega (x)\d \mu (x) \right) ^{q/p}     \right)^{r/q}  \\
		& \lesssim  \omega (Q ')^{r/t- r/p}    \sum_{j=j_Q}^\infty    \sum_{m= -\infty}^{ j_Q - 1} 2^{-|m-j| ( \delta -\epsilon)  r} 2^{ (1/p -1/t) (j_Q -m)n r p/A } \left(   \int_Q         |     g_m (x)     |^p
		\omega (x)\d \mu (x) \right) ^{r/p}     \\
		& \le    \sum_{j=j_Q}^\infty    \sum_{m= -\infty}^{ j_Q - 1} 2^{-|m-j| ( \delta -\epsilon)  r} 2^{ (1/p -1/t) (j_Q -m)n r p/A } \omega (Q ')^{r/t- r/p}  \left(   \int_{Q'}        |     g_m (x)     |^p
		\omega (x)\d \mu (x) \right) ^{r/p} .
	\end{align*}
Note that the relation of $Q$ and $Q'$ is a bijection.
Hence
\begin{align*}
	\left( \sum_{Q  \in \mathcal D}I_{Q,2} ^r \right)^{1/r} 
& \lesssim  \left( \sum_{j=j_Q}^\infty    \sum_{m= \infty}^{ j_Q - 1} 2^{-|m-j| ( \delta -\epsilon)  r} 2^{ (1/p -1/t) (j_Q -m)n r p/A } \right)^{1/r} 	\|g \|_{ \widetilde{ \ell^q ( M_{t,r} ^{p, \omega})   }} 
\\
& \lesssim \|g \|_{ \widetilde{ \ell^q ( M_{t,r} ^{p, \omega})   }}  .
\end{align*}	
	
Next if $0<p \le 1$ and $ p <q \le r <\infty $, still use  	 H\"older's inequality twice and we have
\begin{align*}
		I_Q^r &    \lesssim  \omega (Q)^{r/t-r/p} \left(  \sum_{j=j_Q}^\infty   \sum_{m\in\mathbb Z} 2^{-|m-j| ( \delta -\epsilon /2 )  q}  \left(  \int_Q         |     g_m (x)     |^p
	\omega (x)\d \mu (x)   \right)^{q/p}   \right)^{r/q}     \\
	&    \lesssim  \omega (Q)^{r/t-r/p}  \sum_{j=j_Q}^\infty   \sum_{m\in\mathbb Z} 2^{-|m-j| ( \delta -\epsilon  )  r}  \left(  \int_Q         |     g_m (x)     |^p
	\omega (x)\d \mu (x)   \right)^{r/p}     \\
	& \lesssim \omega (Q ')^{r/t- r/p} \sum_{j=j_Q}^\infty   \sum_{m\in\mathbb Z} 2^{-|m-j| ( \delta -\epsilon  )  r} 2^{ r (1/p -1/t) (j_Q -m)n p/A }  \left(  \int_{Q'}         |     g_m (x)     |^p
	\omega (x)\d \mu (x)   \right)^{r/p}     \\ 
	& \lesssim \omega (Q ')^{r/t- r/p} \sum_{j=j_Q}^\infty   \sum_{m = j_Q}^\infty 2^{-|m-j| ( \delta -\epsilon  )  r} 2^{ r (1/p -1/t) (j_Q -m)n p/A }  \left(  \int_{Q'}         |     g_m (x)     |^p
	\omega (x)\d \mu (x)   \right)^{r/p}     \\ 
	& + \omega (Q ')^{r/t- r/p} \sum_{j=j_Q}^\infty   \sum_{m = -\infty} ^{j_Q -1} 2^{-|m-j| ( \delta -\epsilon  )  r} 2^{ r (1/p -1/t) (j_Q -m)n p/A }  \left(  \int_{Q'}         |     g_m (x)     |^p
	\omega (x)\d \mu (x)   \right)^{r/p} .
\end{align*}
Hence similarly to the above argument, we obtain
\begin{align*}
	\|G \|_{ \widetilde{ \ell^q ( M_{t,r} ^{p, \omega})   }}  & = 	\left( \sum_{Q  \in \mathcal D}	I_Q^r  \right)^{1/r} 
 \lesssim \|g \|_{ \widetilde{ \ell^q ( M_{t,r} ^{p, \omega})   }}  .
\end{align*}
	Next if $0<p \le 1$ and $ p < r <q  \le \infty $, using H\"older's inequality   once and (\ref{ell^d}) , we obtain
	\begin{align*}
		I_Q^r & \lesssim  \omega (Q)^{r/t-r/p} \left(  \sum_{j=j_Q}^\infty   \sum_{m\in\mathbb Z} 2^{-|m-j| ( \delta -\epsilon  )  q}  \left(  \int_Q         |     g_m (x)     |^p
		\omega (x)\d \mu (x)   \right)^{q/p}   \right)^{r/q}     \\
		& \lesssim    \omega (Q)^{r/t-r/p}  \sum_{j=j_Q}^\infty   \sum_{m\in\mathbb Z} 2^{-|m-j| ( \delta -\epsilon  )  r}  \left(  \int_Q         |     g_m (x)     |^p
		\omega (x)\d \mu (x)   \right)^{r/p} .
	\end{align*}
	Repeating the above argument, we get
	\begin{align*}
		\|G \|_{ \widetilde{ \ell^q ( M_{t,r} ^{p, \omega})   }}  & = 	\left( \sum_{Q  \in \mathcal D}	I_Q^r  \right)^{1/r} 
		\lesssim \|g \|_{ \widetilde{ \ell^q ( M_{t,r} ^{p, \omega})   }}  .
	\end{align*}
Hence we prove the subcase  $0< p \le 1 $. 

If $p>1$, using  Minkowski's inequality, we have
\begin{align*}
	I_Q^r \lesssim \omega (Q)^{r/t-r/p} \left(  \sum_{j=j_Q}^\infty \left(  \sum_{m\in\mathbb Z} 2^{-|m-j|\delta}    \left(\int_Q  | g_m (x) |^p \d  \mu (x)   \right)^{1/p}          \right)^{q}   \right)^{r/q}.    
\end{align*}
For the terms like $(\sum \cdots )^{a}$, use H\"older's inequality if $a>1$ and use (\ref{ell^d}) if $0< a\le 1$. We also obtain
	\begin{align*}
	\|G \|_{ \widetilde{ \ell^q ( M_{t,r} ^{p, \omega})   }}  & = 	\left( \sum_{Q  \in \mathcal D}	I_Q^r  \right)^{1/r} 
	\lesssim \|g \|_{ \widetilde{ \ell^q ( M_{t,r} ^{p, \omega})   }}  .
\end{align*}
So we prove the subcase $ 1<p <t <r <\infty $. The proof of case $r<\infty$ is complete. 

	Thus we completes the proof of Theorem  \ref{hardy bourgain}.
\end{proof}

\subsection{Characterization via Peetre maximal function} 
For $\lambda >0$, $j\in\mathbb Z$, $f\in \mathcal S '$ and $\varphi  \in \mathscr S (\mathbb R)$, we define the Peetre-type maximal functions  by
\begin{equation*}
	\varphi_{j, \lambda}^* ( \sqrt{L} ) f(x) = \sup_{y\in X} \frac{| \varphi_j (\sqrt{L} ) f(y) |}{(1+2^j \rho(x,y))^\lambda }, \;\; x\in X,
\end{equation*}
where $\varphi_j (\lambda ) = \varphi(2^{-j} \lambda)$.
Similarly, for $  \alpha ,\lambda >0$, and $f\in \mathcal S ' $,  put
\begin{equation*}
		\varphi_{\lambda}^* ( \alpha \sqrt{L} ) f(x) = \sup_{y\in X} \frac{| \varphi ( \alpha \sqrt{L} ) f(y) |}{(1+ \rho(x,y) /\alpha  )^\lambda }, \;\; x\in X.
\end{equation*}
Note that in the particular case when $\varphi  $ is supported in $(0,\infty )$, these maximal functions can be defined for $f \in \mathcal S _\infty ' $ via (\ref{varphi}).
	
	\begin{lem}[Proposition 2.16, \cite{BBD20}]\label{peetre psi}
		Let $\psi $  and $\varphi $ be a partition of unity. Then for any $\lambda >0$ and $j \in \mathbb Z$, we have 
		\begin{equation*}
			\sup_{\alpha \in [2^{-j-1}  , 2^{-j}]}  \psi_\lambda ^* (  \alpha \sqrt{L} ) f(x) \lesssim  \sum_{k=j-2}^{j+3}  \varphi_{k,\lambda}^* ( \sqrt{L} ) f(x),
		\end{equation*}
		for all $f\in \mathcal S _\infty '$ and $x  \in X$.
	\end{lem}
	
	\begin{lem}[Proposition 2.18, \cite{BBD20}] \label{psi le con}
		Let $\psi$ be a partition of unity and $\varphi \in \mathscr S (\mathbb R)$ be an even function such that $\varphi \neq 0$ on $[1/2, 2]$. Then for any $\lambda>0$, $j \in \mathbb Z$,  $A>0$, and $f\in \mathcal S '$, we have 
		\begin{equation*}
			|\psi_j ( \sqrt{L}) f(x) | \lesssim \bigg(  \int_{2^{-j-2}} ^{ 2^{-j+2} } | \psi_\lambda ^* (\alpha\sqrt{L} ) f(x) |^A  \frac{\d \alpha}{\alpha}  \bigg)^{1/A}.
		\end{equation*}
	\end{lem}
	
	For brevity, we  define
	\begin{align*}
			& \|f\|_{  \dot{\mathcal F}^{s,q,\psi,L}_{p,t,r,\omega} (X) }^{*,\lambda} :=	
		\bigg(    \sum_{Q \in \mathcal D}  \omega (Q)^{r/t-r/p} \Big(  \int_Q  \Big(  \sum_{j=j_Q}^\infty | 2^{js}  \psi_{j,\lambda}^*  ( \sqrt{L} )f (x)|^q \Big)^{p/q}  \omega (x) \d \mu (x)      \Big)^{r/p}                 \bigg)^{1/r} ,  
	\end{align*}
and
\begin{equation*}
	 \|f\|_{  \dot{\mathcal B}^{s,q,\psi,L}_{p,t,r,\omega} (X) }^{*,\lambda} :=	\bigg(    \sum_{Q \in \mathcal D}  \omega (Q)^{r/t-r/p} \Big(  \sum_{j=j_Q}^\infty \Big(   \int_Q | 2^{js}  \psi_{j,\lambda}^ * (   \sqrt{L} )f (x)|^p \omega (x)\d \mu (x)   \Big)^{q/p}   \Big)^{r/q}                 \bigg)^{1/r} .
\end{equation*}

	By Lemma \ref{peetre psi}, we have the following result.

\begin{thm}\label{equivalence}
	 Let  $\omega \in A_{\infty}$. 
	Let $\psi, \varphi $ be partitions of unity. Let  $s\in \mathbb R$, $0<q\le \infty $, and $\lambda >0$.
		Let  $0<p<t< r<\infty$	or let $0<p\le t<r=\infty$.
	Then the following norm equivalence holds: for all $f\in \mathcal S _\infty '$,
	\begin{align*}
		& \|f\|_{  \dot{\mathcal B}^{s,q,\psi,L}_{p,t,r,\omega} (X) }^{*,\lambda} \approx 	  \|f\|_{  \dot{\mathcal B}^{s,q,\varphi,L}_{p,t,r,\omega} (X) }^{*,\lambda} \quad 
		 \operatorname{and} \quad  \|f\|_{  \dot{\mathcal F}^{s,q,\psi,L}_{p,t,r,\omega} (X) }^{*,\lambda}
		\approx   \|f\|_{  \dot{\mathcal F}^{s,q,\varphi,L}_{p,t,r,\omega} (X) }^{*,\lambda}.
	\end{align*}
\end{thm}

\begin{proof}
	We only prove $ \|f\|_{  \dot{\mathcal B}^{s,q,\psi,L}_{p,t,r,\omega} (X) }^{*,\lambda} \approx 	  \|f\|_{  \dot{\mathcal B}^{s,q,\varphi,L}_{p,t,r,\omega} (X) }^{*,\lambda}$ since another case is similar. By symmetry, it suffices to show $ \|f\|_{  \dot{\mathcal B}^{s,q,\psi,L}_{p,t,r,\omega} (X) }^{*,\lambda} \lesssim 	  \|f\|_{  \dot{\mathcal B}^{s,q,\varphi,L}_{p,t,r,\omega} (X) }^{*,\lambda}$.	
	 By  Lemma \ref{peetre psi}, we have
	\begin{align*}
	&	\bigg(    \sum_{Q \in \mathcal D}  \omega (Q)^{r/t-r/p} \Big(  \sum_{j=j_Q}^\infty \Big(   \int_Q | 2^{js}  \psi_{j,\lambda}^ * (   \sqrt{L} )f (x)|^p \omega (x)\d \mu (x)   \Big)^{q/p}   \Big)^{r/q}                 \bigg)^{1/r}  \\
	& \lesssim 	  \sum_{m=-2}^{3} \bigg(    \sum_{Q \in \mathcal D}  \omega (Q)^{r/t-r/p} \Big(  \sum_{j=j_Q}^\infty \Big(   \int_Q | 2^{js}   \varphi_{j+m,\lambda}^* ( \sqrt{L} ) f (x)|^p \omega (x)\d \mu (x)   \Big)^{q/p}   \Big)^{r/q}                 \bigg)^{1/r}  \\
	& =:  \sum_{m=-2}^{3} I_m.
\end{align*}
We only estimate $I_{-2}$ and $I_3$. For $Q \in \D_{-j_Q}$, since $\omega \in A_\infty$, $\omega (Q) \approx \omega (2^2 Q) $ where $2^2 Q \in \D_{-j_Q -2 }$ is the parent of $Q$. Then 
\begin{align*}
	I_{-2} & \lesssim  \bigg(    \sum_{Q \in \mathcal D}  \omega (2^2 Q)^{r/t-r/p} \Big(  \sum_{j=j_Q}^\infty \Big(   \int_{2^2 Q} | 2^{(j-2)s} 2^{2s}  \varphi_{j-2,\lambda}^* ( \sqrt{L} ) f (x)|^p \omega (x)\d \mu (x)   \Big)^{q/p}   \Big)^{r/q}                 \bigg)^{1/r} \\
	& \lesssim  \bigg(    \sum_{Q \in \mathcal D}  \omega (2^2 Q)^{r/t-r/p} \Big(  \sum_{j={ j_{2^2 Q  } }}^\infty \Big(   \int_{2^2 Q} | 2^{(j-2)s} 2^{2s}  \varphi_{j-2,\lambda}^* ( \sqrt{L} ) f (x)|^p \omega (x)\d \mu (x)   \Big)^{q/p}   \Big)^{r/q}                 \bigg)^{1/r} \\
	& \lesssim  \|f\|_{  \dot{\mathcal B}^{s,q,\varphi,L}_{p,t,r,\omega} (X) }^{*,\lambda} .
\end{align*}

As for $ I_3 $, for a cube $Q \in \D _{-j_Q} $, there are at most $c2^{3n}$  cubes $Q_i$ in $ \D _{-j_Q +3}  $ and $\omega (Q ) \approx \omega (Q_i)$.
\begin{align*}
	I_{3} & \lesssim  \bigg(    \sum_{Q \in \mathcal D}  \omega ( Q)^{r/t-r/p} \Big(  \sum_{j=j_Q }^\infty   \Big( \sum_{i=1}^{c2^{3n}}  \int_{ Q_i } | 2^{js}  \varphi_{j+3,\lambda}^* ( \sqrt{L} ) f (x)|^p \omega (x)\d \mu (x)   \Big)^{q/p}   \Big)^{r/q}                 \bigg)^{1/r} \\
	& \lesssim  \sum_{i=1}^{c2^{3n}} \bigg(    \sum_{Q \in \mathcal D}  \omega ( Q_i)^{r/t-r/p} \Big(  \sum_{j=j_Q }^\infty   \Big(  \int_{ Q_i } | 2^{js}  \varphi_{j+3,\lambda}^* ( \sqrt{L} ) f (x)|^p \omega (x)\d \mu (x)   \Big)^{q/p}   \Big)^{r/q}                 \bigg)^{1/r} \\
		& \le \sum_{i=1}^{c2^{3n}} \bigg(    \sum_{Q \in \mathcal D}  \omega ( Q_i)^{r/t-r/p} \Big(  \sum_{j=j_{Q_i} }^\infty   \Big(  \int_{ Q_i } | 2^{js}  \varphi_{j+3,\lambda}^* ( \sqrt{L} ) f (x)|^p \omega (x)\d \mu (x)   \Big)^{q/p}   \Big)^{r/q}                 \bigg)^{1/r} \\
	& \lesssim  \|f\|_{  \dot{\mathcal B}^{s,q,\varphi,L}_{p,t,r,\omega} (X) }^{*,\lambda} .
\end{align*}
Thus the proof is complete.
\end{proof}

To  prove the following Theorem \ref{Peetre},  we need a lemma. Theorem \ref{Peetre} is usually called the characterization via Peetre maximal function.
\begin{lem}[Proposition 2.17, \cite{BBD20}] \label{est peetre psi}
	Let $\psi$ be a partition of unity. Then for any $\lambda ,\alpha >0$ and $A>0$, we have 
	\begin{equation}\label{est peetre psi 1}
		\psi_\lambda ^* (\alpha \sqrt{L} ) f (x)  \lesssim 
		\bigg(  \int_X \frac{1}{V(z,\alpha)}  \frac{ |\psi ( \alpha \sqrt{L} ) f(z) |^A }{ (1+\rho(x,z) / \alpha )^{\lambda A} }  \d \mu (z)  \bigg)^{1/A},
	\end{equation}
	for all $f \in \mathcal S_\infty '$ and $x \in X$.
	\end{lem}

\begin{thm}\label{Peetre}
		 Let  $\omega \in A_\infty$, $0<p<\infty$, $0<q\le \infty$ and  $s\in \mathbb R$.
	Let $\psi$ be a partition of unity. 
		
	{\rm (a)} 	If $0<p<t<\infty$ and $ \max\{t, -nt / \log \beta  \} < r <\infty $, where $\beta = 1-  (1- \alpha_0)^{p/A} /  [\omega]_{A_{p/A}} $ with $0< A<  p / q_\omega$, or $ 0 <p \le t < r =\infty$, then for  $\lambda >  n q_\omega /p$,
	\begin{equation*}
		\|f\|_{  \dot{\mathcal B }^{s,q,\psi,L}_{p,t,r,\omega} (X) }^{*,\lambda}  \approx  	\| f\|_{\dot{\mathcal B }^{s,q,\psi,L}_{p,t,r,\omega} (X)} .
	\end{equation*}
		
			{\rm (b)} 
				If $0<p<t<\infty$ and $ \max\{t, -nt / \log \beta  \} < r <\infty $, where $\beta = 1-  (1- \alpha_0)^{p/A} /  [\omega]_{A_{p/A}} $ with $0< A<  \min \{ q , p/q_\omega  \}$, or $ 0 <p \le t < r =\infty$, then for  $\lambda >  \max \{ n/q , n q_\omega /p\}$,
				\begin{equation*}
					\|f\|_{  \dot{\mathcal F}^{s,q,\psi,L}_{p,t,r,\omega} (X) }^{*,\lambda}     \approx  	\| f\|_{\dot{\mathcal F}^{s,q,\psi,L}_{p,t,r,\omega} (X)} .
			\end{equation*}
	\end{thm}
	
\begin{proof}
	We will only prove (b) since the proof of (a) is similar. Since $  |\psi_j (  \sqrt{L} ) f(x) |\le \psi_{j, \lambda}^* (\sqrt{L}) f (x) $, it suffices to prove that
	\begin{equation} \label{peetre norm less}
		\|f\|_{  \dot{\mathcal F}^{s,q,\psi,L}_{p,t,r,\omega} (X) }^{*,\lambda} \lesssim \|f\|_{  \dot{\mathcal F}^{s,q,\psi,L}_{p,t,r,\omega} (X) }.
	\end{equation}
	Indeed, taking $A < \min\{  p,q, p / q_\omega  \} = \min \{ q , p/q_\omega  \}$ such that $\lambda > n/A$ and $\omega \in A_{p/A}$, then applying (\ref{est peetre psi 1}) in Lemma \ref{est peetre psi}, we have
	\begin{align*}
		\psi_{j, \lambda}^* (\sqrt{L}) f (x) & \lesssim  \bigg(  \int_X \frac{1}{V(z,2^{-j} )}  \frac{ |\psi_j (  \sqrt{L} ) f(z) |^A }{ (1+2^j \rho(x,z)  )^{\lambda A} }  \d \mu (z)  \bigg)^{1/A},  \\
		& \lesssim \mathcal M _A ( |\psi_j (  \sqrt{L} ) f | ) (x),
	\end{align*}
	where we use Lemma \ref{basic est}  in the last inequality. Then using Lemma \ref{max FS Bou}, we obtain (\ref{peetre norm less}). (For the proof of (a), use Theorem \ref{Ma ell q Mptr} instead of Lemma \ref{max FS Bou})
\end{proof}

As a consequence of Theorems \ref{equivalence} and \ref{Peetre}, we obtain that these spaces are independent of the choice of $\psi,\varphi$.

\begin{thm}\label{inde psi}
		 Let  $\omega \in A_\infty$, $0<p<\infty$, $0<q\le \infty$ and  $s\in \mathbb R$.
		Let $\psi$ and $\varphi$ be partitions of unity.  
	
	{\rm (a)} 	If $0<p<t<\infty$ and $ \max\{t, -nt / \log \beta  \} < r <\infty $, where $\beta = 1-  (1- \alpha_0)^{p/A} /  [\omega]_{A_{p/A}} $ with $0< A<  p / q_\omega$, or $ 0 <p \le t < r =\infty$, then the spaces $\dot{\mathcal B}^{s,q,\psi,L}_{p,t,r,\omega} (X)$ and $\dot{\mathcal B}^{s,q,\varphi,L}_{p,t,r,\omega} (X)$ coincide with equivalent norms.
	
	{\rm (b)} 
	If $0<p<t<\infty$ and $ \max\{t, -nt / \log \beta  \} < r <\infty $, where $\beta = 1-  (1- \alpha_0)^{p/A} /  [\omega]_{A_{p/A}} $ with $0< A<  \min \{ q , p/q_\omega  \}$, or $ 0 <p \le t < r =\infty$, then the spaces $\dot{\mathcal F}^{s,q,\psi,L}_{p,t,r,\omega} (X)$ and $\dot{\mathcal F}^{s,q,\varphi,L}_{p,t,r,\omega} (X)$ coincide with equivalent norms.
\end{thm}

\subsection{Characterizations via compact supports}
Let $\psi$ be a partition of unity.
We define
\begin{align*}
	\| f\| _{  \dot{\mathcal F}^{s,q,L}_{p,t,r,\omega} (X)      }^{\bigstar}
		&:= \bigg(    \sum_{Q \in \mathcal D}  \omega (Q)^{r/t-r/p} \Big(  \int_Q  \Big(  \int_{0}^{2^{-j_Q} }| \alpha ^{-s} \psi  (\alpha  \sqrt{L} )f (x)|^q\frac{\d \alpha}{ \alpha} \Big)^{p/q}  \omega (x) \d \mu (x)      \Big)^{r/p}                 \bigg)^{1/r} ,\\
			\| f\| _{  \dot{\mathcal F}^{s,q,L}_{p,t,r,\omega} (X)      }^{\bigstar \lambda}
			& := 	\bigg(    \sum_{Q \in \mathcal D}  \omega (Q)^{r/t-r/p} \Big(  \int_Q  \Big(  \int_{0}^{2^{-j_Q} }| \alpha ^{-s} \psi_\lambda ^*  (\alpha  \sqrt{L} )f (x)|^q\frac{\d \alpha}{ \alpha} \Big)^{p/q}  \omega (x) \d \mu (x)      \Big)^{r/p}                 \bigg)^{1/r} ,\\
				\| f\| _{  \dot{\mathcal B}^{s,q,L}_{p,t,r,\omega} (X)      }^{\bigstar}
				& := \bigg(    \sum_{Q \in \mathcal D}  \omega (Q)^{r/t-r/p} 
				\Big(  \int_{0}^{2^{-j_Q} }  \Big( \int_Q | \alpha ^{-s} \psi  (\alpha  \sqrt{L} )f (x)|^p \omega (x )\d \mu (x) \Big)^{q/p} \frac{\d \alpha}{ \alpha} \Big) ^{r/q}
				          \bigg)^{1/r} ,\\
				          	\| f\| _{  \dot{\mathcal B}^{s,q,L}_{p,t,r,\omega} (X)      }^{\bigstar \lambda}
				          & := \bigg(    \sum_{Q \in \mathcal D}  \omega (Q)^{r/t-r/p} 
				          \Big(  \int_{0}^{2^{-j_Q} }  \Big( \int_Q | \alpha ^{-s} \psi _\lambda ^* (\alpha  \sqrt{L} )f (x)|^p \omega (x )\d \mu (x) \Big)^{q/p} \frac{\d \alpha}{ \alpha} \Big) ^{r/q}
				          \bigg)^{1/r} .
\end{align*}

\begin{thm}\label{char conti}
		 Let  $\omega \in A_\infty$, $0<p<\infty$, $0<q\le \infty$ and  $s\in \mathbb R$.
	
	{\rm (a)} 	If $0<p<t<\infty$ and $ \max\{t, -nt / \log \beta  \} < r <\infty $, where $\beta = 1-  (1- \alpha_0)^{p/A} /  [\omega]_{A_{p/A}} $ with $0< A<  p / q_\omega$, or $ 0 <p \le t < r =\infty$, 	then for $\lambda > n q_\omega / p$ and $f\in \mathcal S_\infty '$,
		\begin{align*}
		\| f\| _{  \dot{\mathcal B}^{s,q,L}_{p,t,r,\omega} (X)  } & \approx 	\| f\| _{  \dot{\mathcal B}^{s,q,L}_{p,t,r,\omega} (X)      }^{\bigstar}
		\approx
		\| f\| _{  \dot{\mathcal B}^{s,q,L}_{p,t,r,\omega} (X)      }^{\bigstar \lambda}.
	\end{align*}

	{\rm (b)} 
	If $0<p<t<\infty$ and $ \max\{t, -nt / \log \beta  \} < r <\infty $, where $\beta = 1-  (1- \alpha_0)^{p/A} /  [\omega]_{A_{p/A}} $ with $0< A<  \min \{ q , p/q_\omega  \}$, or $ 0 <p \le t < r =\infty$, then for  $\lambda >\max \{ n/q, n q_\omega / p  \} $ and 
	$f\in \mathcal S_\infty '$,
	\begin{align*}
		\| f \| _{  \dot{\mathcal F}^{s,q,L}_{p,t,r,\omega} (X)  } & \approx 	\| f\| _{  \dot{\mathcal F}^{s,q,L}_{p,t,r,\omega} (X)      }^{\bigstar} \approx
		\| f\| _{  \dot{\mathcal F}^{s,q,L}_{p,t,r,\omega} (X)      }^{\bigstar \lambda}.
	\end{align*}
\end{thm}

\begin{proof}
	We only prove (b) since the proof of (a) is similar. We divide the proof of (b) into three steps.
	
	Step 1:  We first prove that 
	\begin{equation}\label{eq con le norm}
		\| f\| _{  \dot{\mathcal F}^{s,q,L}_{p,t,r,\omega} (X)      }^{\bigstar}
		  \lesssim 	\| f\| _{  \dot{\mathcal F}^{s,q,L}_{p,t,r,\omega} (X)  } .
	\end{equation}
	Indeed, for $\alpha \in [2^{-j-1}, 2^{-j}]$ with $j\in \mathbb Z$, from Lemma \ref{peetre psi}, we have 
	\begin{equation}\label{eq psi point}
		\sup _{ \alpha \in [2^{-j-1}, 2^{-j}] } |\psi ( \alpha \sqrt{L}) f (x) | \lesssim \sum_{k=j-2 } ^{j+3} \psi_{k,\lambda}^* ( \sqrt{L} ) f(x).
	\end{equation}
	Hence, (\ref{eq con le norm}) follows from (\ref{eq psi point}) and Theorem \ref{Peetre}.
	
	Step 2: We next show that 
	\begin{equation}\label{eq con le norm 2}
			\| f\| _{  \dot{\mathcal F}^{s,q,L}_{p,t,r,\omega} (X)  }  \lesssim  
				\| f\| _{  \dot{\mathcal F}^{s,q,L}_{p,t,r,\omega} (X)      }^{\bigstar \lambda}.
	\end{equation}
	By Lemma \ref{psi le con}, we have
	\begin{equation}
		|\psi_j ( \sqrt{L}) f(x) | \lesssim \bigg(  \int_{2^{-j-2}} ^{ 2^{-j+2} } | \psi_\lambda ^* (\alpha\sqrt{L} ) f(x) |^q  \frac{\d \alpha}{\alpha}  \bigg)^{1/q}.
	\end{equation}
	This implies (\ref{eq con le norm 2}).
	
	Step 3: Finally, we prove that 
	\begin{align}\label{eq con le norm 3}
		\| f\| _{  \dot{\mathcal F}^{s,q,L}_{p,t,r,\omega} (X)      }^{\bigstar \lambda} \lesssim 	\| f\| _{  \dot{\mathcal F}^{s,q,L}_{p,t,r,\omega} (X)      }^{\bigstar}.
	\end{align}
	 By Lemma \ref{est peetre psi}, we have, for all $ \alpha  \in [1,2]$, 
	\begin{equation*}
		|\psi_\lambda ^* (  2^{-j} \alpha \sqrt{L} ) f (x) |^A \lesssim \int _X \frac{1}{V (x,2^{-j}) }  \frac{|\psi (2^{-j} \alpha  \sqrt{L}  ) f(z)|^A }{ (1+ 2^j \rho(x,z)) ^{\lambda A} }  \d \mu (z).
	\end{equation*}
	Since $A < q$, use Minkowski's inequality to obtain
	\begin{align*}		
		& \Bigg(  \int_1^2  | \psi_\lambda^* ( 2^{-j} \alpha \sqrt{L} f (x) ) |^q \frac{d \alpha }{\alpha }   \Bigg)^{A / q} 
		\lesssim  \int _X  \frac{1}{V (x,2^{-j}) }  \frac{ \Big( \int_1^2  |\psi (2^{-j} \alpha  \sqrt{L}  ) f(z)|^q \frac{\d \alpha }{ \alpha }  \Big)^{A/q}   }{ (1+ 2^j \rho(x,z)) ^{\lambda A} }  \d \mu (z).
	\end{align*}
	By a change of variables,
		\begin{align*}		
			\nonumber
		& \Bigg(  \int_{2^{-j}}^{2^{-j+1}} ( \alpha ^{-s}  | \psi_\lambda^* (  \alpha \sqrt{L} f (x) ) |  )^q \frac{\d \alpha }{\alpha }   \Bigg)^{A / q} \\
		& \lesssim  \int _X  \frac{1}{V (x,2^{-j}) }  \frac{ \Big( \int_{2^{-j}}^{2^ { -j+1 } } ( \alpha ^{-s}  |\psi ( \alpha  \sqrt{L}  ) f(z)|^q \frac{\d \alpha }{ \alpha }  \Big)^{A/q}   }{ (1+ 2^j \rho(x,z)) ^{\lambda A} }  \d \mu (z).
	\end{align*}
	Hence, applying Lemma \ref{basic est}, we obtain
	\begin{equation*}
		\Bigg(  \int_{2^{-j}}^{2^{-j+1}} ( \alpha ^{-s}  | \psi_\lambda^* (  \alpha \sqrt{L} f (x) ) |  )^q \frac{d \alpha }{\alpha }   \Bigg)^{1 / q}  \lesssim \mathcal M _A \bigg(  \bigg( \int_{2^{-j}}^{2^{ -j+1} } ( \alpha ^{-s}  |\psi ( \alpha  \sqrt{L}  ) f|^q \frac{\d \alpha}{\alpha} \bigg) ^{1/q} \bigg) (x)
	\end{equation*}
	since $\lambda > n/A$.
	Using Lemma \ref{max FS Bou}, we get the required estimate (\ref{eq con le norm 3}).
	Thus the proof of Theorem  \ref{char conti} is complete.
\end{proof}

\subsection{Continuous characterizations by Schwartz functions}
For each $m\in \mathbb N$, denote by $ \mathscr S _m (\mathbb R)$ the set of all even functions $\varphi \in \mathscr S (\mathbb R) $ such that $\varphi (\xi) = \xi ^{2m}  \phi (\xi) $ for some $\phi \in \mathscr S (\mathbb R)$ and $ \varphi (\xi ) \neq 0$ on $ (1/2, 2).$ 
Now, let $\varphi \in \mathscr S _m (\mathbb R) $ and denote
\begin{align*}
	\| f\| _{  \dot{\mathcal F}^{s,q, \varphi, L}_{p,t,r,\omega} (X)      }^{\bigstar}
	&:= \bigg(    \sum_{Q \in \mathcal D}  \omega (Q)^{r/t-r/p} \Big(  \int_Q  \Big(  \int_{0}^{2^{-j_Q} }| \alpha ^{-s} \varphi  (\alpha  \sqrt{L} )f (x)|^q\frac{\d \alpha}{ \alpha} \Big)^{p/q}  \omega (x) \d \mu (x)      \Big)^{r/p}                 \bigg)^{1/r} ,\\
	\| f\| _{  \dot{\mathcal F}^{s,q, \varphi, L}_{p,t,r,\omega} (X)     }^{\bigstar \lambda}
	& := 	\bigg(    \sum_{Q \in \mathcal D}  \omega (Q)^{r/t-r/p} \Big(  \int_Q  \Big(  \int_{0}^{2^{-j_Q} }| \alpha ^{-s} \varphi_\lambda ^*  (\alpha  \sqrt{L} )f (x)|^q\frac{\d \alpha}{ \alpha} \Big)^{p/q}  \omega (x) \d \mu (x)      \Big)^{r/p}                 \bigg)^{1/r} ,\\
	\| f\| _{  \dot{\mathcal B}^{s,q,\varphi, L}_{p,t,r,\omega} (X)      }^{\bigstar}
	& := \bigg(    \sum_{Q \in \mathcal D}  \omega (Q)^{r/t-r/p} 
	\Big(  \int_{0}^{2^{-j_Q} }  \Big( \int_Q | \alpha ^{-s} \varphi  (\alpha  \sqrt{L} )f (x)|^p \omega (x )\d \mu (x)  \Big)^{q/p} \frac{\d \alpha}{ \alpha} \Big) ^{r/q}
	\bigg)^{1/r} ,\\
	\| f\| _{  \dot{\mathcal B}^{s,q,\varphi, L}_{p,t,r,\omega} (X)      }^{\bigstar \lambda}
	& := \bigg(    \sum_{Q \in \mathcal D}  \omega (Q)^{r/t-r/p} 
	\Big(  \int_{0}^{2^{-j_Q} }  \Big( \int_Q | \alpha ^{-s} \varphi _\lambda ^* (\alpha  \sqrt{L} )f (x)|^p \omega (x )\d \mu (x)  \Big)^{q/p} \frac{\d \alpha}{ \alpha} \Big) ^{r/q}
	\bigg)^{1/r} .
\end{align*}

\begin{thm}\label{char con Sm}
			 Let  $\omega \in A_\infty$, $0<p<\infty$, $0<q\le \infty$ and  $s\in \mathbb R$.
	
	{\rm (a)} 	Let $0<p<t<\infty$ and $ \max\{t, -nt / \log \beta  \} < r <\infty $, where $\beta = 1-  (1- \alpha_0)^{p/A} /  [\omega]_{A_{p/A}} $ with $0< A<  p / q_\omega$, or $ 0 <p \le t < r =\infty$. 
		Let $m\in \mathbb N $,  $  2m  >  (1/p-1/t)n q_\omega +s$. Let $\varphi \in \mathscr S _m (\mathbb R) $. Then
	 for $\lambda > n q_\omega / p$ and $f\in \mathcal S '$,
there exists $g \in \mathcal P$ such that
\begin{align} \label{eq Sm 1}
	\| f-g\| _{  \dot{\mathcal B}^{s,q,\varphi, L}_{p,t,r,\omega} (X)      }^{\bigstar \lambda}
	\lesssim
	\| f\| _{  \dot{\mathcal B}^{s,q,L}_{p,t,r,\omega} (X)  }
	\lesssim  		\| f\| _{  \dot{\mathcal B}^{s,q,\varphi, L}_{p,t,r,\omega} (X)      }^{\bigstar}.
\end{align}
	
	{\rm (b)} 
	If $0<p<t<\infty$ and $ \max\{t, -nt / \log \beta  \} < r <\infty $, where $\beta = 1-  (1- \alpha_0)^{p/A} /  [\omega]_{A_{p/A}} $ with $0< A<  \min \{ q , p/q_\omega  \}$, or $ 0 <p \le t < r =\infty$.
		Let $m\in \mathbb N $,  $  2m  >  (1/p-1/t)n q_\omega +s$. Let $\varphi \in \mathscr S _m (\mathbb R) $. Then for
	  $\lambda >\max \{  n/q, n q_\omega /p  \} $ and $f \in \mathcal S'$, there exists $g \in \mathcal P$ such that
	 \begin{align}\label{eq Sm 2}
	 	\| f-g\| _{  \dot{\mathcal F}^{s,q, \varphi, L}_{p,t,r,\omega} (X)      }^{\bigstar \lambda}
	 	\lesssim  	\| f\| _{  \dot{\mathcal F}^{s,q,L}_{p,t,r,\omega} (X)  }
	 	\lesssim  			\| f\| _{  \dot{\mathcal F}^{s,q, \varphi, L}_{p,t,r,\omega} (X)     }^{\bigstar}.
	 \end{align}
\end{thm}

\begin{rem}
	The proof of Theorem \ref{char con Sm} follows the ideas in \cite[Theorem 3.6]{BBD20}.
	Comparing with the proof of Theorem \ref{char conti}, the difficulty of the  proof of Theorem \ref{char con Sm} is lack of compact support condition for the functions in $\mathscr S _m (\mathbb R)$. Another difference with Theorem \ref{char conti} is that the results are formulated for $f \in \mathcal S ' $. The reason is that, as $K_ {\varphi (\alpha  \sqrt{L}) }  (x, \cdot ) $ may not be in $\mathcal S_\infty$ for $\varphi \in \mathscr{S} _m (\mathbb R)$, $\varphi (\alpha \sqrt{L}) f $ may not be defined when $f\in \mathcal S_\infty '.$ Although each $f \in \mathcal S _\infty '$ has an extension to an element in $\mathcal S'$, the extension in not unique; that is, $\varphi ( \alpha \sqrt{L} ) f$ will depend on the chosen representative of $f$. Theorems \ref{char con Sm} says that there exists a representative such that the left-hand-side inequality in (\ref{eq Sm 1}), (\ref{eq Sm 2}) holds.
\end{rem}

\begin{proof}[Proof of Theorem \ref{char con Sm}.]
	We only prove (\ref{eq Sm 2}) since the proof of (\ref{eq Sm 1}) is similar. We divide the proof into three steps.
	
	Step 1: Let $\psi$ be a partition of unity. By Lemma \ref{iden}, there exists $g\in \mathcal P$ such that 
	\begin{equation*}
		f - g = c_\psi \int_0^\infty \psi (\alpha \sqrt{L} ) f \frac{\d \alpha}{\alpha} \;\; \mathrm{in} \;\mathcal{S} '.
	\end{equation*}
	We will show that
	\begin{equation} \label{eq Sm step 1}
		\| f-g\| _{  \dot{\mathcal F}^{s,q, \varphi, L}_{p,t,r,\omega} (X)      }^{\bigstar \lambda}
		\lesssim  	\| f\| _{  \dot{\mathcal F}^{s,q,L}_{p,t,r,\omega} (X)  } .
	\end{equation}
	Let $\lambda >0$ , $\alpha \in [2^{-\nu -1}  , 2^{-\nu} ]$ for some $\nu \in \mathbb Z$.
	From   \cite[(55)]{BBD20}, we have the following estimate:
	\begin{align*}
		\nonumber
		| \varphi_{\lambda}^* (\alpha \sqrt{L}) (f-g) |  
		& \lesssim \sum_{j \ge \nu -1}  2^{ -  (2M- \lambda) (j-\nu) } \psi_{j,\lambda}^* (\sqrt{L}) f + \sum_{j <  \nu +3}  2^{ -  2m (\nu-j ) } \psi_{j,\lambda}^* (\sqrt{L}) f  \\
		& \lesssim  \sum_{j\in \mathbb Z} 2^{-2m   |\nu -j| } \psi_{j, \lambda}^* ( \sqrt{L}) f.
	\end{align*} 
	for all $\alpha \in [2^{-\nu -1}  , 2^{-\nu} ] $ and $M > m+\lambda /2$.
	Therefore,
		\begin{align} \label{hardy v j}
		\nonumber
	& 2^{\nu s}	| \varphi_{\lambda}^* (\alpha \sqrt{L}) (f-g) |  \\
	\nonumber
		& \lesssim \sum_{j \ge \nu -1}  2^{ -  (2M- \lambda) (j-\nu) } 2^{ -s (j -\nu )} 2^{js} \psi_{j,\lambda}^* (\sqrt{L}) f    + \sum_{j <  \nu +3}  2^{ -  2m (\nu-j ) }  2^{ s (\nu -j )} 2^{js} \psi_{j,\lambda}^* (\sqrt{L}) f  \\
		\nonumber
		& = \sum_{j \ge \nu -1}  2^{ -  (2M- \lambda + s) (j-\nu) }   2^{js} \psi_{j,\lambda}^* (\sqrt{L}) f    + \sum_{j <  \nu +3}  2^{ - ( 2m -s ) (\nu-j ) }   2^{js} \psi_{j,\lambda}^* (\sqrt{L}) f  \\
		& \lesssim  \sum_{j\in \mathbb Z} 2^{- (2m -s )   |\nu -j| } 2^{js}  \psi_{j, \lambda}^* ( \sqrt{L}) f.
	\end{align} 
	if we provide that $ M > (  2m - 2s +\lambda  )/2 $.  
	Hence,
	\begin{equation*}
			\int_{2^{-\nu -1}}  ^{ 2 ^{-\nu} } \Big(  \alpha ^{-s} |   \varphi_{\lambda}^* (\alpha \sqrt{L}) (f-g) |   \Big) ^q \frac{\d \alpha}{\alpha} \lesssim 
			\Big(     \sum_{j\in \mathbb Z} 2^{- ( 2m-s)   |\nu -j|  } 2^{js} \psi_{j, \lambda}^* ( \sqrt{L}) f    \Big)^q .
	\end{equation*}
		Therefore,
	\begin{align*}
		\nonumber
		&	\| f-g\| _{  \dot{\mathcal F}^{s,q, \varphi, L}_{p,t,r,\omega} (X)      }^{\bigstar \lambda} \\
		& \lesssim \bigg(    \sum_{Q \in \mathcal D}  \omega (Q)^{r/t-r/p} \Big(  \int_Q  \Big(   \sum_{\nu = j_Q}^\infty  \Big(     \sum_{j\in \mathbb Z} 2^{-  (2m-s)   |\nu -j|  }   2^{js} \psi_{j, \lambda}^* ( \sqrt{L}) f (x)     \Big)^q 
		\Big)^{p/q}  \omega (x) \d \mu (x)      \Big)^{r/p}                 \bigg)^{1/r}.
	\end{align*}
	Hence, if $(2m-s)  >  (1/p -1/t) n q_\omega $, by Theorem \ref{hardy bourgain},  we have
	\begin{equation*}
			\| f-g\| _{  \dot{\mathcal F}^{s,q, \varphi, L}_{p,t,r,\omega} (X)      }^{\bigstar \lambda}
		\lesssim 	\| f\| _{  \dot{\mathcal F}^{s,q,L}_{p,t,r,\omega} (X)      }^{\bigstar \lambda}.
	\end{equation*}
	Since $ \lambda >\max \{  n/q, n q_\omega /p  \} $, by  Theorem \ref{Peetre}, we have
	\begin{equation*}
			\| f\| _{  \dot{\mathcal F}^{s,q,L}_{p,t,r,\omega} (X)      }^{\bigstar \lambda}  \approx	\| f\| _{  \dot{\mathcal F}^{s,q,L}_{p,t,r,\omega} (X)  } .
	\end{equation*} 
Hence, we get (\ref{eq Sm step 1}). 
	
	Step 2: We will show that 
	\begin{equation*}
		\| f\| _{  \dot{\mathcal F}^{s,q,L}_{p,t,r,\omega} (X)  } 
	\lesssim  			\| f\| _{  \dot{\mathcal F}^{s,q, \varphi, L}_{p,t,r,\omega} (X)     }^{\bigstar \lambda}.
	\end{equation*}
	Let $\psi$ be a partition of unity. By Lemma \ref{psi le con}, we have
	\begin{equation*}
		| \psi_j ( \sqrt{L} ) f (x)  |^q \lesssim \int_{2^{-j-2}}^{2^{-j+2}}  | \varphi _ \lambda ^*  (\alpha \sqrt{L})  f(x) |^q \frac{\d \alpha} {\alpha}.
	\end{equation*}	
	Hence, the desired inequality follows.
	
	Step 3: We will prove that 
	\begin{equation}\label{eq Sm step 3}
		\| f\| _{  \dot{\mathcal F}^{s,q, \varphi, L}_{p,t,r,\omega} (X)     }^{\bigstar \lambda} \lesssim 
		\| f\| _{  \dot{\mathcal F}^{s,q, \varphi, L}_{p,t,r,\omega} (X)     }^{\bigstar }.
	\end{equation}
	Let 
	\begin{equation*}
		F_k = \Big(   \int_{2^{-k}} ^{ 2^{-k+1}} ( \alpha ^{-s}  |\varphi(\alpha \sqrt{L})  f(z) |  )^q  \frac{\d \alpha}{\alpha}   \Big)^{1/q}.
	\end{equation*}
	Note that   $\max\{  n/p, n/q, nq_\omega /p \} $  $ =\max\{  n/q, nq_\omega /p \} <n/A <\lambda $. 
	From the  inequality after \cite[(60)]{BBD20}, we have 
	\begin{equation*}
		\bigg(  \int_{2^{-j}} ^{ 2^{-j+1} } ( \alpha^{-s}  |  \varphi_\lambda^* ( \alpha \sqrt{L} ) f (x) |  )^q  \frac{\d \alpha}{ \alpha}    \bigg)^{A/q} \lesssim  \bigg(  \sum_{k \ge j}  2^{   (j-k)(B-\lambda+s)A  }  \mathcal M _A ( F_k )  (x)  ^q  \bigg)^{A/q}.
	\end{equation*}
	Let $B>0$ sufficiently large ($(B-\lambda+s)A >  (1/p -1/t) n q_\omega  $) such that  we can use Theorem \ref{hardy bourgain}.
	Hence, by Theorem   \ref{hardy bourgain} and Lemma \ref{max FS Bou}, we obtain (\ref{eq Sm step 3}).
	
	Thus we complete the proof of Theorem \ref{char con Sm}.
\end{proof}

\begin{rem}\label{rem Sm}
	\begin{itemize}
		\item[\rm (a)] If $r=\infty$, $  p=t$, then $ M_{p,\omega}^{t,r} =L^{p}(\omega) $. In this case,  Theorem \ref{char con Sm} can be found in \cite[Theorem 3.6]{BBD20}.
		\item[\rm (b)] Case $r=\infty$, $  p=t$, and $L = - \Delta$, the Laplacian on $\rn$, $\mathcal P$ is the set of all polynomials. In this case, Theorem \ref{char con Sm} are in line with  \cite[Theorems 3.1, 4.1]{BPT96}.
		\item[\rm (c)] The presence of the polynomials $g$ in (\ref{eq Sm 1}) and (\ref{eq Sm 2}) can be omitted if $ f\in L^2$. The details can be seen in \cite[Remark 3.7]{BBD20}.
	\end{itemize}
\end{rem}

Let $ \Psi _{m,\alpha}  (L) = (\alpha ^2 L)^m e^{- \alpha^2 L} $  for $\alpha >0$ and $m\in\mathbb N$. For $\lambda >0$, $f\in \mathcal S '$,  we define 
\begin{equation*}
	\Psi _{m,\alpha, \lambda}^*  (L) f(x) = \sup_{y\in X} \frac{ | \Psi _{m,\alpha}  (L)  f(y) | }{ (1+\rho(x,y)/\alpha )^\lambda } .
\end{equation*}
Applying Theorem \ref{char con Sm} and Remark  \ref{rem Sm} for $\varphi (\xi) = \xi^{2m} e^{-\xi^2}$, we have the following heat kernel characterizations for the  Bourgain-Morrey-Besov type spaces and Triebel-Lizorkin type spaces.

\begin{cor} \label{cor Psi}
		 Let  $\omega \in A_\infty$, $0<p<\infty$, $0<q\le \infty$ and  $s\in \mathbb R$.
	
	{\rm (a)} 	Let $0<p<t<\infty$ and $ \max\{t, -nt / \log \beta  \} < r <\infty $, where $\beta = 1-  (1- \alpha_0)^{p/A} /  [\omega]_{A_{p/A}} $ with $0< A<  p / q_\omega$, or $ 0 <p \le t < r =\infty$. 
	Let $m\in \mathbb N $,  $  2m  >  (1/p-1/t)n q_\omega +s$.  Then
	for $\lambda > n q_\omega / p$ and $f\in \mathcal S '$,
	there exists $g \in \mathcal P$ such that
\begin{align} \label{eq heat 1}
	\nonumber
	&  \bigg(    \sum_{Q \in \mathcal D}  \omega (Q)^{r/t-r/p} 
	\Big(  \int_{0}^{2^{-j_Q} }  \Big( \int_Q | \alpha ^{-s}  \Psi_{m,\alpha,\lambda}^* (L) (f-g)  (x)|^p \omega (x ) \d \mu (x) \Big)^{q/p} \frac{\d \alpha}{ \alpha} \Big) ^{r/q}
	\bigg)^{1/r}  \\
	\nonumber
&	\lesssim
	\| f\| _{  \dot{\mathcal B}^{s,q,L}_{p,t,r,\omega} (X)  } \\
	&	\lesssim   \bigg(    \sum_{Q \in \mathcal D}  \omega (Q)^{r/t-r/p} 
		\Big(  \int_{0}^{2^{-j_Q} }  \Big( \int_Q | \alpha ^{-s}  \Psi_{m,\alpha} (L)f (x)|^p \omega (x ) \d \mu (x) \Big)^{q/p} \frac{\d \alpha}{ \alpha} \Big) ^{r/q}
		\bigg)^{1/r} 	.
\end{align}
	
	{\rm (b)} 
	If $0<p<t<\infty$ and $ \max\{t, -nt / \log \beta  \} < r <\infty $, where $\beta = 1-  (1- \alpha_0)^{p/A} /  [\omega]_{A_{p/A}} $ with $0< A<  \min \{ q , p/q_\omega  \}$, or $ 0 <p \le t < r =\infty$.
	Let $m\in \mathbb N $,  $  2m  >  (1/p-1/t)n q_\omega +s$. Then for
	$\lambda >\max \{  n/q, n q_\omega /p  \} $ and $f \in \mathcal S'$, there exists $g \in \mathcal P$ such that
\begin{align} \label{eq heat 2}
	\nonumber
	& \bigg(    \sum_{Q \in \mathcal D}  \omega (Q)^{r/t-r/p} \Big(  \int_Q  \Big(  \int_{0}^{2^{-j_Q} }| \alpha ^{-s}  \Psi_{m,\alpha,\lambda}^* (L) (f-g)  (x)|^q\frac{\d \alpha}{ \alpha} \Big)^{p/q}  \omega (x) \d \mu (x)      \Big)^{r/p}                 \bigg)^{1/r}  \\
	\nonumber 
&	\lesssim  	\| f\| _{  \dot{\mathcal F}^{s,q,L}_{p,t,r,\omega} (X)  } \\
&	\lesssim  	
	 \bigg(    \sum_{Q \in \mathcal D}  \omega (Q)^{r/t-r/p} \Big(  \int_Q  \Big(  \int_{0}^{2^{-j_Q} }| \alpha ^{-s}  \Psi_{m,\alpha} (L)f (x)|^q\frac{\d \alpha}{ \alpha} \Big)^{p/q}  \omega (x) \d \mu (x)      \Big)^{r/p}                 \bigg)^{1/r} 	.
\end{align}

Moreover, if $f\in L^2 $, $g$ can be removed in (\ref{eq heat 1}) and (\ref{eq heat 2}).

\end{cor}

\begin{rem}
	Let $r=\infty$, $ p =t$.   (\ref{eq heat 1}) and (\ref{eq heat 2}) were proved in  \cite{BBD20}.
	Let $r=\infty$, $ p =t$ again.  In \cite{GKKP17}, \cite{KP15}, the authors proved (\ref{eq heat 1}) for $s\in \mathbb R$, $1\le p \le \infty$, $0<q \le \infty$ and (\ref{eq heat 2}) for $s\in \mathbb R$, $1 < p < \infty$, $0 < q \le \infty$ for the inhomogeneous and homogeneous Besov and Triebel-Lizorkin spaces under additional conditions of H\"older continuity and Markov property of the heat kernel $p_\alpha (x,y)$. Moreover, their results are formulated for distributions in $\mathcal S_\infty '$. 
	In \cite{BX25}, Bai and Xu obtained  (\ref{eq heat 1}) and (\ref{eq heat 2}) for the  Bourgain-Morrey-Besov  spaces and Triebel-Lizorkin  spaces.

	The nondegeneracy condition of the function $\varphi$ in the definition of the class $ \mathcal S _m (\mathbb R)$, $\varphi (\xi) \neq 0$ on $ (-2, -1/2) \cup (1/2, 2) $, can be weakened to $ \varphi(\lambda) \neq 0 $ for some $\lambda >0$. Then Theorem \ref{char con Sm} holds under this weaker condition on $\varphi $. The reasons can be founded in \cite[Remark 3.10]{BBD20}.
\end{rem}

\subsection{Characterizations for weighted Bourgain-Morrey Triebel-Lizorkin type spaces  via Lusin functions and  Littlewood-Paley functions}
For $s\in \mathbb R$, $\lambda , a >0$, $ 0< q <\infty$,  each cube $Q \in \mathcal D$, each measurable function $F$, we define the Littlewood-Paley function  associated with $Q$ by
\begin{equation*}
	\mathcal G _{\lambda ,q }^{s ,Q} F (x) = \bigg(    \int_0^{2^{-j_Q}}  \int _X ( \alpha ^{-s} |F(y,\alpha)| )^q \Big(1 + \frac{ \rho(x,y)}{ \alpha}\Big) ^{-\lambda q}  \frac{\d \mu (y) \d \alpha }{\alpha V(x, \alpha)}     \bigg)^{1/q},
\end{equation*}
and the Lusin function associated with $Q$  by
\begin{equation*} 
		\mathcal S _{a, q }^{s,Q} F (x) = \bigg(    \int_0^{ 2^{-j_Q} } \int _{B(x, a \alpha)} ( \alpha ^{-s} |F(y,\alpha)| )^q   \frac{\d \mu (y) \d \alpha }{\alpha V(x, \alpha)}     \bigg)^{1/q}.
\end{equation*}
When either $s=0$ or $a=1$ , we will drop them in the notation of  $\mathcal G _{\lambda ,q }^{s }$ and $ \mathcal S _{a, q }^{s}$.

\begin{thm} \label{Lusin change area}
	Let $a >1$, $\omega \in A_u$, $ 1\le u <\infty$, $0<q \le p <\infty$,  and $s \in \mathbb R$  . 
	Let    $0<p<t<r<\infty$ 	or $0<p\le t<r=\infty$. 
	Let  $	\gamma : = 
	( \tilde{n} +n u ) / q + nu ( 1/p-1/t) + n/r$.
	Then there exists a positive constant $C$ such that 
	\begin{align*}
	\nonumber
	&  \bigg(    \sum_{Q \in \mathcal D}  \omega (Q)^{r/t-r/p} \Big(  \int_Q  \Big( 	\mathcal S _{a, q }^{s,Q} F (x)
	\Big)^{p}  \omega (x) \d \mu (x)      \Big)^{r/p}                 \bigg)^{1/r} \\
	& \le C  a ^ \gamma
	\bigg(    \sum_{ Q \in \mathcal D}  \omega (Q)^{r/t-r/p} \Big(  \int_{Q}  \Big( 	\mathcal S _{1, q }^{s,Q} F (x)
	\Big)^{p}  \omega (x) \d \mu (x)      \Big)^{r/p}                 \bigg)^{1/r}.
\end{align*}
\end{thm}

\begin{rem}
	When $p=t, r =\infty , q \le p <\infty$, Theorem  \ref{Lusin change area}  is same as \cite[Proposition 3.11]{BBD20}. However, we read \cite{AS77} and find that the methods in \cite{AS77} are not suitable for Bourgain-Morrey Triebel-Lizorkin type spaces. Hence, we only get the case $q\le p$.
\end{rem}

\begin{proof}[Proof of Theorem  \ref{Lusin change area}]
	It suffices to prove Theorem  \ref{Lusin change area} for $s = 0$ and $q=2$; since in the general case of $s$ and $q$, we set $\tilde{F} (y, \alpha ) = ( \alpha ^{-s}  | F (y,t) |) ^{q/2}  $ and then apply the result for the case $s = 0$ and $q=2$, we will get the desired estimate.
	
	 Let $b$ be a real number and  $[b] $ be the integer part of $b$.
	
	Suppose that $s = 0$, $q =2$ and $ p\ge 2$.   We  adapt the proof of \cite[Proposition 3.11]{BBD20}.  For a positive function $g \in L ^ { ( p/2 )' } _ \omega  (Q)$,  by using (\ref{V x r V y r}), we have 
	\begin{align*}
		\nonumber
		\langle ( \mathcal S _{a,2} ^ {0,Q} F  )^2, g \rangle _\omega  
		\nonumber
		& := \int_Q \int_0  ^{ 2 ^{ -j_Q  }	}	 \int_{\rho( x,y) < a \alpha } | F (y, \alpha) |^2   \frac{\d \mu (y) \d \alpha }{\alpha V(x, \alpha)}  g (x) \omega (x) \d \mu (x) \\
			\nonumber
		& \lesssim a^ { \tilde{n} }  \int_Q \int_0  ^{ 2 ^{ -j_Q  }	}	 \int_{\rho( x,y) < a \alpha } | F (y, \alpha) |^2   \frac{\d \mu (y) \d \alpha }{\alpha V(y, \alpha)}  g (x) \omega (x) \d \mu (x) \\
		 & \lesssim  a^ { \tilde{n} }  \int_{Q'} \int_0  ^{ 2 ^{ -j_Q  }	}	 | F (y, \alpha) |^2    M_{a \alpha,\omega} g (y)  \omega ( B (y,a \alpha ) )\frac{\d \mu (y) \d \alpha }{\alpha V(y, \alpha)} ,
	\end{align*}
	where 
	\begin{equation*}
		M_{a \alpha,\omega} g (y) = \frac{1 } { \omega (B(y,a \alpha)  ) }  \int_{B(y ,  a \alpha)} g (x) \omega (x) \d \mu (x),
	\end{equation*}
	and $Q' $ is the cube with the same center with $Q$ and with the side length $2^{-j_Q + 1 + [\log a]}$. 
	Note that $Q'$ can be covered by finite cubes in  $\mathcal D _{ j_Q -1 - [\log a] }$.
 (Since in the following, we will sum up all cube $Q \in \mathcal D$, we may think that $Q' $ is a dyadic cube in $\mathcal D _{ j_Q -1 - [\log a] }$.) 
	
	For simplicity of writing (involving some constants), suppose that $\rho$ is a metric.
	Observe that 
		\begin{equation*}
	\chi_{B(y, a \alpha)} (x) \le \frac{1}{ \omega( B(y,\alpha) ) } \int_{ B( y, \alpha) }  \chi_{  B(x, 2  a \alpha) } (z) \omega (z)  \d \mu (z).
	\end{equation*}
	It follows that 
	\begin{align*}
		\nonumber
		 & \frac{1 } { \omega (B(y, a \alpha)  ) }  \int_{B(y ,  a \alpha)}  \chi_{ B (y ,a \alpha)  } (x)  g (x) \omega (x) \d \mu (x) \\
		 \nonumber
		 & \le  \frac{1}{ \omega( B(y,\alpha) ) } \int_{ B( y, \alpha) }    \frac{1 } { \omega (B(y,a \alpha)  ) }  
	    \int_{B(y ,  a \alpha)}  \chi_{  B(x, 2 a \alpha) } (z) g (x) \omega (x) \d \mu (x) \omega (z) \d \mu (z).
	\end{align*}
	Note that in this situation, we have $\rho(y,z) \le \alpha$ and $\rho(x,y) \le a \alpha $, and, hence $ B(y, a \alpha) \subset B (z , 4 a \alpha ),   B(x, 2 a\alpha ) \subset B(z, 4 a \alpha)  $ and $\omega ( B(y, a \alpha) )  \approx \omega (B (z, 4 a \alpha) ) $.
	As a consequence, we obtain
	\begin{align*}
		 \frac{1 } { \omega (B(y,a \alpha)  ) }  \int_{B(y ,  a \alpha)}  \chi_{ B (y ,a \alpha)  } (x)  g (x) \omega (x) \d \mu (x) 
&	\lesssim  \frac{1}{  \omega (B(z,4 a \alpha)  )} \int_{  B(z,4 a \alpha) } g (x) \omega (x) \d \mu (x) \\
	& \lesssim \mathcal M _\omega g (z).
	\end{align*} 
	This implies that 
	\begin{equation*}
		M_{ a \alpha,\omega} g (y)  \lesssim M_{\alpha,\omega}  \big(  \mathcal M _\omega g  \big) (y).
	\end{equation*}
	Hence, $\langle ( \mathcal S _{a,2} ^ {0,Q} F  )^2, g \rangle _\omega $ is less than
	\begin{align}
		\nonumber
		&	 a^ { \tilde{n} }  \int_{Q'} \int_0  ^{ 2 ^{ -j_Q  }	}	 | F (y, \alpha) |^2    M_{\alpha,\omega}  \big(  \mathcal M _\omega g  \big) (y)  \omega ( B (y,a \alpha ) )\frac{\d \mu (y) \d \alpha }{\alpha V(y, \alpha)}  \\
		\nonumber
		& =  a^ { \tilde{n} }  \int_{Q'} \int_0  ^{ 2 ^{ -j_Q  }	}	 \int_{ \rho(x,y) < \alpha} | F (y, \alpha) |^2   
		\frac{  \omega ( B (y,a \alpha )  }{ \omega ( B (y,\alpha )}  \frac{\d \mu (y) \d \alpha }{\alpha V(y, \alpha)}   \mathcal M _\omega g (x) \omega(x) \d \mu (x)  \\
		\nonumber
			& \lesssim a^ { \tilde{n} +n u }  \int_{Q'} \int_0  ^{ 2 ^{ -j_Q  }	}	 \int_{ \rho(x,y) < \alpha} | F (y, \alpha) |^2   
	  \frac{\d \mu (y) \d \alpha }{\alpha V(y, \alpha)}   \mathcal M _\omega g (x) \omega(x) \d \mu (x)  \\
	  \nonumber
	  & \lesssim a^ { \tilde{n} +n u }  \bigg\{  \int_{Q'} 
	 \Big\{  \int_0  ^{ 2 ^{ -j_Q  }	}	 \int_{ \rho(x,y) < \alpha} | F (y, \alpha) |^2   
	  \frac{\d \mu (y) \d \alpha }{\alpha V(y, \alpha)}  \Big\} ^{p/2} \omega(x) \d \mu (x) \bigg\}^{2/p}  \| \mathcal M _\omega g \| _ { L ^ { ( p/2 )' } _ \omega (X)}  .
	\end{align}
	Note that $g$ is supported in $Q$. Taking the supremum over all $g \in L ^ { ( p/2 )' } _ \omega (Q) $ with norm $1$, we obtain
	\begin{align*}
	&	\Big(  \int_Q  \Big( 	\mathcal S _{a, 2 }^{0,Q} F (x)
		\Big)^{p}  \omega (x) \d \mu (x)      \Big)^{1/p} \\
		& \lesssim     a^ { ( \tilde{n} +n u ) /2}  
		  \bigg(  \int_{Q'}    \Big( \int_0  ^{ 2 ^{ -j_Q  }	}	 \int_{ \rho(x,y) < \alpha} | F (y, \alpha) |^2   
		\frac{\d \mu (y) \d \alpha }{\alpha V(y, \alpha)}  \Big)^{p/2} \omega(x) \d \mu (x) \bigg)^{1/p} .
	\end{align*}
Since $\omega (Q ' ) \lesssim  [\omega ]_{A_u} (  \mu (Q') / \mu (Q) )^u \omega (Q) \lesssim a^{nu} \omega (Q)  $ and $ 1/t -1/p \le 0$, we have
\begin{equation*}
	\omega (Q) ^{1/t-1/p} \lesssim a^{nu (1/p -1/t)}\omega (Q ' ) .
\end{equation*}

    	Therefore,
		\begin{align}
		\nonumber
		&  \bigg(    \sum_{Q \in \mathcal D}  \omega (Q)^{r/t-r/p} \Big(  \int_Q  \Big( 	\mathcal S _{a, 2 }^{0,Q} F (x)
		\Big)^{p}  \omega (x) \d \mu (x)      \Big)^{r/p}                 \bigg)^{1/r} \\
		\nonumber
		& \lesssim    a^ { ( \tilde{n} +n u ) /2}   \\
		\nonumber 
		& \quad \times \bigg(    \sum_{Q  \in \mathcal D}  \omega (Q )^{r/t-r/p}  \bigg(  \int_{Q'}    \Big(  \int_0  ^{ 2 ^{ -j_Q  }	}	 \int_{ \rho(x,y) < \alpha} | F (y, \alpha) |^2   
		\frac{\d \mu (y) \d \alpha }{\alpha V(y, \alpha)}  \Big)^{p/2} \omega(x) \d \mu (x) \bigg)^{r/p}             \bigg)^{1/r} \\
				\nonumber
		& \lesssim    a^ { ( \tilde{n} +n u ) /2}   a ^{nu ( 1/p-1/t)}  a ^{n/r}  \\
		\nonumber
		& \times
		 \bigg(    \sum_{Q ' \in \mathcal D}  \omega (Q ')^{r/t-r/p}  \bigg(  \int_{Q'}    \Big(  \int_0  ^{ 2 ^{ -j_Q  + [ \log a ] }	}	 \int_{ \rho(x,y) < \alpha} | F (y, \alpha) |^2   
		\frac{\d \mu (y) \d \alpha }{\alpha V(y, \alpha)}  \Big)^{p/2} \omega(x) \d \mu (x) \bigg)^{r/p}             \bigg)^{1/r} \\
		\nonumber
		& \lesssim a^{   ( \tilde{n} +n u ) /2 + nu ( 1/p-1/t) + n/r   }  
		\bigg(    \sum_{ Q \in \mathcal D}  \omega (Q)^{r/t-r/p} \Big(  \int_{Q}  \Big( 	\mathcal S _{1, q }^{0,Q} F (x)
		\Big)^{p}  \omega (x) \d \mu (x)      \Big)^{r/p}                 \bigg)^{1/r}.
	\end{align}
	Thus we finish the proof.
\end{proof}

\begin{cor} \label{cor Little Lusin}
				Let $a \ge 1$, $\omega \in A_u$, $ 1\le u <\infty$, $0<q \le p <\infty$,  and $s \in \mathbb R$  . 
		Let    $0<p<t<r<\infty$ 	or $0<p\le t<r=\infty$. 
		Let  $	\gamma : = 
		( \tilde{n} +n u ) / q + nu ( 1/p-1/t) + n/r$ and
		 $\lambda > \gamma$. Then
	\begin{align}
	\nonumber
	&\bigg(    \sum_{Q \in \mathcal D}  \omega (Q)^{r/t-r/p} \Big(  \int_Q  \Big( 	\mathcal G _{\lambda, q }^{s,Q} F (x)
	\Big)^{p}  \omega (x) \d \mu (x)      \Big)^{r/p}                 \bigg)^{1/r} \\
	\nonumber
	&  \approx
	\bigg(    \sum_{Q \in \mathcal D}  \omega (Q)^{r/t-r/p} \Big(  \int_Q  \Big( 	\mathcal S _{ a ,q }^{s,Q} F (x)
	\Big)^{p}  \omega (x) \d \mu (x)      \Big)^{r/p}                 \bigg)^{1/r}
\end{align}
	for all measurable function $F$.
\end{cor}

\begin{proof}
	Due to Theorem \ref{Lusin change area}, we need only to consider case $a=1$.
	Since $ \mathcal    S _{a, q }^{s, Q} F \le   \mathcal G _{\lambda ,q }^{s ,Q} F $ for any $\lambda >0$, it suffices to show that 
	\begin{align*}
		\nonumber
		  &\bigg(    \sum_{Q \in \mathcal D}  \omega (Q)^{r/t-r/p} \Big(  \int_Q  \Big( 	\mathcal G _{\lambda, q }^{s,Q} F (x)
		\Big)^{p}  \omega (x) \d \mu (x)      \Big)^{r/p}                 \bigg)^{1/r} \\
		&\lesssim
		  \bigg(    \sum_{Q \in \mathcal D}  \omega (Q)^{r/t-r/p} \Big(  \int_Q  \Big( 	\mathcal S _{ q }^{s,Q} F (x)
		\Big)^{p}  \omega (x) \d \mu (x)      \Big)^{r/p}                 \bigg)^{1/r}
	\end{align*}
	Indeed, it is easy to see that
	\begin{equation}\label{eq G S}
	( \mathcal G _{\lambda ,q }^{s ,Q} F )^q \le \sum_{\ell=0}^\infty 2^{-q\ell \lambda}  ( \mathcal S _{2^\ell, q }^{s, Q}  F )^q.
	\end{equation}
		From (\ref{eq G S}) and we have
			\begin{align}
			\nonumber
			&\bigg(    \sum_{Q \in \mathcal D}  \omega (Q)^{r/t-r/p} \Big(  \int_Q  \Big( 	\mathcal G _{\lambda, q }^{s,Q} F (x)
			\Big)^{q p/q}  \omega (x) \d \mu (x)      \Big)^{r/p}                 \bigg)^{\min(1,q) /r} \\
			\nonumber
			&\le  \bigg(    \sum_{Q \in \mathcal D}  \omega (Q)^{r/t-r/p} \Big(  \int_Q \big(  \sum_{\ell=0}^\infty 2^{-\ell q \lambda}  [ \mathcal S _{2^\ell, q }^{s  ,Q} F ]^{q}  \big)^{p/q} \omega (x) \d \mu (x)      \Big)^{r/p}                 \bigg)^{\min(1,q)/r} \\
			\nonumber
			& \le \sum_{\ell=0}^\infty 2^{-\ell \lambda \min(1,q) } 
			\bigg(    \sum_{Q \in \mathcal D}  \omega (Q)^{r/t-r/p} \Big(  \int_Q   [ \mathcal S _{2^\ell, q }^{s  ,Q} F ]^{p}  \omega (x) \d \mu (x)      \Big)^{r/p}                 \bigg)^{\min(1,q)/r}  \\
			\nonumber
			& \lesssim  \sum_{\ell=0}^\infty 2^{-\ell \lambda \min(1,q) } 2^{\ell  \gamma \min(1,q) }  
			\bigg(    \sum_{Q \in \mathcal D}  \omega (Q)^{r/t-r/p} \Big(  \int_Q   [ \mathcal S _{ q }^{s  ,Q} F ]^{p}  \omega (x) \d \mu (x)      \Big)^{r/p}                 \bigg)^{\min(1,p)/r}  \\
			\nonumber
			& \lesssim   \bigg(    \sum_{Q \in \mathcal D}  \omega (Q)^{r/t-r/p} \Big(  \int_Q   [ \mathcal S _{ q }^{s  ,Q} F ]^{p}  \omega (x) \d \mu (x)      \Big)^{r/p}                 \bigg)^{\min(1,p)/r}  ,
		\end{align}
			as long as  $\lambda >  \gamma$.
\end{proof}

\begin{rem}
	Note that  both Theorem \ref{Lusin change area}  and  Corollary \ref{cor Little Lusin} have no the condition that $r >- n t / \log  \beta  $ since the Hardy-Littlewood maximal function is not used.	
\end{rem}

We have the following characterization for weighted Bourgain-Morrey Triebel-Lizorkin type spaces  via Lusin functions and the Littlewood-Paley functions.

\begin{thm} \label{char Lusin LP}
	Let  $\psi$ be a partition of unity. Let $\omega \in A_\infty$, $0<q \le  p<\infty$,  $s\in \mathbb R$. 
		If $ 0<p <t<\infty $  and $ \max\{t, -nt / \log \beta  \} < r <\infty $, where 
	 $\beta = 1-  (1- \alpha_0)^{p/A} /  [\omega]_{A_{p/A}} $ with $ 0< A < \min \{ q , p/q_\omega  \}$, or  $0<p\le t<r=\infty$, then for  $	\gamma : = 
	 ( \tilde{n} +n q_\omega ) / q + n q_\omega ( 1/p-1/t) + n/r$ 
	 and $\lambda > \gamma$, we have 
	\begin{align}
		\nonumber
		\| f\| _{  \dot{\mathcal F}^{s,q,L}_{p,t,r,\omega} (X)  }  & \approx 	\bigg(    \sum_{Q \in \mathcal D}  \omega (Q)^{r/t-r/p} \Big(  \int_Q  \Big( 	\mathcal S _{ q }^{s,Q} ( \psi (\alpha \sqrt{L}) f ) (x)
		\Big)^{p}  \omega (x) \d \mu (x)      \Big)^{r/p}                 \bigg)^{1/r}  \\
		\nonumber
		&  \approx  \bigg(    \sum_{Q \in \mathcal D}  \omega (Q)^{r/t-r/p} \Big(  \int_Q  \Big( 	\mathcal G _{\lambda, q }^{s,Q}   (\psi  (\alpha \sqrt{L} ) f)  (x)
		\Big)^{p}  \omega (x) \d \mu (x)      \Big)^{r/p}                 \bigg)^{1/r},
	\end{align}
	for every $f\in \mathcal S _\infty '$.
\end{thm}

\begin{proof}
	We first prove that 
	\begin{equation} \label{eq Lus}
			\bigg(    \sum_{Q \in \mathcal D}  \omega (Q)^{r/t-r/p} \Big(  \int_Q  \Big( 	\mathcal S _{ q }^{s,Q} ( \psi (\alpha \sqrt{L}) f ) (x)
		\Big)^{p}  \omega (x) \d \mu (x)      \Big)^{r/p}                 \bigg)^{1/r}
		\lesssim 
		 \| f\| _{  \dot{\mathcal F}^{s,q,L}_{p,t,r,\omega} (X)  }  .
	\end{equation}
 Observe that 
	\begin{equation*}
		| \psi  (\alpha \sqrt{L} ) f (y) | \le \psi_\lambda^*   ( \alpha \sqrt{L} ) f (x)
	\end{equation*}
	for all $\lambda >0 $ and $\rho(x,y)  < \alpha$. Hence,
	\begin{align}
		\nonumber
		\mathcal S _{ q }^{s,Q  } ( \psi (\alpha \sqrt{L}) f ) (x)   & \le  \bigg(   \int_0^{ 2 ^{-j_Q} } \int_{B(x,  \alpha)}  (  \alpha ^{-s} |    \psi_\lambda^*   ( \alpha \sqrt{L} ) f (x) | )^q  \frac{\d \mu(y) \d\alpha }{\alpha V(x,\alpha)}   \bigg)^{1/q}
		  \\
		  \nonumber
		& \lesssim \bigg(   \int_0^{ 2 ^{-j_Q} }  (  \alpha ^{-s} |    \psi_\lambda^*   ( \alpha \sqrt{L} ) f (x) | )^q  \frac{\d\alpha }{\alpha }   \bigg)^{1/q}.
	\end{align}
	Theorem \ref{char conti}  implies (\ref{eq Lus}).
	
	Due to Corollary \ref{cor Little Lusin}, it remains to show that 
	\begin{equation} \label{eq Lus 2}
		  \| f\| _{  \dot{\mathcal F}^{s,q,L}_{p,t,r,\omega} (X)  } \lesssim 
		  \bigg(    \sum_{Q \in \mathcal D}  \omega (Q)^{r/t-r/p} \Big(  \int_Q  \Big( 	\mathcal G _{\lambda, q }^{s,Q}   (\psi  (\alpha \sqrt{L} ) f)  (x)
		  \Big)^{p}  \omega (x) \d \mu (x)      \Big)^{r/p}                 \bigg)^{1/r} .
	\end{equation}
	By Lemma \ref{est peetre psi}, we have
	\begin{equation*}
		| \psi  (\alpha \sqrt{L} ) f (x) | \lesssim \bigg(  \int_X \frac{1}{ V(x,\alpha) }  \frac{   | \psi  (\alpha \sqrt{L} ) f (z)|^q  }{ ( 1+ \rho(x,z) /\alpha )^{\lambda q} }     \d \mu(z)      \bigg) ^{1/q}
	\end{equation*} 
	for all $x\in X$, $\lambda >0$ and $\alpha >0$.
	This implies that 
	\begin{align}
		\nonumber
		 & \bigg(    \sum_{Q \in \mathcal D}  \omega (Q)^{r/t-r/p} \Big(  \int_Q  \Big(  \int_{0}^{2^{-j_Q} }| \alpha ^{-s} \psi  (\alpha  \sqrt{L} )f (x)|^q\frac{\d \alpha}{ \alpha} \Big)^{p/q}  \omega (x) \d \mu (x)      \Big)^{r/p}                 \bigg)^{1/r} \\
		 \nonumber
		 &\lesssim	    \bigg(    \sum_{Q \in \mathcal D}  \omega (Q)^{r/t-r/p} \Big(  \int_Q  \Big( 	\mathcal G _{\lambda, q }^{s,Q}   (\psi  (\alpha \sqrt{L} ) f)  (x)
		 \Big)^{p}  \omega (x) \d \mu (x)      \Big)^{r/p}                 \bigg)^{1/r}.
	\end{align}
	Using Theorem \ref{char conti} , we obtain  (\ref{eq Lus 2}).
\end{proof}

\section{Atomic decompositions} \label{atm dec}
We first recall the definition of weighted atoms related to $L$, which can be seen in \cite{BBD20}.

\begin{defn}
	Let $0 <p \le \infty $, $M \in \mathbb N _+$  and $\omega \in A_\infty$.  A function $a$ is said to be an $ (L,M, p , \omega ) $ atom  if there exists a dyadic cube $Q \in \mathcal D$ such that:
	
	{\rm (i)} $a =L^M b $ with $b  \in D (L^M)$, where $D (L^M) $ is the domain of  $L^M$;
	
	{\rm (ii)} supp $L^k b \subset 3 B_Q$, $ k =0, \ldots , 2M$, 	where $B_Q$ is a ball associated with $Q$ defined in Remark  \ref{dya cube};
	
	{\rm (iii)} $ | L^k b(x) | \le  \ell(Q) ^{2(M-k) }  \omega (Q) ^{-1/p} $,  $ k =0, \ldots , 2M$.
\end{defn}

\subsection{Atomic decompositions for $  \dot{\mathcal B}^{s,q,L}_{p,t,r,\omega} (X)$}

\begin{thm} \label{atm B 1}
	Let  $\omega \in A_\infty$,  $0<q\le \infty$, $s\in \mathbb R$, $M \in\mathbb N _+$. 
	Let  $0<p<t<\infty$  and $ \max\{t, -nt / \log \beta  \} < r <\infty $, where  $\beta = 1-  (1- \alpha_0)^{p/A} /  [\omega]_{A_{p/A}} $ with $0< A < p / q_\omega$, 	or let $0<p\le t<r=\infty$.	
	 If $f \in  \dot{\mathcal B}^{s,q,L}_{p,t,r,\omega} (X)  $, then there exist a sequence of $ (L, M, p ,\omega ) $ atoms $ \{  a_Q \} _ { Q \in \mathcal D _\nu, \nu \in \mathbb Z  } $ and a sequence of coefficients $ \{  s_Q \} _{Q\in \mathcal D _\nu , \nu \in \mathbb Z } $ such that 
	\begin{equation*}
		f =  \sum_{\nu \in \mathbb Z} \sum_{Q \in \mathcal D _\nu} s_Q a_Q \; \; \mathrm{in} \; \mathcal S _\infty '.
	\end{equation*}
	Moreover, 
	\begin{align} \label{eq atm B 1 1}
&	 \left( 	\sum_{Q \in \mathcal D }\omega (Q) ^{r/t - r/p} \left( \sum_{\nu = j_Q}^\infty  \left(   \sum_{P \in \D_{\nu} , P \subset Q} 2^{\nu s  p }|s_P|^p  \right)^{q/p}
\right) ^{r/q}  \right)^{1/r} 	\lesssim
		\| f \| _{    \dot{\mathcal B}^{s,q,L}_{p,t,r,\omega} (X)   }. 
	\end{align}
\end{thm}

\begin{proof}
	We use the idea from \cite[Theorem 4.2]{BBD20}.
	Let $\psi $ be a partition of unity. Let $\varphi \in \mathscr S (\mathbb R)$ be an even function with supp $\varphi \subset (-1, 1)$ and $\int_{-1}^1 \varphi (t) \d t = 2\pi $. Let $\Phi $ be the Fourier transform of $\varphi$. By Lemma \ref{iden}, for $f\in \mathcal S _\infty '$, we have
	\begin{equation*}
		f = c \int_0^\infty \psi (\alpha \sqrt{L}) \Phi(\alpha \sqrt{L}) f \frac{\d \alpha}{\alpha}
	\end{equation*}
	in $\mathcal S _\infty '$, where $c = \big( \int_0^\infty \psi(\xi)  \Phi (\xi)   \xi ^{-1} \d \xi  \big) ^{-1}$.  For brevity, we write $c=1$. Let $\psi_M (\xi)  = \xi ^{-2M} \psi (\xi)$. By Lemma \ref{cube}, we have
	\begin{align}
		\nonumber
		 f = & \sum_{\nu \in \mathbb Z} \int_{2 ^{-\nu-1} } ^{ 2^{-\nu} } (\alpha^2 L) ^M \Phi( \alpha \sqrt{L})  ( \psi_M (\alpha \sqrt{L})  f) \frac{\d \alpha}{\alpha}\\
		 \nonumber
		 = &  \sum_{\nu \in \mathbb Z} \sum_{Q \in \mathcal D _\nu} \int_{2 ^{-\nu-1} } ^{ 2^{-\nu} } (\alpha^2 L) ^M \Phi( \alpha \sqrt{L})  ( \psi_M (\alpha \sqrt{L})  f  \cdot \chi _Q )  \frac{\d \alpha}{\alpha}
	\end{align}
	For each $ \nu \in \mathbb Z $ and $Q \in \mathcal D _\nu$, set 
	\begin{equation*}
		s_Q = \omega(Q) ^{1/p} \sup_{y\in Q}  \int_{2 ^{-\nu-1} } ^{ 2^{-\nu} } |  \psi_M (\alpha \sqrt{L})  f (y) |  \frac{\d \alpha}{\alpha},
	\end{equation*}
	and $a_Q = L^M b_Q$, where 
	\begin{equation}\label{eq b_Q}
		b_Q = \frac{1}{s_Q} \int_{2 ^{-\nu-1} } ^{ 2^{-\nu} } \alpha^{2M}  \Phi(\alpha \sqrt{L})  ( \psi_M (\alpha \sqrt{L})  f  \cdot \chi _Q ) \frac{\d \alpha}{\alpha}.
	\end{equation}
	From  \cite[Theorem 4.2]{BBD20}, we have
	\begin{itemize}
		\item[(i)]  $ f =  \sum_{\nu \in \mathbb Z} \sum_{Q \in \mathcal D _\nu} s_Q a_Q $ in $\mathcal S _\infty '$;
		\item[(ii)] supp $L^k b \subset 3B_Q$, for $k = 0, \ldots, 2M$;
		\item[(iii)] $ |L^k b_Q (x) | \lesssim 2^ { - \nu (2M-2k)}  \omega(Q) ^{-1/p}  $ , for $k = 0, \ldots, 2M$, $Q\in\mathcal D _\nu$.
	\end{itemize}
	Hence $a_Q$ is (a multiple of) an $ (L, M, p, \omega) $ atom. It remains to prove (\ref{eq atm B 1 1}).
	Indeed, for any $\lambda >0$, 
	\begin{align*}
		\omega (Q) ^{-1/p} s_Q\chi _Q  & =\chi_Q   \sup_{y\in Q}  \int_{2 ^{-\nu-1} } ^{ 2^{-\nu} } |  \psi_M (\alpha \sqrt{L})  f (y) |  \frac{\d \alpha}{\alpha} \\
		 &  \lesssim  \chi _Q  	F_{M,\lambda} ^*  (\sqrt{L}) f,
	\end{align*}
	where 
	\begin{equation*}
		F_{M,\lambda} ^*  (\sqrt{L}) f (x) = \sup_{y\in X} \frac{   \int_{2 ^{-\nu-1} } ^{ 2^{-\nu} } |  \psi_M (\alpha \sqrt{L})  f (y) |  \frac{\d \alpha}{\alpha}       }{ (1+ 2^\nu \rho(x,y)) ^\lambda }.
	\end{equation*}
Note that $P \in\mathcal D_{\nu} $ are disjoint with each other.
	As a consequence, 
	\begin{equation*}
		\chi_Q \left( \sum_{P \in \mathcal D_{\nu}} 2^{	\nu s} |s_P| \omega(P) ^{-1/p}  \chi_P \right) \lesssim	\chi_Q 2^{	\nu s}	F_{M,\lambda} ^*  (\sqrt{L}) f .
	\end{equation*}
	On the other hand,  arguing similarly to (\ref{hardy v j}), we obtain
	\begin{equation} \label{F hardy}
		2^{\nu s } |  F_{M,\lambda} ^*  (\sqrt{L}) f (x)| \lesssim \sum_{j\in \mathbb Z} 2^{ - (2m - s) |\nu -j| } 2^{js} \psi_{j,\lambda}^* (\sqrt{L}) f (x) . 
	\end{equation}
	(Equation (\ref{F hardy}) can also be seen in \cite[(86)]{BBD20}.)	
 Let $m$ be sufficiently large such that $ 2m -s  > (1/p -1/t) n p /A.$ Let $\lambda $  be sufficiently large such that satisfy the condition of Theorem  \ref{Peetre}.
	Hence,  by Theorems \ref{hardy bourgain}, \ref{Peetre}, we obtain
		\begin{align*}
		\nonumber
		& \left( 	\sum_{Q \in \mathcal D }\omega (Q) ^{r/t - r/p} \left( \sum_{\nu = j_Q}^\infty  \left(   \sum_{P \in \D_{\nu} , P \subset Q} 2^{\nu s  p }|s_P|^p  \right)^{q/p}
		\right) ^{r/q}  \right)^{1/r}  \\
		\nonumber
			&   \approx \left( 	\sum_{Q \in \mathcal D }\omega (Q) ^{r/t - r/p} \left( \sum_{\nu = j_Q}^\infty  \left(  \int_Q    \left(  \sum_{P \in \D_{\nu} , P \subset Q} 2^{\nu s}|s_P| \omega (P)^{-1/p} \chi_P  \right) ^p \omega (x) \d \mu (x) \right)^{q/p}
		\right) ^{r/q}  \right)^{1/r}  \\
		\nonumber
		&\lesssim
		\left( 
		\sum_{Q \in \mathcal D }\omega (Q) ^{r/t - r/p} \left( \sum_{\nu = j_Q}^\infty 
		\left(    \int_Q  \big( \sum_{j\in \mathbb Z} 2^{ - (2m - s) |\nu -j| } 2^{js} \psi_{j,\lambda}^* (\sqrt{L}) f (x)   \big)  ^p \omega (x) \d \mu (x)
		\right)^{q/p}
		\right) ^{r/q}  \right)^{1/r} \\
		\nonumber
		&\lesssim
		\bigg\{
		\sum_{Q \in \mathcal D }\omega (Q) ^{r/t - r/p} \bigg( \sum_{j = j_Q}^\infty 
		\bigg(   \int_Q  \big(  2^{js} \psi_{j,\lambda}^* (\sqrt{L}) f (x)   \big)  ^p \omega (x) \d \mu (x)
		\bigg)^{q/p}
		\bigg) ^{r/q}  \bigg\}^{1/r} \\
		\nonumber
		&\lesssim
		\| f \| _{    \dot{\mathcal B}^{s,q,L}_{p,t,r,\omega} (X)   }. 
	\end{align*}

Thus the proof is complete.
\end{proof}

\begin{rem}
	If $M_{p,\omega}^{t,r} = L^p(w)$, then 
 Theorem \ref{atm B 1} is just Theorem 4.2  in \cite{BBD20}.
\end{rem}

Conversely, each atomic decomposition with suitable coefficients belongs to the spaces $  \dot{\mathcal B}^{s,q,L}_{p,t,r,\omega} (X) .$
\begin{thm}\label{atm B 2}
	Let  $\omega \in A_\infty$,  $0<q\le \infty$, $s\in \mathbb R$, $M \in\mathbb N _+$. 
	Let  $0<p<t<\infty$  and $ \max\{t, -nt / \log \beta  \} < r <\infty $, where  $\beta = 1-  (1- \alpha_0)^{p/B} /  [\omega]_{A_{p/B}} $ with $0< B< p / q_\omega$, 	or let $0<p\le t<r=\infty$.
		Let $ M$ be a sufficiently large number such that
		\begin{equation*}
			M > \frac{1}{2} \bigg( \max \Big( n q_\omega  / \min(1,p,q) -s ,   s \Big) +n +(1/p -1/t) n  q_\omega \bigg).
		\end{equation*}
	If 
	\begin{equation*}
		f =  \sum_{\nu \in \mathbb Z} \sum_{Q \in \mathcal D _\nu} s_Q a_Q \; \;  \mathrm{in} \; \mathcal S _\infty '
	\end{equation*} 	
	where $\{a_Q\} _ { Q \in \mathcal D _\nu ,\nu \in \mathbb Z }$ is a sequence of $  (L, M, p , \omega) $ atoms and $  \{  s_Q \} _ { Q \in \mathcal D _\nu ,\nu \in \mathbb Z }$ is a sequence of coefficients satisfying
	\begin{equation*}
	 \left( 	\sum_{Q \in \mathcal D }\omega (Q) ^{r/t - r/p} \left( \sum_{\nu = j_Q}^\infty  \left(   \sum_{P \in \D_{\nu} , P \subset Q} 2^{\nu s  p }|s_P|^p  \right)^{q/p}
	\right) ^{r/q}  \right)^{1/r} 
	<\infty,
	\end{equation*}
	then $ f \in  \dot{\mathcal B}^{s,q,L}_{p,t,r,\omega} (X)   $ and 
	\begin{equation} \label{eq atm B 2 2}
		\| f \| _{ \dot{\mathcal B}^{s,q,L}_{p,t,r,\omega} (X)   }  \lesssim 	 \left( 	\sum_{Q \in \mathcal D }\omega (Q) ^{r/t - r/p} \left( \sum_{\nu = j_Q}^\infty  \left(   \sum_{P \in \D_{\nu} , P \subset Q} 2^{\nu s  p }|s_P|^p  \right)^{q/p}
		\right) ^{r/q}  \right)^{1/r} .
	\end{equation}
\end{thm}

\begin{proof}
	Let $\psi$ be a partition of unity. Since 
	\begin{equation*}
	f =  \sum_{\nu \in \mathbb Z} \sum_{Q \in \mathcal D _\nu} s_Q a_Q \; \;  \mathrm{in} \; \mathcal S _\infty '   ,
\end{equation*} 
for each $j \in \mathbb Z $, we have	
\begin{align}
	\nonumber
	\psi_j (  \sqrt{L} ) f = \sum_{\nu \in \mathbb Z} \sum_{Q \in \mathcal D _\nu } s_Q \psi_j( \sqrt{L} )a_Q.
\end{align}
Fix $ \tilde q  \in (q_\omega, \infty )$
and $A <\min \{1,p,q \} $.
From \cite[(89)]{BBD20}, we have
	\begin{align}
	\nonumber
	2^{js} |\psi_j (\sqrt{L})  f |  & \lesssim \sum_{  \nu : \nu  \ge j} 2^{  -(\nu -j)(2M -n -n \tilde q /A + s)  }  \mathcal M _{\omega, A} \Big(   \sum_{Q \in \mathcal D _\nu} 2^{\nu s} |s_Q| \omega (Q) ^{-1/p} \chi _Q    \Big) \\
	\nonumber
	& +  \sum_{  \nu : \nu  \ge j} 2^{  -(2M -n- s) (j -\nu)  }  \mathcal M _{\omega, A} \Big(   \sum_{Q \in \mathcal D _\nu} 2^{\nu s} |s_Q| \omega (Q) ^{-1/p} \chi _Q    \Big) .
\end{align}
Let $2M -n -n \tilde q /A + s > (1/p -1/t) n  \tilde q$ and $2M -n- s > (1/p -1/t) n  \tilde q $ such that we can use Theorem \ref{hardy bourgain}. That is 
\begin{equation*}
	M > \frac{1}{2} \bigg( \max \Big( n \tilde q /A -s ,   s \Big) +n + (1/p -1/t) n  \tilde q\bigg).
\end{equation*}

By Theorem \ref{hardy bourgain}, we obtain 
\begin{align}
	\nonumber
	&	\bigg(    \sum_{Q \in \mathcal D}  \omega (Q)^{r/t-r/p} \Big(  \sum_{j=j_Q}^\infty \big(   \int_Q | 2^{js}  \psi_j (   \sqrt{L} )f (x)|^p \omega (x)\d \mu (x)   \big)^{q/p}   \Big)^{r/q}                 \bigg)^{1/r}  \\
	\nonumber
	& \lesssim 		\bigg(    \sum_{Q \in \mathcal D}  \omega (Q)^{r/t-r/p} \Big(  \sum_{\nu = j_Q }^\infty \big(   \int_Q |  \mathcal M _{\omega, A} \Big(   \sum_{P \in \mathcal D _\nu} 2^{\nu s} |s_P| \omega (P) ^{-1/p} \chi _P    \Big)|^p \omega (x)\d \mu (x)   \big)^{q/p}   \Big)^{r/q}                 \bigg)^{1/r}  \\
	\nonumber
	& \lesssim 		 \left( 	\sum_{Q \in \mathcal D }\omega (Q) ^{r/t - r/p} \left( \sum_{\nu = j_Q}^\infty  \left(   \sum_{P \in \D_{\nu} , P \subset Q} 2^{\nu s  p }|s_P|^p  \right)^{q/p}
	\right) ^{r/q}  \right)^{1/r}. \qedhere
\end{align}
\end{proof}

\subsection{Atomic decompositions for $  \dot{\mathcal F}^{s,q,L}_{p,t,r,\omega} (X) .$}

\begin{thm} \label{atm F 1}
		Let  $\omega \in A_\infty$,  $0<q\le \infty$, $s\in \mathbb R$, $M \in\mathbb N _+$. 
	Let  $0<p<t<\infty$  and $ \max\{t, -nt / \log \beta  \} < r <\infty $, where  $\beta = 1-  (1- \alpha_0)^{p/A} /  [\omega]_{A_{p/A}} $ with $0< A < \min \{ q , p/q_\omega  \}$, 	or let $0<p\le t<r=\infty$.
	 If $f \in  \dot{\mathcal F}^{s,q,L}_{p,t,r,\omega} (X)  $, then there exist a sequence of $ (L, M, p ,\omega ) $ atoms $ \{  a_Q \} _ { Q \in \mathcal D _\nu, \nu \in \mathbb Z  } $ and a sequence of coefficients $ \{  s_Q \} _{Q\in \mathcal D _\nu , \nu \in \mathbb Z } $ such that 
	\begin{equation*}
		f =  \sum_{\nu \in \mathbb Z} \sum_{Q \in \mathcal D _\nu} s_Q a_Q \; \; \mathrm{in} \; \mathcal S _\infty '.
	\end{equation*}
	Moreover, 
	\begin{align} \label{eq atm F 1 1}
		\nonumber
		&	\bigg(
		\sum_{Q \in \mathcal D }\omega (Q) ^{r/t - r/p} 
		\bigg(      \int_Q  \bigg( \sum_{\nu = j_Q}^\infty 
		2^{\nu sq}   \sum_{ P \in \D _{\nu}, P \subset Q }  
		\big(   | \omega( P )^{-1/p}  s_ {P} | \chi_P (x) \big)^q      \bigg)^{p/q}  \omega(x) \d \mu (x)
		    \bigg)^{r/p}
  \bigg)^{1/r}  \\
		& \lesssim  \| f \| _{    \dot{\mathcal F}^{s,q,L}_{p,t,r,\omega} (X)   }.
	\end{align}
\end{thm}

\begin{proof}
	Similar as Theorem \ref{atm B 1}, we  have the representation
		\begin{itemize}
		\item[(i)]  $ f =  \sum_{\nu \in \mathbb Z} \sum_{Q \in \mathcal D _\nu} s_Q a_Q $ in $\mathcal S _\infty '$;
		\item[(ii)] supp $L^k b \subset 3B_Q$, for $k = 0, \ldots, 2M$;
		\item[(iii)] $ |L^k b_Q (x) | \lesssim 2^ { - \nu (2M-2k)}  \omega(Q) ^{-1/p}  $ , for $k = 0, \ldots, 2M$, $Q\in\mathcal D _\nu$,
	\end{itemize}
	where 
	\begin{equation*}
	s_Q = \omega(Q) ^{1/p} \sup_{y\in Q}  \int_{2 ^{-\nu-1} } ^{ 2^{-\nu} } |  \psi_M (\alpha \sqrt{L})  f (y) |  \frac{\d \alpha}{\alpha},
\end{equation*}
	and $a_Q = L^M b_Q$  is an $(L,M,p,\omega)$ atom defined by
	\begin{equation*}
		b_Q = \frac{1}{s_Q} \int_{2 ^{-\nu-1} } ^{ 2^{-\nu} } \alpha^{2M}  \Phi(\alpha \sqrt{L})  [\psi_M (\alpha \sqrt{L})  f  \cdot \chi _Q ]  \frac{\d \alpha}{\alpha}.
	\end{equation*}
	It remains to prove (\ref{eq atm F 1 1}). Indeed, for any $\lambda >0$, 
	\begin{equation*}
	\omega(Q) ^{-1/p} s_Q \chi _Q  \lesssim 	  \chi _Q F_{M,\lambda} ^*  (\sqrt{L}) f,
	\end{equation*}
	where 
	\begin{equation*}
		F_{M,\lambda} ^*  (\sqrt{L}) f (x) = \sup_{y\in X} \frac{   \int_{2 ^{-\nu-1} } ^{ 2^{-\nu} } |  \psi_M (\alpha \sqrt{L})  f (y) |  \frac{\d \alpha}{\alpha}       }{ (1+ 2^\nu \rho(x,y)) ^\lambda }.
	\end{equation*}
	Note that $ P \in \mathcal D_{nu} $ are disjoint with each other.
	As a consequence, 
	\begin{equation*}
		 \sum_{ P \in \mathcal D_{nu}, P \subset Q }  
		\big(   | \omega( P )^{-1/p}  s_ {P} | \chi_{P} \big)^q  
		\lesssim \chi _Q  | 	F_{M,\lambda} ^*  (\sqrt{L}) f |^q .
	\end{equation*}
	On the other hand, by (\ref{F hardy}), we have
		\begin{equation*} 
		2^{\nu s } |  F_{M,\lambda} ^*  (\sqrt{L}) f (x)| \lesssim \sum_{j\in \mathbb Z} 2^{  - (2m - s) |\nu -j| } 2^{js} \psi_{j,\lambda}^* (\sqrt{L}) f (x) . 
	\end{equation*}
	 Let $m$ be sufficiently large such that $ 2m -s  >  (1/p -1/t) n  q_\omega .$ Let $\lambda $  be sufficiently large such that satisfy the condition of Theorem  \ref{Peetre}.
	Therefore, 	  the left side of  (\ref{eq atm F 1 1}) is less that
	\begin{align}
		\nonumber
		&	\bigg\{
		\sum_{Q \in \mathcal D }\omega (Q) ^{r/t - r/p} 
		\bigg(      \int_Q  \bigg( \sum_{\nu = j_Q}^\infty  \Big( 
	\sum_{j\in \mathbb Z} 2^{ - (2m - s) |\nu -j| } 2^{js} \psi_{j,\lambda}^* (\sqrt{L}) f (x)  \Big) ^q     \bigg)^{p/q}  \omega(x) \d \mu (x)
		\bigg)^{r/p}
		\bigg\}^{1/r}  \\
		\nonumber
		& \lesssim 
		\bigg\{
		\sum_{Q \in \mathcal D }\omega (Q) ^{r/t - r/p} 
		\bigg(      \int_Q  \bigg( \sum_{ j  = j_Q}^\infty 
		 \big( 2^{js} \psi_{j,\lambda}^* (\sqrt{L}) f (x)    \big)^q     \bigg)^{p/q}  \omega(x) \d \mu (x)
		\bigg)^{r/p}
		\bigg\}^{1/r} \\
		\nonumber
		& \lesssim  \| f \| _{    \dot{\mathcal F}^{s,q,L}_{p,t,r,\omega} (X)   }. \qedhere
	\end{align}
\end{proof}

	If $ 0<  p=t <r =\infty$, Theorem \ref{atm B 1} becomes \cite[Theorem 4.6]{BBD20}.
For the converse direction, we have the following result.
\begin{thm} \label{atm F 2}
		Let  $\omega \in A_\infty$,  $0<q\le \infty$, $s\in \mathbb R$, $M \in\mathbb N _+$. 
	Let  $0<p<t<\infty$  and $ \max\{t, -nt / \log \beta  \} < r <\infty $, where  $\beta = 1-  (1- \alpha_0)^{p/B} /  [\omega]_{A_{p/B}} $ with $0< B< \min \{ q , p/q_\omega  \}$, 	or let $0<p\le t<r=\infty$.
			Let $ M$ be a sufficiently large number such that
	\begin{equation*}
		M > \frac{1}{2} \bigg( \max \Big( n  q_\omega / \min(1,p,q) -s ,   s \Big) +n +(1/p -1/t) n  q_\omega  \bigg).
	\end{equation*}
	If 
	\begin{equation*}
		f =  \sum_{\nu \in \mathbb Z} \sum_{Q \in \mathcal D _\nu} s_Q a_Q \; \;  \mathrm{in} \; \mathcal S _\infty '
	\end{equation*} 	
	where $\{a_Q\} _ { Q \in \mathcal D _\nu ,\nu \in \mathbb Z }$ is a sequence of $  (L, M, p , \omega) $ atoms and $  \{  s_Q \} _ { Q \in \mathcal D _\nu ,\nu \in \mathbb Z }$ is a sequence of coefficients satisfying
	\begin{equation*}
\bigg(    \sum_{Q \in \mathcal D}  \omega (Q)^{r/t-r/p} \Big(  \int_Q  \Big(  \sum_{\nu =j_Q}^\infty 2^{\nu sq} \big(    \sum_{ P \in \mathcal D_{\nu}, P \subset Q  }  |s_P| \omega (P) ^{-1/p} \chi _{P}  \big)^q \Big)^{p/q}  \omega (x) \d \mu (x)      \Big)^{r/p}                 \bigg)^{1/r}  <\infty,
	\end{equation*}
	then $ f \in  \dot{\mathcal F}^{s,q,L}_{p,t,r,\omega} (X)   $ and 
	\begin{align} \label{eq atm F 2 2}
		\nonumber
	&	\| f \| _{ \dot{\mathcal F}^{s,q,L}_{p,t,r,\omega} (X)   } \\
	&  \lesssim \bigg(    \sum_{Q \in \mathcal D}  \omega (Q)^{r/t-r/p} \Big(  \int_Q  \Big(  \sum_{\nu =j_Q}^\infty 2^{\nu sq} \big(    \sum_{ P \in \mathcal D_{\nu}, P \subset Q  }  |s_P| \omega (P) ^{-1/p} \chi _{P}  \big)^q \Big)^{p/q}  \omega (x) \d \mu (x)      \Big)^{r/p}                 \bigg)^{1/r} .
	\end{align}
\end{thm}

\begin{proof}
		Let $\psi$ be a partition of unity. Since 
	\begin{equation*}
		f =  \sum_{\nu \in \mathbb Z} \sum_{Q \in \mathcal D _\nu} s_Q a_Q \; \;  \mathrm{in} \; \mathcal S _\infty '   ,
	\end{equation*} 
	for each $j \in \mathbb Z $, we have	
	\begin{align}
		\nonumber
		\psi_j (  \sqrt{L} ) f = \sum_{\nu \in \mathbb Z} \sum_{Q \in \mathcal D _\nu } s_Q \psi_j( \sqrt{L} )a_Q .
	\end{align}
	Fix $ \tilde q  \in (q_\omega, \infty )$
	and $A <\min \{1,p,q \} $.
	From \cite[(89)]{BBD20}, we have
	\begin{align}
		\nonumber
		2^{js} |\psi_j (\sqrt{L})  f |  & \lesssim \sum_{  \nu : \nu  \ge j} 2^{  -(\nu -j)(2M -n -n \tilde q /A + s)  }  \mathcal M _{\omega, A} \Big(   \sum_{Q \in \mathcal D _\nu} 2^{\nu s} |s_Q| \omega (Q) ^{-1/p} \chi _Q    \Big) \\
		\nonumber
		& +  \sum_{  \nu : \nu  \ge j} 2^{  -(2M -n- s) (j -\nu)  }  \mathcal M _{\omega, A} \Big(   \sum_{Q \in \mathcal D _\nu} 2^{\nu s} |s_Q| \omega (Q) ^{-1/p} \chi _Q    \Big) .
	\end{align}
	Let $2M -n -n \tilde q /A + s > (1/p -1/t) n  \tilde q $ and $2M -n- s > (1/p -1/t) n  \tilde q $ such that we can use Theorem \ref{hardy bourgain}. That is 
	\begin{equation}
		M > \frac{1}{2} \bigg( \max \Big( n \tilde q /A -s ,   s \Big) +n +(1/p -1/t) n  \tilde q  \bigg) .
	\end{equation}
	By Theorem \ref{hardy bourgain}, we obtain
	\begin{align}
		\nonumber
		&\bigg(    \sum_{Q \in \mathcal D}  \omega (Q)^{r/t-r/p} \Big(  \int_Q  \Big(  \sum_{j=j_Q}^\infty | 2^{js}  \psi_j  ( \sqrt{L} )f (x)|^q \Big)^{p/q}  \omega (x) \d \mu (x)      \Big)^{r/p}                 \bigg)^{1/r}  \\
		\nonumber
		&\lesssim \bigg(    \sum_{Q \in \mathcal D}  \omega (Q)^{r/t-r/p} \Big(  \int_Q  \Big(  \sum_{\nu =j_Q}^\infty \big(    \mathcal M _{\omega, A} \Big(   \sum_{P \in \mathcal D _\nu} 2^{\nu s} |s_P| \omega (P) ^{-1/p} \chi _P \Big) \big)^q \Big)^{p/q}  \omega (x) \d \mu (x)      \Big)^{r/p}                 \bigg)^{1/r}  \\
			\nonumber
		&\lesssim \bigg(    \sum_{Q \in \mathcal D}  \omega (Q)^{r/t-r/p} \Big(  \int_Q  \Big(  \sum_{\nu =j_Q}^\infty \big(    \sum_{P \in \mathcal D _\nu} 2^{\nu s} |s_P| \omega (P) ^{-1/p} \chi _P \big)^q \Big)^{p/q}  \omega (x) \d \mu (x)      \Big)^{r/p}                 \bigg)^{1/r}  \\
		\nonumber
		&= \bigg(    \sum_{Q \in \mathcal D}  \omega (Q)^{r/t-r/p} \Big(  \int_Q  \Big(  \sum_{\nu =j_Q}^\infty 2^{\nu sq} \big(    \sum_{ P \in \mathcal D_{\nu}, P \subset Q  }  |s_P| \omega (P) ^{-1/p} \chi _{P}  \big)^q \Big)^{p/q}  \omega (x) \d \mu (x)      \Big)^{r/p}                 \bigg)^{1/r}  .
	\end{align}
Thus we prove (\ref{eq atm F 2 2}) and the proof is complete.
\end{proof}
\begin{rem}
	{\rm (i)}  Let  $\omega \in A_\infty$,  $0<q< \infty$, $s\in \mathbb R$. 
	Let  $0<p<t<\infty$  and $ \max\{t, -nt / \log \beta  \} < r <\infty $, where  $\beta = 1-  (1- \alpha_0)^{p/B} /  [\omega]_{A_{p/B}} $ with $0< B<  p/q_\omega $.
	It is easy to see that each atom $a_Q = L^{2M} b_Q $, defined by (\ref{eq b_Q}) belongs to the spaces of test functions $\mathcal S _\infty $ (see also \cite[Remark 4.8]{BBD20}).  As a direct consequence of the atomic decomposition results in Theorem \ref{atm B 1},   $\mathcal S _\infty $ is dense in both $\dot{\mathcal B}^{s,q,L}_{p,t,r,\omega} (X)  $.
	
	{\rm (ii)}  Let  $\omega \in A_\infty$,  $0<q< \infty$, $s\in \mathbb R$. 
	Let  $0<p<t<\infty$  and $ \max\{t, -nt / \log \beta  \} < r <\infty $, where  $\beta = 1-  (1- \alpha_0)^{p/B} /  [\omega]_{A_{p/B}} $ with $0< B< \min \{ q , p/q_\omega  \}$.
		It is easy to see that each atom $a_Q = L^{2M} b_Q $, defined by (\ref{eq b_Q}) belongs to the spaces of test functions $\mathcal S _\infty $.  As a direct consequence of the atomic decomposition results in Theorem  \ref{atm F 1},  $\mathcal S _\infty $ is dense in  $\dot{\mathcal  F}^{s,q,L}_{p,t,r,\omega} (X)  $.
\end{rem}

\subsection{Molecular decompositions}	
Let us recall the definition of molecular; see \cite[page 72]{BBD20}.
\begin{defn}
	Let $ 0<p <\infty $, $N>0$, $ M \in  \mathbb N _+ $ and $\omega \in A _\infty $. A function $m$ is  said to be an $( L,M,N, p,\omega )$ molecule if there exists a dyadic cube $Q \in \mathcal D$ such that
	\begin{itemize}
		\item[\rm (i)]  $m = L^M b; $
		\item[\rm (ii)]  $ | L^k b(x) | \le  \ell(Q) ^{   2(M-k)  } \omega(Q)^{-1/p} \Big(1+ \frac{|x-x_Q|}{\ell(Q)} \Big)^{-N} $, $k=0,\ldots,2M$.
	\end{itemize}
\end{defn}	
	Since an atom $a$ is also an molecule,  
	we immediately get the following result.
	
	\begin{thm}
			Let  $\omega \in A_\infty$,  $0<q\le \infty$, $s\in \mathbb R$, $M \in\mathbb N _+$. 
		
	{\rm (a)}	
		Let  $0<p<t<\infty$  and $ \max\{t, -nt / \log \beta  \} < r <\infty $, where  $\beta = 1-  (1- \alpha_0)^{p/A} /  [\omega]_{A_{p/A}} $ with $0< A < p / q_\omega$, 	or let $0<p\le t<r=\infty$.	
		If $f \in  \dot{\mathcal B}^{s,q,L}_{p,t,r,\omega} (X)  $, then there exists a sequence of $ (L, M, N, p ,\omega ) $  molecules $ \{  a_Q \} _ { Q \in \mathcal D _\nu, \nu \in \mathbb Z  } $ and a sequence of coefficients $ \{  s_Q \} _{Q\in \mathcal D _\nu , \nu \in \mathbb Z } $ such that 
		\begin{equation*}
			f =  \sum_{\nu \in \mathbb Z} \sum_{Q \in \mathcal D _\nu} s_Q a_Q \; \; \mathrm{in} \; \mathcal S _\infty '.
		\end{equation*}
		Moreover, 
		\begin{align*}
			&	 \left( 	\sum_{Q \in \mathcal D }\omega (Q) ^{r/t - r/p} \left( \sum_{\nu = j_Q}^\infty  \left(   \sum_{P \in \D_{\nu} , P \subset Q} 2^{\nu s  p }|s_P|^p  \right)^{q/p}
			\right) ^{r/q}  \right)^{1/r} 	\lesssim
			\| f \| _{    \dot{\mathcal B}^{s,q,L}_{p,t,r,\omega} (X)   }. 
		\end{align*}

	{\rm (b)} 
	Let  $0<p<t<\infty$  and $ \max\{t, -nt / \log \beta  \} < r <\infty $, where  $\beta = 1-  (1- \alpha_0)^{p/A} /  [\omega]_{A_{p/A}} $ with $0< A < \min \{ q , p/q_\omega  \}$, 	or let $0<p\le t<r=\infty$.
	If $f \in  \dot{\mathcal F}^{s,q,L}_{p,t,r,\omega} (X)  $, then there exist a sequence of $(L, M, N, p ,\omega ) $ molecules  $ \{  a_Q \} _ { Q \in \mathcal D _\nu, \nu \in \mathbb Z  } $ and a sequence of coefficients $ \{  s_Q \} _{Q\in \mathcal D _\nu , \nu \in \mathbb Z } $ such that 
	\begin{equation*}
		f =  \sum_{\nu \in \mathbb Z} \sum_{Q \in \mathcal D _\nu} s_Q a_Q \; \; \mathrm{in} \; \mathcal S _\infty '.
	\end{equation*}
	Moreover, 
	\begin{align*}
		\nonumber
		&	\bigg(
		\sum_{Q \in \mathcal D }\omega (Q) ^{r/t - r/p} 
		\bigg(      \int_Q  \bigg( \sum_{\nu = j_Q}^\infty 
		2^{\nu sq}   \sum_{ P \in \D _{\nu}, P \subset Q }  
		\big(   | \omega( P )^{-1/p}  s_ {P} | \chi_P (x) \big)^q      \bigg)^{p/q}  \omega(x) \d \mu (x)
		\bigg)^{r/p}
		\bigg)^{1/r}  \\
		& \lesssim  \| f \| _{    \dot{\mathcal F}^{s,q,L}_{p,t,r,\omega} (X)   }.
	\end{align*}
	\end{thm}

To get the converse direction, we need some lemmas.

\begin{lem}[Lemma 4.1, \cite{BX25}] \label{psi m_Q}
	Let $\psi $  be a partition of unity and let $m_Q$ be an $( L,M,N, p,\omega )$ molecule with some $Q \in \mathcal D _\nu$. Then if $ N -n >  W >0$, we have that
	for $\alpha > 2^{-\nu }$,
	\begin{equation}
		\nonumber
		\big| \psi \big( \alpha \sqrt{L}\big) m_Q (x) \big| \lesssim
		\bigg( \frac{2^{-\nu }}{\alpha } \bigg)^{2M} \omega(Q) ^{-1/p}  \bigg(  1+ \frac{  \rho(x,x_Q)}{\alpha}  \bigg) ^{-W} ;
	\end{equation}
	and for $\alpha \le 2^{-\nu }$,
	\begin{equation}
		\nonumber
		\big| \psi \big( \alpha \sqrt{L}\big) m_Q \big| \lesssim \bigg(   \frac{\alpha }{ 2^{-\nu } } \bigg)^{2M-n} \omega (Q) ^{-1/p}   \bigg(  1+ \frac{  \rho(x,x_Q)}{ 2^{-\nu} }  \bigg) ^{-W}   .
	\end{equation}
	
\end{lem}

\begin{lem} [Lemma 4.4, \cite{BBD20}]   \label{f_Q}
	Let $\omega \in A_{\tilde q}$, $W >n$, $0<A \le 1$, $\kappa \in [0, \tilde q /A]$ and $\eta, \nu \in \mathbb Z$, $\nu \ge \eta$. Assume that $ \{ f_Q\} _{Q \in \mathcal D _\nu } $ is  a sequence of functions satisfying
	\begin{equation*}
		| f_Q (x) | \lesssim \bigg(   \frac{V(Q)}{ V(x_Q , 2^ {-\eta}) }  \bigg)^\kappa  \bigg(  1 + \frac{\rho(x,x_Q)}{ 2^{-\eta} } \bigg) ^{-W}.
	\end{equation*}
	Then for  $ n \tilde q / W < A \le 1 $  and a sequence of numbers $\{ s_Q\}_{Q\in \mathcal D}$, we have
	\begin{equation*}
		\sum_{Q\in \mathcal D _\nu} |s_Q| |f_Q (x)| \lesssim 2^{n (\nu -\eta) (\tilde q  /r -\kappa ) }  \mathcal M _{\omega, A}  \bigg(   	\sum_{Q\in \mathcal D _\nu} |s_Q| \chi _Q  \bigg)  (x).
	\end{equation*}
\end{lem}

\begin{rem}
	In  \cite[Lemma 4.4]{BBD20},  $\kappa \in [0,1]$. However, repeating the proof in \cite[Lemma 4.4]{BBD20},  we find that if  the condition of $\kappa$ is  changed to $\kappa \in [0,\tilde q /A]$, the result still hold true. Note that since $\tilde q \ge 1 $ and $A\le 1$,  we have $\tilde q /A \ge 1$ and $[0,1] \subset [0,\tilde q /A] $.
\end{rem}

Then we are ready to get the inverse direction of molecular decompositions, that is, each molecular decomposition with suitable coefficients belongs to spaces  $ \dot{\mathcal B}^{s,q,L}_{p,t,r,\omega} (X) ,  \dot{\mathcal F}^{s,q,L}_{p,t,r,\omega} (X) $.

\begin{thm}
		Let  $\omega \in A_\infty$,  $0<q\le \infty$, $s\in \mathbb R$, $M \in\mathbb N _+$. 
	
	{\rm (a)}	
	Let  $0<p<t<\infty$  and $ \max\{t, -nt / \log \beta  \} < r <\infty $, where  $\beta = 1-  (1- \alpha_0)^{p/B} /  [\omega]_{A_{p/B}} $ with $0< B < p / q_\omega$, 	or let $0<p\le t<r=\infty$.
		Let $ M,N$ be a sufficiently large number such that
	\begin{align*}
		&	2M >  \max \left( n + \frac{n q_\omega}{\min(1,p / q_\omega ,q)} -s , s  \right)  + (1/p -1/t) n \tilde q  \\
		& N > n +  \frac{n q_\omega}{\min(1,/ q_\omega,q)} .
	\end{align*}
	If
\begin{equation*}
	f =  \sum_{\nu \in \mathbb Z} \sum_{Q \in \mathcal D _\nu} s_Q m_Q \; \;  \mathrm{in} \; \mathcal S _\infty ' ,
\end{equation*} 	
where $\{m_Q\} _ { Q \in \mathcal D _\nu ,\nu \in \mathbb Z }$ is a sequence of $  (L, M, N, p , \omega) $ molecules and $  \{  s_Q \} _ { Q \in \mathcal D _\nu ,\nu \in \mathbb Z }$ is a sequence of coefficients satisfying
\begin{equation*}
	\Bigg\{  \sum_{\nu\in \mathbb Z}  
	\bigg\|  \big(      \sum_{P \in \mathcal D _\nu} 2^{\nu s} |s_P| \omega (P) ^{-1/p} \chi _P \big) \bigg\|_{ M_{p,\omega}^{t,r} } ^q \Bigg\}^{1/q}
	<\infty,
\end{equation*}
then $ f \in  \dot{\mathcal B}^{s,q,L}_{p,t,r,\omega} (X)   $ and
\begin{equation*} 
	\| f \| _{ \dot{\mathcal B}^{s,q,L}_{p,t,r,\omega} (X)   }  \lesssim 		\Bigg\{  \sum_{\nu\in \mathbb Z}  
	\bigg\|  \big(      \sum_{P \in \mathcal D _\nu} 2^{\nu s} |s_P| \omega (P) ^{-1/p} \chi _P \big) \bigg\|_{ M_{p,\omega}^{t,r} } ^q \Bigg\}^{1/q} .
\end{equation*}
	
		{\rm (b)}	
	Let  $0<p<t<\infty$  and $ \max\{t, -nt / \log \beta  \} < r <\infty $, where  $\beta = 1-  (1- \alpha_0)^{p/B} /  [\omega]_{A_{p/B}} $ with $0< B < \min \{ q , p/q_\omega  \} $, 	or let $0<p\le t<r=\infty$.
	Let $ M,N$ be a sufficiently large number such that
	\begin{align*}
		&	2M >  \max \left( n + \frac{n q_\omega}{\min(1,p / q_\omega ,q)} -s , s  \right)  + (1/p -1/t) n \tilde q  \\
		& N > n +  \frac{n q_\omega}{\min(1,/ q_\omega,q)} .
	\end{align*}
	If
\begin{equation*}
	f =  \sum_{\nu \in \mathbb Z} \sum_{Q \in \mathcal D _\nu} s_Q m_Q \; \;  \mathrm{in} \; \mathcal S _\infty ' ,
\end{equation*} 	
where $\{m_Q\} _ { Q \in \mathcal D _\nu ,\nu \in \mathbb Z }$ is a sequence of $  (L, M, N, p , \omega) $ molecules and $  \{  s_Q \} _ { Q \in \mathcal D _\nu ,\nu \in \mathbb Z }$ is a sequence of coefficients satisfying
\begin{equation*}
	\bigg(    \sum_{Q \in \mathcal D}  \omega (Q)^{r/t-r/p} \Big(  \int_Q  \Big(  \sum_{\nu =-\infty }^\infty \big(      \sum_{P \in \mathcal D _\nu} 2^{\nu s} |s_P| \omega (P) ^{-1/p} \chi _P  \big)^q \Big)^{p/q}  \omega (x) \mathrm {d} \mu (x)      \Big)^{r/p}                 \bigg)^{1/r}  <\infty,
\end{equation*}
then $ f \in  \dot{\mathcal F}^{s,q,L}_{p,t,r,\omega} (X)   $ and
\begin{equation*} 
	\| f \| _{ \dot{\mathcal F}^{s,q,L}_{p,t,r,\omega} (X)   }  \lesssim \bigg(    \sum_{Q \in \mathcal D}  \omega (Q)^{r/t-r/p} \Big(  \int_Q  \Big(  \sum_{\nu =-\infty }^\infty \big(      \sum_{P \in \mathcal D _\nu} 2^{\nu s} |s_P| \omega (P) ^{-1/p} \chi _P  \big)^q \Big)^{p/q}  \omega (x) \mathrm {d} \mu (x)      \Big)^{r/p}                 \bigg)^{1/r}  .
\end{equation*}

\end{thm}

\begin{proof}
	The proof can be done by using similar arguments to those in \cite{BX25,BBD20, FJ85, FJ90}. However, for the   convenience of the reader, we  provide the details.
	
	Fix $\tilde q \in (q_\omega ,\infty)$ and $A  < \min(1,p /q_\omega ,q)$ such that $\omega  \in A_{\tilde q}$ . Now fix $W >n \tilde q /A $. Let $\psi $ be a partition of unity. Since
	\begin{equation*}
		f =  \sum_{\nu \in \mathbb Z} \sum_{Q \in \mathcal D _\nu} s_Q m_Q \; \;  \mathrm{in} \; \mathcal S _\infty ',
	\end{equation*} 	
	then, for each $j\in \mathbb Z$, we obtain
	\begin{align*}
		\psi_j (\sqrt{L})  f &=  \sum_{\nu \in \mathbb Z} \sum_{Q \in \mathcal D _\nu} s_Q 	\psi_j (\sqrt{L})   m_Q \\
		&=  \sum_{\nu: \nu \ge j} \sum_{Q \in \mathcal D _\nu} s_Q 	\psi_j (\sqrt{L})   m_Q 
		+  \sum_{\nu: \nu < j} \sum_{Q \in \mathcal D _\nu} s_Q 	\psi_j (\sqrt{L})   m_Q.
	\end{align*}
	Using Lemma \ref{psi m_Q}  and Lemma  \ref{f_Q} (with $\kappa =0 $), we obtain
	\begin{align} \label{nu  ge j}
		\nonumber
		& \sum_{\nu: \nu \ge j} \sum_{Q \in \mathcal D _\nu} s_Q 	\psi_j (\sqrt{L})   m_Q  \\
		\nonumber
		& \lesssim \sum_{\nu: \nu \ge j} \sum_{Q \in \mathcal D _\nu}  |s_Q|   \bigg(   \frac{ 2^{-j} }{ 2^{-\nu } } \bigg)^{2M-n} \omega (Q) ^{-1/p}   \bigg(  1+ \frac{  \rho(x,x_Q)}{ 2^{-j}}  \bigg) ^{-W}  \\
		& \lesssim \sum_{\nu: \nu \ge j} 
		2^ { -(\nu -j) (2M-n - n \tilde q / A) } \mathcal M _ {\omega, A}  \bigg(  	\sum_{Q \in \mathcal D _\nu}  |s_Q|   \omega (Q) ^{-1/p}  \chi_Q \bigg).
	\end{align}
	In (\ref{nu  ge j}), it is required that   $N-n > W >0$.
	Here, $ n\tilde q / W < A <\min(1, q, p/q_\omega)  \le 1 $.

	And for the another sum, we have
	\begin{align} \label{nu  < j}
		\nonumber
		&  \sum_{\nu: \nu < j} \sum_{Q \in \mathcal D _\nu} s_Q 	\psi_j (\sqrt{L})   m_Q \\
		\nonumber
		& \lesssim  \sum_{\nu: \nu < j} \sum_{Q \in \mathcal D _\nu} |s_Q| 	\bigg( \frac{2^{-\nu }}{ 2^{-j} } \bigg)^{2M} \omega(Q) ^{-1/p}  \bigg(  1+ \frac{  \rho(x,x_Q)}{ 2^{-j} }  \bigg) ^{-W} \\
		& \lesssim  \sum_{\nu: \nu < j} 
		2 ^{ - (j-\nu ) 2M  } \mathcal M _{\omega , A} \bigg(   \sum_{Q \in \mathcal D _\nu}  |s_Q|   \omega (Q) ^{-1/p}  \chi_Q \bigg).
	\end{align}
	In (\ref{nu  < j}), again we use the condition that    $N-n > W >0$. 
	Follow  (\ref{nu  ge j}) and (\ref{nu  < j}), and we get 
	\begin{align}
		\nonumber
		|2^{js} 	\psi_j (\sqrt{L})  f | &  \lesssim   \sum_{\nu: \nu \ge j} 
		2^ { -(\nu -j) (2M-n - n \tilde q / A +s ) } \mathcal M _ {\omega, A}  \bigg(  	\sum_{Q \in \mathcal D _\nu} 2^{\nu s}  |s_Q|   \omega (Q) ^{-1/p}  \chi_Q \bigg)  \\
		\nonumber
		& +  \sum_{\nu: \nu < j} 
		2 ^{ - (j-\nu ) ( 2M  -s  ) } \mathcal M _{\omega , A} \bigg(   \sum_{Q \in \mathcal D _\nu} 2^{\nu s}  |s_Q|   \omega (Q) ^{-1/p}  \chi_Q \bigg).
	\end{align}
	Let $ 2M-n - n \tilde q / A +s > (1/p -1/t) n \tilde q $ and  $2M  -s > (1/p -1/t) n \tilde q  $. Note that $W > n\tilde{q} / A $, hence we let $W = n\tilde q/ A +\epsilon$ for any sufficiently small $\epsilon >0$.    Note that 
		\begin{align*}
		& \begin{cases}
			&  N-n >W, \\
			&  2M-n - n \tilde q / A +s > (1/p -1/t) n \tilde q, \\
			& 2M  -s > (1/p -1/t) n \tilde q, 
		\end{cases}	
	\end{align*}
	

	Then repeating the proofs of Theorems \ref{atm B 2}, \ref{atm F 2}, we get the results and we omit the details here.
\end{proof}

\subsection{Embeddings}
For brevity, we define the spaces $\dot{\mathcal A}^{s,q,L}_{p,t,r,\omega} (X)$ to be any spaces $\dot{\mathcal A}^{s,q,\psi,L}_{p,t,r,\omega} (X), \mathcal A \in\{\mathcal B, \mathcal F\}$ with any partitions of unity $\psi$.
Recall that the symbol $\hookrightarrow$ stands for continuous embedding. To begin with, we list so-called basic embeddings.

\begin{pro} \label{basic embed}
	Let $ 0< q\le \infty$, $s\in \mathbb R$ and $\omega \in A_\infty $. Let $0<p<t<r<\infty$ 	or $0<p\le t<r=\infty$.
	
	{\rm (i)} The space $\dot{\mathcal A}^{s,q,L}_{p,t,r,\omega} (X), \mathcal A \in\{ \mathcal B, \mathcal F\}$ is monotone with $q$, namely, if $q_1 \le q_2$, then 
	\begin{equation*}
		\dot{\mathcal A}^{s,q_1,L}_{p,t,r,\omega} (X) \hookrightarrow \dot{ \mathcal A}^{s,q_2,L}_{p,t,r,\omega} (X).
	\end{equation*}
	
	{\rm (ii)} The space $\dot{\mathcal A}^{s,q,L}_{p,t,r,\omega} (X), \mathcal A \in\{ \mathcal B, \mathcal F\}$ is monotone with $r$, namely, if $r_1 \le r_2$, then 
	\begin{equation*}
		\dot{\mathcal A}^{s,q,L}_{p,t,r_1,\omega} (X) \hookrightarrow \dot{ \mathcal A}^{s,q,L}_{p,t,r_2,\omega} (X).
	\end{equation*}

	{\rm (iii)}
	\begin{equation} \label{embedding 3}
		\dot{\mathcal B}^{s,\min\{p,q\},L}_{p,t,r,\omega} (X) \hookrightarrow 	\dot{\mathcal F}^{s,q,L}_{p,t,r,\omega} (X)  \hookrightarrow
		\dot{\mathcal B}^{s,\max\{p,q\},L}_{p,t,r,\omega} (X).
	\end{equation}

	{\rm (iv)} If $ 0< p_1 \le p_2 < t < r  \le \infty  $, then for $\mathcal A\in \{ \mathcal B, \mathcal F\}$, we have
	\begin{equation*}
		\dot{\mathcal A}^{s,q,L}_{p_2 ,t,r,\omega} (X)  \hookrightarrow  \dot{\mathcal A}^{s,q,L}_{p_1,t,r,\omega} (X)  . 
	\end{equation*}

\end{pro}

\begin{proof}
	The properties  (i) and (ii) are  coming from the monotonicity of the $\ell^q$-norm on $q$.
	
	(iii)
	It follows from the (generalized)
	Minkowski’s inequality. Indeed, let $f_j := | 2^{js} \psi_j (\sqrt{L}) f|$.
	If $ 0 <q \le p <\infty$, then (\ref{embedding 3}) follows from
	\begin{align*}
		\| f \|_ {  \dot{\mathcal B}^{s, p ,L}_{p,t,r,\omega} (X)  }&  =  \left(  \sum_{Q \in \mathcal D} \omega (Q) ^{r/t - r/p}  \left( \sum_{j = j_Q}^\infty   \int_Q           f_j^p  \omega (x) \d \mu (x)   \right)^{r/p} \right) ^{1/r} \\
		&  =  \left(  \sum_{Q \in \mathcal D} \omega (Q) ^{r/t - r/p}  \left(  \int_Q     \sum_{j = j_Q}^\infty        f_j^p \omega (x) \d \mu (x)   \right)^{r/p} \right) ^{1/r} \\
		& \le  \left(  \sum_{Q \in \mathcal D} \omega (Q) ^{r/t - r/p}  \left(  \int_Q   \left(  \sum_{j = j_Q}^\infty           f_j^q \right)^{p/q}\omega (x) \d \mu (x)   \right)^{r/p} \right) ^{1/r} \\
		& \le  \left(  \sum_{Q \in \mathcal D} \omega (Q) ^{r/t - r/p}  \left(  \sum_{j = j_Q}^\infty  \left( \int_Q           f_j^p \omega (x) \d \mu (x) \right)^{q/p}  \right)^{r/q} \right) ^{1/r}  = 	\| f \|_ {  \dot{\mathcal B}^{s, q ,L}_{p,t,r,\omega} (X)  }.
	\end{align*}
	If $ 0 < p < q <\infty$, then (\ref{embedding 3}) follows from
	\begin{align*}
		\| f \|_ {  \dot{\mathcal B}^{s, q ,L}_{p,t,r,\omega} (X)  }&  =   \left(  \sum_{Q \in \mathcal D} \omega (Q) ^{r/t - r/p}  \left(  \sum_{j = j_Q}^\infty  \left( \int_Q           f_j^p \omega (x)  \d \mu (x) \right)^{q/p}  \right)^{r/q} \right) ^{1/r} \\
		& \le  \left(  \sum_{Q \in \mathcal D} \omega (Q) ^{r/t - r/p}  \left(  \int_Q   \left(  \sum_{j = j_Q}^\infty           f_j^q \right)^{p/q} \omega (x) \d \mu (x)   \right)^{r/p} \right) ^{1/r} \\
		& \le  \left(  \sum_{Q \in \mathcal D} \omega (Q) ^{r/t - r/p}  \left(  \int_Q   \left(  \sum_{j = j_Q}^\infty           f_j^p \right)^{p/p} \omega (x) \d \mu (x)   \right)^{r/p} \right) ^{1/r}  = 	\| f \|_ {  \dot{\mathcal B}^{s, p ,L}_{p,t,r,\omega} (X)  }.
	\end{align*}
	
	(iv) It is obtained by  H\"older's inequality. 
\end{proof}

Use the atomic decompositions for weighted Bourgain-Morrey-Beov type spaces, we obtain the following embeddings,
 which will be used to prove  that  spaces  $\dot{\mathcal B}^{s,q,L}_{p_1,t_1,r,\omega} (X)$   are continuously embedded into $\mathcal S _\infty '$.

\begin{pro} \label{Sobolev  type embedding}
		Let  $\omega \in A_\infty$,  $0<q\le \infty$, $s\in \mathbb R$, $M \in\mathbb N _+$. 
		
	Let  $0<p_0 <t_0 <\infty$  and $ \max\{t_0, -nt_0 / \log \beta  \} < r <\infty $, where  $\beta = 1-  (1- \alpha_0)^{p_0/A} /  [\omega]_{A_{p_0/A}} $ with $0< A < p_0 / q_\omega$, 	or let $0<p_0 \le t_0 <r=\infty$.	
	
		Let  $0<p_1 <t_1 <\infty$  and $ \max\{t_1, -nt_1 / \log \beta  \} < r <\infty $, where  $\beta = 1-  (1- \alpha_0)^{p_1 /A} /  [\omega]_{A_{p_1/A}} $ with $0< A < p_1/ q_\omega$, 	or let $0<p_1 \le t_1 <r=\infty$.	
	
If $ 1/t_0 -1/p_0 = 1/t_1 -1/p_1$ and $p_0 >p_1 $, then 
 $  \dot{\mathcal B}^{s,q,L}_{p_1,t_1,r,\omega} (X)  \hookrightarrow \dot{\mathcal B}^{s,q,L}_{p_0,t_0,r,\omega} (X) $.

\end{pro}
\begin{proof}
	We use the ideas from \cite[Proposition 2.5]{YSY10}.
	By atomic decompositions for weighted Bourgain-Morrey-Beov type spaces, 
	we have
	\begin{align*}
		\| f\|_{   \dot{\mathcal B}^{s,q,L}_{p_0,t_0,r,\omega} (X)} & \approx  \left( 	\sum_{Q \in \mathcal D }\omega (Q) ^{r/t_0 - r/p_0} \left( \sum_{\nu = j_Q}^\infty  \left(   \sum_{P \in \D_{\nu} , P \subset Q} 2^{\nu s  p_0 }|s_P|^{p_0}  \right)^{q/p_0}
		\right) ^{r/q}  \right)^{1/r}  \\
		& = \left( 	\sum_{Q \in \mathcal D }\omega (Q) ^{r/t_0 - r/p_0} \left( \sum_{\nu = j_Q}^\infty 2^{\nu s q} \left(   \sum_{P \in \D_{\nu} , P \subset Q} |s_P|^{p_0}  \right)^{q/p_0}
		\right) ^{r/q}  \right)^{1/r}  \\
		& \le   \left( 	\sum_{Q \in \mathcal D }\omega (Q) ^{r/t_0 - r/p_0} \left( \sum_{\nu = j_Q}^\infty  \left(   \sum_{P \in \D_{\nu} , P \subset Q} 2^{\nu s  p_1 }|s_P|^{p_1}  \right)^{q/p_1}
		\right) ^{r/q}  \right)^{1/r}  \\
		& =   \left( 	\sum_{Q \in \mathcal D }\omega (Q) ^{r/t_1 - r/p_1} \left( \sum_{\nu = j_Q}^\infty  \left(   \sum_{P \in \D_{\nu} , P \subset Q} 2^{\nu s  p_1 }|s_P|^{p_1}  \right)^{q/p_1}
		\right) ^{r/q}  \right)^{1/r}  \\
		& \approx \| f\|_{   \dot{\mathcal B}^{s,q,L}_{p_1,t_1,r,\omega} (X)} .
	\end{align*}
	Hence $  \dot{\mathcal B}^{s,q,L}_{p_1,t_1,r,\omega} (X)  \hookrightarrow \dot{\mathcal B}^{s,q,L}_{p_0,t_0,r,\omega} (X) $.
\end{proof}

\begin{lem}[p. 29, \cite{BBD20}] \label{Lp embedding complete}
	{\rm  (i) }	Let $0< p,q \le \infty$, $s\in \mathbb R$ and $\omega \in A_\infty $.
	The spaces $\dot{\mathcal B}^{s,q,L}_{p,p,\infty,\omega} (X)$ are complete, and  continuously embedded into $\mathcal S _\infty '$.
	
	{\rm  (ii) }	Let  $0<p<\infty$, $0< q \le \infty$, $s\in \mathbb R$ and $\omega \in A_\infty $.
	The spaces $\dot{\mathcal F}^{s,q,L}_{p,p,\infty,\omega} (X)$ are complete, and  continuously embedded into $\mathcal S _\infty '$.
\end{lem}

\begin{lem}[Proposition 2.10, \cite{BBD20}] \label{conv in S infty}
	Let $\psi $ be a partition of unity. Then for any $f\in \mathcal S_\infty '$, we have
	\[
	f = \sum_{j \in \mathbb Z} \psi_j \big(\sqrt{L}\big) (f) \;   \mathrm{in} \; \mathcal S_\infty '.
	\]
\end{lem}

\begin{lem}[Lemma 2.6, \cite{BBD20}]  \label{kernel est}
	Let $\varphi \in \mathscr S (\mathbb R)$ be an even function. Then for any $W >0$, there exists $C>0$ such that
	\[
	\Big| K _ { \varphi ( \alpha \sqrt{L}) }  (x,y) \Big| \le \frac{C}{V(x  \vee y , \alpha  )}\bigg(  1+ \frac{\rho(x,y) }{\alpha} \bigg)^{-W}
	\]
	for all $\alpha >0$ and $x,y \in X$.
\end{lem}


Now we are ready to show spaces are continuously embedded into $\mathcal S _\infty '$ and complete.
\begin{thm} \label{embedding S infty '}
	Let  $\omega \in A_\infty$, $0<p<\infty$, $0<q\le \infty$ and  $s\in \mathbb R$.
	
	{\rm (a)} 	If $0<p<t<\infty$ and $ \max\{t, -nt / \log \beta  \} < r <\infty $, where $\beta = 1-  (1- \alpha_0)^{p/A} /  [\omega]_{A_{p/A}} $ with $0< A<  p / q_\omega$, or $ 0 <p \le t < r =\infty$, 	then the spaces $\dot{\mathcal B}^{s,q,L}_{p,t,r,\omega} (X)$ are complete  and  continuously embedded into $\mathcal S _\infty '$. 
	
	{\rm (b)} 
	If $0<p<t<\infty$ and $ \max\{t, -nt / \log \beta  \} < r <\infty $, where $\beta = 1-  (1- \alpha_0)^{p/A} /  [\omega]_{A_{p/A}} $ with $0< A<  \min \{ q , p/q_\omega  \}$, or $ 0 <p \le t < r =\infty$, then the spaces $\dot{\mathcal F}^{s,q,L}_{p,t,r,\omega} (X)$ are complete  and  continuously embedded into $\mathcal S _\infty '$. 
\end{thm}

\begin{proof}
	The  proof is similar	as \cite[Theorem  3.7]{BX25}. The idea also comes from \cite[Theorem 5.3]{GKKP17}.  
	For the reader's convenience, we give some details.
	We first prove that $\dot{\mathcal F}^{s,q,L}_{p,t,r,\omega} (X)$ and $\dot{\mathcal B}^{s,q,L}_{p,t,r,\omega} (X)$  are  continuously embedded into $\mathcal S _\infty '$. By  Proposition \ref{basic embed}, it suffices to show that  $\dot{\mathcal B}^{s,\infty,L}_{p,t,r,\omega} (X) \hookrightarrow \mathcal S _\infty '$. By Lemma \ref{Lp embedding complete}, it remains to prove the cases $0< p< t< r =\infty$ and $ 0< p<t<\infty$, $ \max\{t, -nt / \log \beta  \} < r <\infty  $. 
	
	We first consider the case $0< p< t< r =\infty$. Let $f \in \dot{\mathcal B}^{s,\infty,L}_{p,t,\infty,\omega} (X)$  and $\phi \in \mathcal S _\infty$. We will show that there exists an $m \in \mathbb N$ such that
	\begin{equation} \label{eq wanted est}
		| \langle f , \phi \rangle | \lesssim \| f \| _{ \dot{\mathcal B}^{s,\infty,L}_{p,t,\infty,\omega} (X)   } \mathcal P _{  m, m,m} ^ *(\phi).
	\end{equation}
	First recall that since $\phi \in \mathcal S _\infty$, for each $k\in \mathbb N$, there exists $g_k \in \mathcal S$ such that $\phi = L^k g_k$. For $m>0$ and $\ell , k \in \mathbb N$, we have
	\begin{equation} \label{semi norm S infty}
		\mathcal P _{m, \ell ,k} ^* (\phi ) = \sup_{x\in X} (1 + \rho (x, x_0)) ^m  \big|  L^\ell g_k (x)    \big|.
	\end{equation}
	Let $\psi \in C_0^\infty (\mathbb R _+) $ be a real-valued function  such that supp $\psi  \subset [1/2,2]$ and $ \sum_{j\in \mathbb Z} \psi^2 ( 2^{-j} \lambda ) =1 $ for all $ \lambda \in \mathbb R _+$. Set $\psi_j (\lambda) = \psi(2^ {-j} \lambda), j\in \mathbb Z$. Then $\sum_{j \in \mathbb Z}  \psi_{j}^2 (\lambda) =1 $ for $\lambda \in \mathbb R_+$. Hence, by Lemma \ref{conv in S infty}, we have
	\begin{equation} \label{eq f = sum psi 2}
		f = \sum_{j\in \mathbb Z} \psi_j ^2 \big(\sqrt{L}\big) f  \quad \operatorname{in} \;  \mathcal S _\infty '.
	\end{equation}
	Note that $\{\psi _j \}_{j\in \mathbb Z}$  can be used to define an equivalent norm on $   \dot{\mathcal B}^{s, \infty ,L}_{p,t,\infty,\omega} (X) $. From (\ref{eq f = sum psi 2}), we obtain
	\begin{equation} \label{inner product}
		\langle f , \phi \rangle   = \sum_{j\in \mathbb Z} 	\big\langle  \psi_j^2 \big(\sqrt{L}\big)  f , \phi \big\rangle =  \sum_{j\in \mathbb Z} 	\big\langle  \psi_j  \big(\sqrt{L}\big) f ,  \psi_j \big(\sqrt{L}\big) \phi \big\rangle .
	\end{equation}
	We next estimate  $\big|  	\big\langle  \psi_j  \big(\sqrt{L}\big) f ,  \psi_j \big(\sqrt{L}\big) \phi \big\rangle   \big|   $ for each $j \in \mathbb Z$. We consider two subcases. 
	
	Subcase: $j \ge 0$. Let $m>   \tilde n + 3n/2+ n q_\omega /t  + |s| $. We first estimate $   \big|   \psi_j \big(\sqrt{L}\big) \phi (x) \big|   $. Set $ \Psi (\lambda) := \lambda ^{-2m} \psi (\lambda)$. Then $\psi_j \big(\sqrt{L}\big)  = 2^{-2mj}  \Psi \big(2^{-j} \sqrt{L} \big) L ^m$ and hence
	\[
	\psi_j \big(\sqrt{L}\big) \phi (x)  = 2^{-2mj} \int_X  K_{   \Psi (2^{-j} \sqrt{L} ) } (x,y) L^m \phi (y) \d \mu  (y).
	\]
	Clearly, $\Psi \in C_0^\infty (\mathbb R _+)$ and  supp $\Psi \subset [1/2,2]$. Hence, by Lemma \ref{kernel est}, we have
	\[
	K_{   \Psi (2^{-j} \sqrt{L} ) } (x,y)	 \le \frac{C}{V(x  \vee y , 2^{-j} )}\bigg(  1+ \frac{\rho(x,y) }{2^{-j}} \bigg)^{-m}.
	\]
	On the other hand, since $\phi \in \mathcal S _\infty$, by (\ref{semi norm S infty}), we have for each $k\in \mathbb N$,
	\[
	| L^m \phi (y) | \lesssim (1 + \rho (x_0,y ) ) ^{-m} \mathcal P _{m,m,k}^* (\phi).
	\]
	Putting the above estimate together, we obtain
	\[
	\big| 	\psi_j \big(\sqrt{L}\big) \phi (x) \big| \lesssim 2^{-2mj}  \mathcal P _{m,m,k}^* (\phi) \int_X \frac{\d \mu  (y) }{  V (y ,2^{-j})    (1 + 2^j  \rho(x,y) )^m  (1+ \rho (y,x_0)) ^m     }.
	\]
	From (\ref{V x r V y r}) and the doubling condition, we have
	\[
	V (x_0,1) \le C (1 +\rho (y,x_0) ) ^{ \tilde n} V (y,1) \le C 2^{jn} (1 +\rho (y,x_0) ) ^{ \tilde n} V (y, 2^{-j}).
	\]
	Hence,
	\begin{align}
		\nonumber
		\big| 	\psi_j \big(\sqrt{L}\big) \phi (x) \big| & \lesssim 2^{- j( 2m -n)}  \mathcal P _{m,m,k}^* (\phi) \int_X \frac{\d \mu  (y) }{  V (x_0 ,1)    (1 +  \rho(x,y) )^m  (1+ \rho (y,x_0)) ^{m-\tilde n}     } \\
		\nonumber
		& \lesssim 2^{- j( 2m -n)}  \mathcal P _{m,m,k}^* (\phi) \int_X \frac{\d \mu  (y) }{  V (x_0 ,1)    (1 +  \rho(x,y) )^{m - \tilde n}  (1+ \rho (y,x_0)) ^{m-\tilde n}     } \\
		\nonumber
		& \lesssim 2^{- j( 2m -n)}  \mathcal P _{m,m,k}^* (\phi)  ( 1+ \rho (x,x_0) ) ^{-m + \tilde n +n}.
	\end{align}
	Here for the last inequality we used  $m > \tilde n +n$.
	Now we are ready to estimate the inner products in (\ref{inner product}).
	We may  consider the case $1< p $ thanks to  Proposition  \ref{Sobolev  type embedding}. Then applying the H\"older inequality, we get

	\begin{align*}
		& \big|	\big\langle  \psi_j  \big(\sqrt{L}\big) f ,  \psi_j \big(\sqrt{L}\big) \phi \big\rangle  \big| \\
		& \le \sum_{ \tau  \in I_j} \int_{ Q^j _\tau   }  \big| \psi_j  \big(\sqrt{L}\big) f (x) \big|  \big|  \psi_j \big(\sqrt{L}\big) \phi  (x) \big| \d \mu (x) \\
		& \le 2^{- j( 2m -n)}  \mathcal P _{m,m,k}^* (\phi)  \sum_{ \tau  \in I_j} \int_{ Q^j _\tau   }  \big| \psi_j  \big(\sqrt{L}\big) f (x) \big| \omega (x) ^{1/p} \omega (x) ^{-1/p}   ( 1+ \rho (x,x_0) ) ^{-m + \tilde n +n}  \d \mu (x) \\
		& \le 2^{- j( 2m -n)}  \mathcal P _{m,m,k}^* (\phi) 2^{-js}  \| f\| _{  \dot{\mathcal B}^{s,\infty,L}_{p,t,r,\omega} (X)  } \\
		& \times \sum_{ \tau  \in I_j}  \omega ( Q^j_\tau ) ^{1/p -1/t}  \bigg(  \int_{ Q^j _\tau   }  \omega (x) ^{-p'/p}   ( 1+ \rho (x,x_0) ) ^{ (-m + \tilde n +n) p'}  \d \mu (x)   \bigg)^{1/p'} .
	\end{align*}
	From \cite[(20)]{BX25}, we have that for $m >   \tilde n + 3n/2+ n q_\omega /t$,
	\begin{equation} \label{sum le 1}
		\sum_{ \tau  \in I_j}  \omega ( Q^j_\tau ) ^{1/p -1/t}  \bigg(  \int_{ Q^j _\tau   }  \omega (x) ^{-p'/p}   ( 1+ \rho (x,x_0) ) ^{ (-m + \tilde n +n) p'}  \d \mu (x)   \bigg)^{1/p'}  \lesssim 1.
	\end{equation}
	Hence, we have
	\[
	\big|	\big\langle  \psi_j  \big(\sqrt{L}\big) f ,  \psi_j \big(\sqrt{L}\big) \phi \big\rangle  \big|  \lesssim 2^{- j( 2m -n)}  \mathcal P _{m,m,k}^* (\phi) 2^{-js}  \| f\| _{  \dot{\mathcal B}^{s,\infty,L}_{p,t,r,\omega} (X)  }, \quad j\ge 0,
	\]
	if we provide that $m >   \tilde n + 3n/2+ n q_\omega /t$. Summing up these estimates we obtain
	\begin{equation} \label{eq conti emded 1}
		\sum _{j\ge 0} 	 \big|	\big\langle  \psi_j  \big(\sqrt{L}\big) f ,  \psi_j \big(\sqrt{L}\big) \phi \big\rangle  \big|  \lesssim  \mathcal P _{m,m,k}^* (\phi)   \| f\| _{   \dot{\mathcal B}^{s,\infty,L}_{p,t,r,\omega} (X)  },
	\end{equation}
	where we used that $  2m > n-s$.
	
	Subcase: $j<0$. Choose $m >  \tilde n + 3n/2+ n q_\omega /t  + |s| $. Set $\Psi (\lambda) := \lambda ^{2m} \psi (\lambda).$  Then $\psi_j \big(\sqrt{L}\big) = 2 ^{2mj} L^{-m} \Psi \big(2 ^{ - j}  \sqrt{L}\big)$  and
	\[
	\psi_j \big(\sqrt{L}\big) \phi (x) =  2 ^{2mj}  \Psi \big( 2^{-j} \sqrt{L}\big) L^{-m} \phi (x) =  2 ^{2mj} \int_X K_{   \Psi ( 2^{-j} \sqrt{L} )  }  (x,y) L^{-m} \phi (y) \d \mu  (y).
	\]
	Clearly, $\Psi \in C_0^\infty (\mathbb R _+)$ and  supp $\Psi \subset [1/2,2]$. Hence, by Lemma \ref{kernel est}, we have
	\[
	K_{   \Psi (2^{-j} \sqrt{L} ) } (x,y)	 \le \frac{C}{V(x  \vee y , 2^{-j} )}\bigg(  1+ \frac{\rho(x,y) }{2^{-j}} \bigg)^{-m}.
	\]
	On the other hand, since $\phi \in \mathcal S _\infty$, by (\ref{semi norm S infty}), we have
	
	\[
	| L^{-m} \phi (y) | \lesssim (1 + \rho (x_0,y ) ) ^{-m} \mathcal P _{m,m,m}^* (\phi).
	\]
	Putting the above estimate together, we obtain
	\[
	\big| 	\psi_j \big(\sqrt{L}\big) \phi (x) \big| \lesssim 2^{2mj}  \mathcal P _{m,m,m}^* (\phi) \int_X \frac{\d \mu  (y) }{  V (y ,2^{-j})    (1 + 2^j  \rho(x,y) )^m  (1+ \rho (y,x_0)) ^m     }.
	\]
	From   $ V(x, 2^{-j}) \le c (1 + 2^j  \rho (x,y)) ^{\tilde n}  V (y, 2^{-j}) $, we have
	
	\begin{align}
		\nonumber
		\big| 	\psi_j \big(\sqrt{L}\big) \phi (x) \big| & \le
		2^{2mj}  \mathcal P _{m,m,m}^* (\phi) \int_X \frac{\d \mu  (y) }{  V (x ,2^{-j})    (1 + 2^j  \rho(x,y) )^ {m - \tilde n}  (1+ \rho (y,x_0)) ^{ m -\tilde n}     }
		\\
		\nonumber
		& \le
		2^{2mj}  \mathcal P _{m,m,m}^* (\phi) ( 1+ 2^j \rho ( x,x_0) )^ {-m +\tilde n +n}
		\\
		\nonumber
		& \le
		2^{j (m +\tilde n +n) }  \mathcal P _{m,m,m}^* (\phi) ( 1+ \rho ( x,x_0) )^ {-m +\tilde n +n}, \quad j <0.
	\end{align}
	Here for the second inequality we used that $m > \tilde n +n$.
	Now we are ready to estimate the inner products in (\ref{inner product}).
	We may  consider the case $1< p $ thanks to  Proposition  \ref{Sobolev  type embedding}. Then applying the H\"older inequality, we get
	\begin{align*}
		& \big|	\big\langle  \psi_j  \big(\sqrt{L}\big) f ,  \psi_j \big(\sqrt{L}\big) \phi \big\rangle  \big| \\
		& \le \sum_{ \tau  \in I_j} \int_{ Q^j _\tau   }  \big| \psi_j  \big(\sqrt{L}\big) f (x) \big|  \big|  \psi_j \big(\sqrt{L}\big) \phi  (x) \big| \d \mu (x) \\
		& \le 	2^{j (m +\tilde n +n) }  \mathcal P _{m,m,m}^* (\phi)   \sum_{ \tau  \in I_j} \int_{ Q^j _\tau   }  \big| \psi_j  \big(\sqrt{L}\big) f (x) \big| \omega (x) ^{1/p} \omega (x) ^{-1/p}   ( 1+ \rho (x,x_0) ) ^{-m + \tilde n +n}  \d \mu (x) \\
		\nonumber
		& \le 	2^{j (m +\tilde n +n) }  \mathcal P _{m,m,m}^* (\phi)  \sum_{ \tau  \in I_j} \int_{ Q^j _\tau   }  \big| \psi_j  \big(\sqrt{L}\big) f (x) \big| \omega (x) ^{1/p} \omega (x) ^{-1/p}   ( 1+ \rho (x,x_0) ) ^{-m + \tilde n +n}  \d \mu (x) \\
		& \le 	2^{j (m +\tilde n +n) }  \mathcal P _{m,m,m}^* (\phi) 2^{-js}  \| f\| _{   \dot{\mathcal B}^{s,\infty,L}_{p,t,r,\omega} (X)  } \\
		&	\times
		\sum_{ \tau  \in I_j}  \omega ( Q^j_\tau ) ^{1/p -1/t}  \bigg(  \int_{ Q^j _\tau   }  \omega (x) ^{-p'/p}   ( 1+ \rho (x,x_0) ) ^{ (-m + \tilde n +n) p'}  \d \mu (x)   \bigg)^{1/p'} .
	\end{align*}
	Note that (\ref{sum le 1}) also holds for $j <0$. Then if $  m > s - \tilde n -n$, we have
	\begin{equation} \label{eq conti emded 2}
		\sum_{j <0}  \big|	\big\langle  \psi_j  \big(\sqrt{L}\big) f ,  \psi_j \big(\sqrt{L}\big) \phi \big\rangle  \big|  \lesssim \mathcal P _{m,m,m}^* (\phi)  \| f\| _{  \dot{\mathcal B}^{s,\infty,L}_{p,t,r,\omega} (X)   } .
	\end{equation}
	By (\ref{eq conti emded 1}) and (\ref{eq conti emded 2}), we have $ \dot{\mathcal B}^{s,\infty,L}_{p,t,r,\omega} (X)   \hookrightarrow \mathcal {S}_\infty '  $.
	
	Case $0<p<t<\infty $, $ \max\{t, -nt / \log \beta  \} < r <\infty  $. By the fact that $  \dot{\mathcal B}^{s,q,L}_{p,t,r,\omega} (X)  \hookrightarrow  \dot{\mathcal B}^{s,q,L}_{p,t,\infty,\omega} (X) $,  we have $ \dot{\mathcal B}^{s,q,L}_{p,t,r,\omega} (X)  \hookrightarrow \mathcal  S _\infty '$.
	
	Finally, we prove the completeness. Let $ \{ f_\ell\}_{\ell =1}^\infty  $ be a fundamental sequence in $  \dot{\mathcal B}^{s,q,L}_{p,t,r,\omega} (X) $.  By the fact that $  \dot{\mathcal B}^{s,q,L}_{p,t,r,\omega} (X) \hookrightarrow \mathcal  S _\infty '$, we obtain $ \{ f_\ell\}_{\ell =1}^\infty  $  is also a fundamental sequence in $\mathcal S _\infty '$. Because $\mathcal  S _\infty '$ is complete, we find a limit element $f \in \mathcal S _\infty '$.  If $ \psi $ be a partition of unity, then $ \psi_j \big(\sqrt{L}\big) f _\ell $ converge to $ \psi_j \big(\sqrt{L}\big) f $ in $  \mathcal S _\infty '$ if $\ell \to \infty $. On the other hand, $\big\{ \psi_j \big(\sqrt{L}\big)  f_\ell \big\}_{ \ell=1} ^\infty$ is a fundamental sequence in $ M_{p,\omega}^{t,r}$. By \cite[Lemma 2.9]{BBD20}, (also can be found in \cite[p. 22]{BBD20}), we have
	\[
	\big|\psi_j \big(\sqrt{L}\big) f_\ell (x)\big| \le \psi_{j,\lambda} ^* \big(\sqrt{L}\big) f_\ell(x) < \infty
	\]
	for all $x\in X$, $\ell \in  \{ 1, 2,3,\ldots\}$. This show  $\big\{ \psi_j \big(\sqrt{L}\big) f_\ell \big\}_{ \ell=1} ^\infty$ is also a  fundamental sequence in $L^\infty$. This show the limiting element of $ \big\{ \psi_j \big(\sqrt{L}\big) f_\ell \big\} _{\ell=1}^\infty  $ in $ M_{p,\omega}^{t,r}$  (which is the same as in $L^\infty$) coincides with $ \psi_j \big(\sqrt{L}\big) f $. Using  Fatou's lemma, we obtain
	\begin{align*}
		\| f \| _{   \dot{\mathcal B}^{s,q,L}_{p,t,r,\omega} (X) } & = 	\bigg(    \sum_{Q \in \mathcal D}  \omega (Q)^{r/t-r/p} \Big(  \sum_{j=j_Q}^\infty \Big(   \int_Q | 2^{js}  \psi_j (   \sqrt{L} )f (x)|^p \omega (x)\d \mu (x)   \Big)^{q/p}   \Big)^{r/q}                 \bigg)^{1/r}\\
		& = 	\bigg(    \sum_{Q \in \mathcal D}  \omega (Q)^{r/t-r/p} \Big(  \sum_{j=j_Q}^\infty \Big(   \int_Q 2^{jsp} \liminf_{\ell\to \infty} |  \psi_j \big(\sqrt{L}\big) f_\ell (x)|^p \omega (x)\d \mu (x)   \Big)^{q/p}   \Big)^{r/q}                 \bigg)^{1/r} \\
		& \le   \liminf_{\ell\to \infty}	\bigg(    \sum_{Q \in \mathcal D}  \omega (Q)^{r/t-r/p} \Big(  \sum_{j=j_Q}^\infty \Big(   \int_Q  | 2^{js} \psi_j \big(\sqrt{L}\big) f_\ell (x)|^p \omega (x)\d \mu (x)   \Big)^{q/p}   \Big)^{r/q}                 \bigg)^{1/r} \\
		& <\infty.
	\end{align*}
	This shows that $f \in \dot{\mathcal B}^{s,q,L}_{p,t,r,\omega} (X)$ and $\{f_\ell \}_{\ell=1}^\infty $ converges to $f$ in  $\dot{\mathcal B}^{s,q,L}_{p,t,r,\omega} (X)$. Hence, $\dot{\mathcal B}^{s,q,L}_{p,t,r,\omega} (X)$ is complete. Similarly, we obtain that  $\dot{\mathcal F}^{s,q,L}_{p,t,r,\omega} (X)$ is also complete.
\end{proof}

\section{Applications}\label{app}

First we recall the notation of fractional powers.
Let $\tau \in \mathbb R$ and define $L^{\tau /2}  : \mathcal S_\infty  \to \mathcal  S_\infty $  by
\begin{equation}\label{tau}
	L^{\tau /2}  f = \frac{1}{\Gamma  ( m -\tau/2  )} \int_0^\infty  \alpha ^{-\tau /2} (\alpha L)^m e^{-\alpha L} f \frac{\d \alpha}{\alpha}
\end{equation}
for any $m \in \mathbb N$, $m > \tau /2$. 
From \cite[Section 5.4]{BBD20},  we know that
\begin{itemize}
	\item[(i)] $ L ^{\tau /2}  $ is well defined as an operator from $ \mathcal S _\infty $ into $ \mathcal S _\infty.  $
	\item[(ii)]  $ L ^{\tau_1} ( L^{\tau_2}f) =L ^{\tau_1+ \tau_2} f$, for all $ f \in \mathcal S _\infty  .$
\end{itemize}

\begin{thm} \label{Frac pow}
		Let  $\omega \in A_\infty$,  $0<q\le \infty$, $s\in \mathbb R$. Let $\tau \in \mathbb R$ and let $ L^{\tau /2} $ be defined as in (\ref{tau}). 
		
			{\rm (a)}	Let  $0<p<t<\infty$  and $ \max\{t, -nt / \log \beta  \} < r <\infty $, where  $\beta = 1-  (1- \alpha_0)^{p/A} /  [\omega]_{A_{p/A}} $ with $0< A < p / q_\omega$, 	or let $0<p\le t<r=\infty$.	
		Then the fractional integral $L^{\tau /2}$ is a bounded operator from $  \dot{\mathcal B}^{s,q,L}_{p,t,r,\omega} (X)  $ into $  \dot{\mathcal B}^{s +\tau ,q,L}_{p,t,r,\omega} (X)  $.
		
		{\rm (b)}
		Let  $0<p<t<\infty$  and $ \max\{t, -nt / \log \beta  \} < r <\infty $, where  $\beta = 1-  (1- \alpha_0)^{p/A} /  [\omega]_{A_{p/A}} $ with $0< A <\min \{ q , p/q_\omega  \} $, 	or let $0<p\le t<r=\infty$.	
	Then the fractional integral $L^{\tau /2}$ is a bounded operator from $  \dot{\mathcal F}^{s,q,L}_{p,t,r,\omega} (X)  $ into $  \dot{\mathcal F}^{s +\tau ,q,L}_{p,t,r,\omega} (X)  $.
\end{thm}

\begin{proof}
	Let $\psi$ be a partition of unity. From \cite[Theorme 7.1]{BBD20}, we have
	\begin{equation*}
		| \psi(\alpha \sqrt{L} ) L^{ \tau /2} ) f (x) | \lesssim \alpha^{-\tau} \psi_\lambda^* ( \alpha  \sqrt{L} )  f(x).
	\end{equation*}
	By Theorem \ref{char conti}, we have
	\begin{align}
		\nonumber
		& 	\|L^{\tau/2}  f \| _{ \dot{\mathcal F}^{s,q,L}_{p,t,r,\omega} (X)   } \\
		\nonumber
			 &  \approx 
			\bigg(    \sum_{Q \in \mathcal D}  \omega (Q)^{r/t-r/p} \Big(  \int_Q  \Big(  \int_{0}^{2^{-j_Q} }| \alpha ^{-s} \psi  (\alpha  \sqrt{L} )L^{\tau/2} f (x)|^q\frac{\d \alpha}{ \alpha} \Big)^{p/q}  \omega (x) \d \mu (x)      \Big)^{r/p}                 \bigg)^{1/r} \\
			\nonumber
			& \lesssim 	\bigg(    \sum_{Q \in \mathcal D}  \omega (Q)^{r/t-r/p} \Big(  \int_Q  \Big(  \int_{0}^{2^{-j_Q} }| \alpha ^{-s -\tau} \psi_\lambda^* ( \alpha  \sqrt{L} )  f(x) |^q\frac{\d \alpha}{ \alpha} \Big)^{p/q}  \omega (x) \d \mu (x)      \Big)^{r/p}                 \bigg)^{1/r} \\
			\nonumber
			& \approx \|  f \| _{ \dot{\mathcal F}^{s+\tau ,q,L}_{p,t,r,\omega} (X)   }.
	\end{align}
	The proof of $  \dot{\mathcal B}^{s,q,L}_{p,t,r,\omega} (X)  $  is similar by using Theorem \ref{char conti} again. This completes our proof.
\end{proof}

Next we consider spectral multiplier of Laplace transform type.
Let $m : [0,\infty ) \to \mathbb C $ be a bounded function. Define the spectral multiplier of Laplace transform type of $L$ by
\begin{equation} \label{spe mul}
	\tilde{m}  (L) = \int_0^\infty \alpha L  e ^{-\alpha ^2 L}  m(\alpha^2) \d\alpha.
\end{equation}

\begin{thm}\label{spectral}
		Let  $\omega \in A_\infty$,  $0<q\le \infty$, $s\in \mathbb R$. Let $\tilde{m}  (L)   $ be the spectral multiplier of Laplace transform type  defined by (\ref{spe mul}).
		
	{\rm (a)}	Let  $0<p<t<\infty$  and $ \max\{t, -nt / \log \beta  \} < r <\infty $, where  $\beta = 1-  (1- \alpha_0)^{p/A} /  [\omega]_{A_{p/A}} $ with $0< A < p / q_\omega$, 	or let $0<p\le t<r=\infty$.	
	Then  $\tilde{m}  (L)   $  is bounded on $ \dot{\mathcal B}^{s,q,L}_{p,t,r,\omega} (X)  $.

		{\rm (b)}
		Let  $0<p<t<\infty$  and $ \max\{t, -nt / \log \beta  \} < r <\infty $, where  $\beta = 1-  (1- \alpha_0)^{p/A} /  [\omega]_{A_{p/A}} $ with $0< A < \min \{ q , p/q_\omega  \}$, 	or let $0<p\le t<r=\infty$.	
	Then $\tilde{m}  (L)   $ is bounded on $ \dot{\mathcal F}^{s,q,L}_{p,t,r,\omega} (X)  $.
\end{thm}

\begin{proof}
	We only prove the case of spaces $\dot{\mathcal F}^{s,q,L}_{p,t,r,\omega} (X)$   since the proof of  $ \dot{\mathcal B}^{s,q,L}_{p,t,r,\omega} (X)  $ is similar.
	Let $\psi $ be a partition of unity. From \cite[Theorem 7.2]{BBD20}, we have 
	\begin{equation*}
		| \psi (\alpha \sqrt{L} )  \tilde{m} (L) f(x) | \lesssim \psi_\lambda^* ( \alpha \sqrt{L}) f(x). 
	\end{equation*}
	Hence, Theorem \ref{char conti} implies the result.
\end{proof}

\vspace{0.4cm}

\medskip 

\noindent{\bf Acknowledgements}

\medskip 

\noindent{\bf Data Availability} Our manuscript has no associated data.

\medskip 
\noindent{\bf\Large Declarations}
\medskip 

\noindent{\bf Conflict of interest} The authors state that there is no conflict of interest.

%
%
%
%
%
%
%
%
%
%
%
%
%
%
%
%
%
%
%
%
%

\end{document}